\documentclass{amsart}
\usepackage{hyperref, bbm, amssymb, mathtools, url}
\usepackage{amsmath,stmaryrd}
\usepackage{cleveref}
\usepackage{fullpage}
\usepackage{tikz}
\usepackage[alphabetic]{amsrefs}
\usetikzlibrary{matrix,arrows,decorations.pathmorphing, cd}
\usepackage{comment}

\newcommand{\Hom}       {\operatorname{Hom}}

\newcommand{\Map}       {\operatorname{Map}}
\newcommand{\Mul}       {\operatorname{Mul}} 

\newcommand{\Fun}       {\operatorname{Fun}}

\newcommand{\op}        {\operatorname{op}}
\newcommand{\laxlim}    {\operatorname*{laxlim}}
\newcommand{\laxcolim}{\operatorname*{laxcolim}} 
\newcommand{\laxlimdag}    {\operatorname*{laxlim^\dagger}}
\newcommand{\laxcolimdag}    {\operatorname*{laxcolim^\dagger}}

\newcommand{\sh}	{\mathrm{sh}}

\newcommand{\Spc}{\mathcal{S}}          
\newcommand{\Spcp}{\mathcal{S}_\ast}    
\newcommand{\Spcgl}{\mathcal{S}_{gl}}   
\newcommand{\I}{\mathcal{I}}            
\newcommand{\Tp}{\mathcal{T}_\ast}      
\newcommand{\Top}{\mathcal{T}}           
\newcommand{\Li}{\CI}             
\newcommand{\OO}{\mathbf{O}}            
\newcommand{\Ar}{\operatorname{Ar}}     
\newcommand{\Arinj}{\Ar_{\mathrm{inj}}} 

\newcommand{\CB}        {{\mathcal{B}}}
\newcommand{\CC}        {{\mathcal{C}}}
\newcommand{\CD}        {{\mathcal{D}}}
\newcommand{\CE}        {{\mathcal{E}}}

\newcommand{\CI}        {{\mathcal{I}}}
\newcommand{\CJ}        {{\mathcal{J}}}
\newcommand{\CK}        {{\mathcal{K}}}
\newcommand{\CM}        {{\mathcal{M}}}
\newcommand{\CP}        {{\mathcal{P}}} 
\newcommand{\CQ}        {{\mathcal{Q}}}
\newcommand{\CT}        {{\mathcal{T}}}

\newcommand{\CV}        {{\mathcal{V}}}
\newcommand{\CW}        {{\mathcal{W}}}

\newcommand{\CO}        {{\mathcal{O}}}

\renewcommand{\CW}      {{\mathcal{W}}}

\newcommand{\CMod}{\mathcal{CM}}       
\newcommand{\tCMod}{\widetilde{\CMod}} 
\newcommand{\Mod}{\operatorname{Mod}}   
\newcommand{\Finp}{\mathrm{Fin}_\ast} 
\newcommand{\Alg}{\operatorname{Alg}}   
\newcommand{\CAlg}{\operatorname{CAlg}} 

\newcommand{\Glo}       {\mathrm{Glo}}
\newcommand{\Orb}       {\mathrm{Orb}}

\newcommand{\Rep}		{\mathrm{Rep}}
\newcommand{\Fin}       {\mathrm{Fin}}
\newcommand{\Sp}        {\mathrm{Sp}}
\newcommand{\PSp}       {\mathrm{PSp}}

\newcommand{\Cat}       {\mathrm{Cat}}
\newcommand{\Op}        {\mathrm{Op}}

\newcommand{\Tw}        {\mathrm{Tw}}

\newcommand{\pr}        {\mathrm{pr}}
\newcommand{\sur}       {\mathrm{sur}}

\newcommand{\h}{\mathrm{h}}

\newcommand{\Fgtglo}[1]		{U^{gl}_{#1}}
\newcommand{\ORfgl}		   {\OR_{fgl}}
\newcommand{\Sfgl}			{S_{fgl}}
\newcommand{\PSpfgl}	  {\PSp_{fgl}}

\newcommand{\Fgtfgl}[2]  {U^{fgl}_{#1,#2}}

\newcommand{\OR}		   {\mathbf{OR}}
\newcommand{\ORf}		  {\ORfgl}  
\newcommand{\ORgl}		  {\OR_{gl}}
\newcommand{\tOR}		  {\OR}
\newcommand{\ttOR}		  {\widetilde{\OR}}
\newcommand{\Slax}		  {S_{gl}}
\newcommand{\PSplax}    {\PSp_{gl}^\dagger}
\newcommand{\Splax}      {\Sp_{gl}^\dagger}

\newcommand{\Llax}        {L_{gl}}

\newcommand{\Spgl}		 {\Sp_{gl}}
\newcommand{\Sgl}		  {S_{gl}}

\newcommand{\colim}  		{\operatornamewithlimits{colim}}
\newcommand{\cocolon}{\nobreak \mskip6mu plus1mu \mathpunct{}\nonscript\mkern-\thinmuskip {:}\mskip2mu \relax}
\newcommand{\Unco}[1]       {\mathrm{Un}^\mathrm{co}(#1)}
\newcommand{\Unct}[1]       {\mathrm{Un}^\mathrm{ct}(#1)}
\newcommand{\DgSpc}[1]		{{{#1}\text{$-$}\Spc}}
\newcommand{\DgSpcp}[1]	   {{{#1}\text{$-$}\Spcp}}
\newcommand{\DgTop}[2]	   {{{#1}\text{$-$}{#2}\Top}}
\newcommand{\DgTp}[2]	    {{{#1}\text{$-$}{#2}\Tp}}

\newtheorem{theorem}{Theorem}[section]
\newtheorem*{thm*}{Theorem}

\newtheorem{lemma}[theorem]{Lemma}
\newtheorem{proposition}[theorem]{Proposition}
\newtheorem{corollary}[theorem]{Corollary}
\theoremstyle{definition}
\newtheorem{remark}[theorem]{Remark}

\newtheorem{definition}[theorem]{Definition}
\newtheorem{example}[theorem]{Example}
\newtheorem{construction}[theorem]{Construction}
\newtheorem{notation}[theorem]{Notation}
\numberwithin{equation}{theorem}
\newtheorem*{remark*}{Remark}

\newtheorem{Step}{Step}
\title{Global homotopy theory via partially lax limits}

\author{Sil Linskens, Denis Nardin and Luca Pol}

\begin{document}
\begin{abstract}
 We provide new $\infty$-categorical models for unstable and stable global homotopy theory. 
 We use the notion of partially lax limits to formalize the 
 idea that a global object is a collection of $G$-objects, one for each compact 
 Lie group $G$, which are compatible with the restriction-inflation functors. 
 More precisely, we show that the $\infty$-category of global spaces is equivalent to 
 a partially lax limit of the functor sending a compact 
 Lie group $G$ to the $\infty$-category of $G$-spaces.
 We also prove the stable version of this result, showing that the 
 $\infty$-category of global spectra is equivalent to the partially lax limit of a diagram 
 of $G$-spectra. Finally, the techniques employed in the previous cases allow us to describe 
 the $\infty$-category of proper $G$-spectra for a Lie group $G$, as a limit of a diagram of 
 $H$-spectra for $H$ running over all compact subgroups of $G$.
\end{abstract}

\maketitle
\setcounter{tocdepth}{1}
\tableofcontents

\section{Introduction}

It has been noted since the beginning of equivariant homotopy theory that there are equivariant objects which exist uniformly and compatibly for all compact Lie groups in a certain family, and which exhibit extra functoriality. For example given compact Lie groups $\Pi$ and $G$, there exists a construction for the classifying space of $G$-equivariant $\Pi$-principal bundles which is uniform on the group $G$ and which is functorial on all continuous group homomorphism, \cite{Schwede18}*{Remark 1.1.29}. Similarly, there are uniform constructions for many equivariant cohomology theories, such as K-theory, cobordism and stable cohomotopy, just to mention a few. The objects exhibiting such a ``global'' behaviour are the subject of study of \emph{global homotopy theory}.

In this paper we provide a new $\infty$-categorical model for global homotopy theory by formalizing the idea that a global stable/unstable object is a collection of $G$-objects, one for each compact Lie group $G$, which are compatible with the restriction-inflation functors. The key categorical construction that we will use to make this slogan precise is that of a partially lax limit, which we recall below. The main result of our paper is that this construction agrees with the models of global homotopy theory considered in the literature. Specifically we will compare to the models of \cite{GH} and \cite{Schwede18} in the unstable and stable case respectively. We first present our result in the simpler context of unstable global homotopy theory, and then consider the stable analogue of our main result. Finally we discuss an application of the techniques developed in this paper to proper equivariant homotopy theory.

\subsection*{Unstable global homotopy theory}
Global spaces were first proposed in \cite{GH} as a powerful framework for studying the homotopy theory of topological stacks and topological groupoids, which in turn generalize orbifolds and complexes of groups. This homotopy theory records the isotropy data of such objects as a particular diagram of fixed points spaces. To make this precise, \cite{GH} defined the $\infty$-category of \emph{global spaces} as the presheaf $\infty$-category 
\[
\Spc_{gl} = \Fun(\Glo^{\op},\Spc).
\]
Here $\Glo$ is the $\infty$-category whose objects are all compact Lie groups $G$, and whose morphism spaces are given by $\hom(H,G)_{hG}$; the homotopy orbits of the conjugation $G$-action on the space of continuous group homomorphisms. In particular, a global space $X$ consists of the data of a fixed point space $X^G$ for every compact Lie group $G$ which are functorial in all continuous group homomorphisms. Furthermore, the conjugation actions have been trivialized, reflecting the fact that spaces of isotropy are insensitive to inner automorphisms. 

This definition is motivated by Elmendorf's theorem in equivariant homotopy theory which states that the $\infty$-category of \emph{$G$-spaces} $\Spc_G$ is equivalent to the presheaf $\infty$-category on the $G$-orbit category $\OO_G$. Here $\Spc_G$ is defined as the $\infty$-categorical localization of $G$-CW-complexes at the homotopy equivalences, and $\OO_G$ is the full subcategory of $G$-spaces spanned by the transitive $G$-spaces $G/H$ for a closed subgroup $H\subseteq G$. 

There is in fact a strong connection between equivariant and global homotopy theory. Let $\Orb$ denote the wide subcategory of $\Glo$ spanned by the injective group homomorphisms. Gepner-Henriques~\cite{GH} observed that the slice $\infty$-category 
$\Orb_{/G}$ is equivalent to the $G$-orbit category $\OO_G$. In particular, this allows us to define a restriction functor 
\[
\mathrm{res}_G\colon \Spcgl \to \Fun(\OO_G^{\op}, \Spc)\simeq \Spc_G
\]
by precomposing with forgetful functor $\OO_G\simeq \Orb_{/G}\to \Glo$. Thus a global space has an associated underlying $G$-space for all compact Lie groups $G$. Furthermore, that all these $G$-spaces come from the same global object imposes strong compatibility conditions among them. 

We would like to understand how to recover a global space $X$ from its restrictions $\mathrm{res}_G X$ to all compact Lie groups $G$, together with the previously mentioned compatibility conditions. The precise sense in which this is possible requires the notion of a (partially) lax limit, which we now recall following~\cite{GHN} and \cite{Berman}.

\subsection*{Partially lax limits}
Let $\CI$ be an $\infty$-category and consider a functor $F\colon \CI\to \Cat_{\infty}$. Intuitively, the \emph{lax limit of} $F$ is the $\infty$-category $\laxlim F$ whose objects consist of the following data
\begin{itemize}
	\item  an object $X_i\in F(i)$ for each $i\in\CI$;
	\item and compatible morphisms $f_{\alpha}\colon F(\alpha)(X_i)\to X_j$ for every arrow $\alpha\colon i\to j$ in $\CI$.
\end{itemize}
A morphism $\{X_i,f_\alpha\}\rightarrow \{X_i',f'_\alpha\}$ is a suitably natural collection of maps $\{g_i\colon X_i\rightarrow X_i'\}$. More precisely, $\laxlim F$ is the $\infty$-category of sections of the cocartesian fibration associated to $F.$ For our description we will require that for certain arrows $\alpha$ in $\CI$, the map $f_\alpha$ is an equivalence. We therefore fix a collection of edges $\mathcal{W} \subset \CI$, which contains all equivalences and which is stable under homotopy and composition, and denote by $\CI^\dagger$ the resulting marked $\infty$-category. The \emph{partially lax limit} of $F$ is then the subcategory of $\laxlim F$ spanned by those objects $(\{X_i\}, \{f_\alpha\})$ for which the canonical map $f_\alpha$ is an equivalence for all edges $\alpha \in \mathcal{W}$. Note that if $\mathcal{W}$ contains only equivalences, then we recover the lax limit of $F$. On the other hand, if $\mathcal{W}$ contains all edges, we recover the usual notion of the limit of $F$. In particular we obtain canonical functors 
\[
\lim F \to \laxlimdag F \to \laxlim F,
\]
which indicates that a partially lax limit interpolates between the limit and the lax limit of a diagram.
For exposition's sake, we have only defined the partially lax limit of a functor with values in $\Cat_\infty$, but there are similar definitions if we replace $\Cat_\infty$ with $\Cat_\infty^\otimes$, the $\infty$-category of symmetric monoidal $\infty$-categories. We refer the reader to Section~\ref{sec-partially-lim} for more details on this construction.

As mentioned, in this paper we show that a global space can be thought of as a compatible collection of $G$-spaces. We can formalize what ``compatible'' means using the language of partially lax limits. To this end, let $(\Glo^{\op})^\dagger$ denote the $\infty$-category $\Glo^{\op}$ where we marked all the edges in $\Orb^{\op}\subseteq \Glo^{\op}$, i.e. all the injective edges. We prove the following theorem, which summarizes the main result of Section~\ref{sec-global-spaces}.

\begin{thm*}[\ref{thm:unstablelaxlim}]
	There exists a functor $\Spc_\bullet \colon \Glo^{\op} \to \Cat_\infty^\otimes$ 
	which sends a compact Lie group $G$ to the $\infty$-category of $G$-spaces 
	$\Spc_G$ endowed with the cartesian symmetric monoidal structure, and a 
	continuous group homomorphism $\alpha\colon H \to G$ to the restriction-inflation 
	functors. Furthermore, there is a  symmetric monoidal equivalence 
	\[
	\Spcgl \simeq \laxlimdag_{G\in(\Glo^{\op})^\dagger} \Spc_G
	\]
	between the $\infty$-category of global spaces with the cartesian monoidal 
	structure and the partially lax limit over $(\Glo^{\op})^\dagger$ of the diagram 
	$\Spc_\bullet$.
\end{thm*}
By the above theorem a global space $X$ consists of the following data:
\begin{itemize}
	\item a $G$-space $\mathrm{res}_G X$ for each compact Lie group $G$,
	\item an $H$-equivariant map $f_\alpha\colon \alpha^*\mathrm{res}_G X\to \mathrm{res}_H X$ for each continuous group homomorphism $\alpha\colon H \to G$.
	\item the maps $f_\alpha$ are functorial, so that $f_{\beta\circ \alpha}\simeq f_\beta\circ \beta^*(f_{\alpha})$ for all composable maps $\alpha$ and $\beta$, and $f_{\mathrm{id}}=\mathrm{id}$;
	\item $f_\alpha$ is an equivalence for every continuous \emph{injective} homomorphism $\alpha$.
    \item a homotopy between  the map $f_{c_g}$ induced by the conjugation isomorphism and the map $l_g \colon c_g^* \mathrm{res}_G X \to \mathrm{res}_G X$ given by left multiplication by $g$. 
	\item higher coherences for the homotopies.
\end{itemize}
This is a precise formulation of the compatibility conditions encoded in a global space.

\subsection*{Global stable homotopy theory}

Our discussion so far has been limited to the homotopy theory of global spaces, but there are also numerous examples of equivariant cohomology theories exhibiting a global behaviour. These cohomology theories are represented by global spectra, and their study is called \emph{global stable homotopy theory}.

The consideration of ``global spectra" grew out of the literature on equivariant stable homotopy theory, and was considered in works such as \cite{GreenleesMay97}. Morally, a global spectrum models a compatible family of equivariant spectra for all compact Lie groups at once. Our main result makes this moral precise, and provides the same description as in the unstable case.

There are multiple models for the homotopy theory of global spectra. In this paper we will use the framework developed by Schwede in~\cite{Schwede18}. His approach has the advantage of being very concrete; the category of global spectra is modelled by the usual category of orthogonal spectra but with a finer notion of equivalence, the global equivalences. The category of orthogonal spectra with the global stable model structure of~\cite{Schwede18}*{Theorem 4.3.17} underlies a symmetric monoidal $\infty$-category $\Spgl$.
As any orthogonal spectrum is a global spectrum, this approach comes with a good range of examples. For instance, there are global analogues of the sphere spectrum, cobordism, topological and algebraic $K$-theory spectra, Borel cohomology, symmetric product spectra and many others. Global spectra have also been shown to give cohomology theories on orbifolds and topological stacks in~\cite{Juran}, thereby establishing them as a natural home for (genuine) cohomology theories on topological stacks. As part of the framework developed by Schwede, the $\infty$-category of global spectra comes with symmetric monoidal restriction functors
\[
\mathrm{res}_G \colon \Spgl \to \Sp_G
\]
into the $\infty$-category of $G$-spectra, for all compact Lie groups $G$. As a first indication that a global spectrum should consist of just this data, together with various comparison maps, note that the functors $\mathrm{res}_G$ are jointly conservative by the very definition of global equivalences. 

However, not all equivariant spectra admit global refinements. In fact being a ``global'' object forces strong compatibility conditions between the underlying $G$-spectra for different $G$. For example, $\mathrm{res}_GX$ is always a split $G$-spectrum by~\cite{Schwede18}*{Remark 4.1.2} and its $G$-homotopy groups for all $G$ together admit the structure of a global functor, see~\cite{Schwede18}*{Example 4.2.3}. We can again formalize how a global spectrum is determined by its restrictions for all compact Lie groups using the language of partially lax limits. Recall that $(\Glo^{\op})^\dagger$ denotes the $\infty$-category $\Glo^{\op}$, marked by all the edges in $\Orb^{\op}$, i.e. the injective group homomorphisms.

\begin{thm*}[\ref{thm-laxlim-global-spectra}]\label{thm-main-result-global-spectra}
	There exists a functor $\Sp_\bullet \colon \Glo^{\op} \to \Cat_\infty^\otimes$ 
	which sends a compact Lie group $G$ to the symmetric monoidal $\infty$-category of 
	$G$-spectra $\Sp_G^\otimes$, and a continuous group homomorphism 
	$\alpha\colon H \to G$ to the restriction-inflation functor. Furthermore, there is 
	a symmetric monoidal equivalence 
	\[
	\Spgl \simeq \laxlimdag_{G\in (\Glo^{\op})^\dagger} \Sp_G
	\] 
	between Schwede's $\infty$-category of global spectra  and the partially lax limit 
	over $(\Glo^{\op})^\dagger$ of the diagram $\Sp_\bullet$.
\end{thm*}

\subsection*{Proper equivariant stable homotopy theory}
The techniques employed in the proof of Theorem~\ref{thm-laxlim-global-spectra} can also be used in other settings. Given a (not necessarily compact) Lie group $G$, we can consider the $\infty$-category of proper $G$-spectra $\Sp_{G,\pr}$. This is the $\infty$-category underlying the category of orthogonal $G$-spectra with the proper stable model structure of \cite{Proper}, in which a  map $f\colon X \to Y$ is a weak equivalence if and only if for all compact subgroups $H\leq G$, the map induced on homotopy groups $\pi_*^H(f)\colon \pi_*^H(X)\to \pi_*^H(Y)$ is an isomorphism. Write $\OO_{G,\pr}$ for the proper $G$-orbit category, which is defined to be the subcategory of $\OO_G$ spanned by the cosets $G/H$, where $H$ a compact subgroup of $G$. Our techniques allow us to prove:

\begin{thm*}[\ref{thm-lim-proper-spectra}]
	Let $G$ be a Lie group. There is a symmetric monoidal equivalence
	\[
	\Sp_{G,\pr} \simeq \lim_{H\in \OO_{G,\pr}^{\op}} \Sp_H
	\]
	between the $\infty$-category of proper $G$-spectra and the limit of the functor 
	$\Sp_\bullet$ restricted along the canonical functor 
	$\iota_G \colon \OO_{G,\pr}^{\op} \to \Glo^{\op}$ sending $G/H$ to $H$.
\end{thm*}

Having introduced the main theorems of this article. We continue the introduction by discussing the proof strategy for each in some detail.

\subsection*{The proof strategy for Theorem~\ref{thm:unstablelaxlim}} 
We begin with a discussion of the proof of the unstable result. Implicit in~\cite{Rezk} is the following crucial observation (see also Proposition~\ref{prop:Glofactsystem}): the space of factorizations of any map $\alpha\colon H\rightarrow G$ in $\Glo$ into a surjective followed by an injective group homomorphism is contractible. In fewer words, the surjective and injective maps form an orthogonal factorization system on $\Glo$. This is the main ingredient in the proof of Theorem~\ref{thm:unstablelaxlim}, and moreover, we would like to argue it is at the core of the relationship between global and $G$-equivariant homotopy theory. 

This claim is justified by the following two facts. The first is that the functoriality under the restriction-inflation functors of the different $\infty$-categories of equivariant spaces is equivalent to the previous observation. The second is that the observation formally implies that one can recover a global space $X$ from the $\Glo^{\op}$-indexed diagram of $G$-spaces $\mathrm{res}_G X$. 

Let us first explain how the $\infty$-categories of equivariant spaces are functorial in the category $\Glo^{\op}$. Due to the existence of a non-trivial topology on the morphism spaces, this is not immediate. For example, note that exhibiting this functoriality also entails giving a homotopy coherent trivialization of the conjugation action on $\Spc_G$. The key is that the existence of the orthogonal factorization system allows one to define functors \[\alpha_!\colon \Orb_{/H}\rightarrow \Orb_{/G}, \quad (K\hookrightarrow H) \mapsto (\alpha(K)\hookrightarrow G).\] On objects $\alpha_!$ factorizes the composite $K\hookrightarrow H\rightarrow G$ into a surjection followed by an injection, and then only remembers the injective part. The fact that such factorizations are unique is equivalent to the fact that this functor is well-defined. Precomposing with $\alpha_!^{\op}$ gives the standard restriction functor $\alpha^*\colon \Spc_G\rightarrow \Spc_H.$ Furthermore given this description of the individual restriction functors, it is clear that they are functorial in $\Glo^{\op}$.

Next we explain how the observation implies that one can recover a global space from its restrictions. When one takes an object $(\{\mathrm{res}_G X\},\{f_\alpha\})$ of the partially lax limit over $\Glo^\dagger$ of the diagram $\Spc_\bullet$, the functoriality of the associated global space in injections is recorded by restricting to each $\mathrm{res}_G X$, and the functoriality in surjections is given by the morphisms $f_\alpha$. One recovers the functoriality in all morphisms in $\Glo$ by factorizing an arbitrary morphism into an injection followed by a surjection. The ability to split the functoriality in this way again reduces to the observation that the surjective and injective maps form an orthogonal factorization system. We make precise all of the ideas sketched here in Section \ref{sec-global-spaces}.

\subsection*{The proof strategy for Theorem~\ref{thm-laxlim-global-spectra}}

The proof of Theorem~\ref{thm-laxlim-global-spectra} is considerably more involved than its unstable analogue, and takes up the majority of the second half of the paper. Therefore we now give an overview of the proof as a roadmap for the reader. 

Firstly, we discuss the existence of the functor $\Sp_\bullet$. Recall that a $G$-spectrum can be thought as a pointed $G$-space together with a compatible collection of deloopings for all representation spheres. With modern tools we can give this construction a universal property: as a symmetric monoidal $\infty$-category $\Sp_G$ is obtained from the $\infty$-category of pointed $G$-spaces by freely inverting the representation spheres $S^V$ for every $G$-representation $V$, see~\cite{gepner2020equivariant}*{Appendix C}. This universal property, combined with the unstable functor $\Spc_\bullet$ of Theorem~\ref{thm:unstablelaxlim}, immediately gives the functoriality of $G$-spectra in $\Glo^{\op}$ as in our theorem. 

Unfortunately, constructing the functor $\Sp_\bullet$ via the universal property of equivariant spectra is unhelpful for our purposes, as it is too inexplicit for calculating the partially lax limit. For example, note that for a surjective group homomorphism $\alpha\colon H\rightarrow G$ and $G$-spectrum $E$, to obtain the $H$-spectrum $\alpha^*E$ one has to freely add deloopings with respect to representation spheres not in the image of $\alpha^*\colon \Rep(G)\rightarrow \Rep(H)$. This is a process which one cannot easily control.

Therefore, pivotal to our proof is an explicit construction of the functor $\Sp_\bullet$. The calculation of the partially lax limit of $\Sp_\bullet$ will then follow from this by a long series of nontrivial formal arguments. The crucial idea is construct and calculate with a functoriality on prespectrum objects rather than at the level of spectrum objects. In this setting, we are able to built the functoriality of equivariant prespectra  explicitly using the functoriality of the $\infty$-categories $\OO_G$ and $\Rep(G)$, the category of representations and linear isometries. 

To make this precise, let us first specify our model of $G$-prespectra. We define an $\infty$-category $\OR_G$, naturally fibred over $\OO_G^{\op}$, whose objects are pairs $(H,V)$, where $H$ is a closed subgroup of $G$ and $V$ is a $H$-representation, see Definition~\ref{def-OR_G}. This is canonically symmetric promonoidal and so the $\infty$-category of functors $\Fun(\OR_G,\Spcp)$ is symmetric monoidal via Day convolution. There is a functor $S_G\colon \OR_G\rightarrow \Spcp$ which sends the object $(H,V)$ to the pointed space $(S^V)^H$. This is a commutative algebra object in $\Fun(\OR_G,\Spcp)$ via the universal property of Day convolution. The first ingredient of the proof is the following:

\begin{Step}
	The $\infty$-category $\Sp_G$ is equivalent to an explicit Bousfield localization of the $\infty$-category 
 \[
 \PSp_G:= \Mod_{S_G}\Fun(\OR_G,\Spcp).
 \]
\end{Step}   

We obtain this description by reinterpreting the construction of $G$-spectra as a Bousfield localization of the level model structure on orthogonal $G$-spectra internally to $\infty$-categories. This identification is the culmination of Sections \ref{sec-cat-of-eq-prespectra} and \ref{sec:eq_prespectra}, and the reader can find a precise statement as Proposition~\ref{prop-Sp_G-local} and Corollary \ref{cor:G-prespectra-are-modules}.

Having obtained this identification, we can build the functoriality of equivariant prespectra by exhibiting the pairs $(\OR_G,S_G)$ as functorial in $\Glo^{\op}$. In fact the categories $\OR_G$ will only be (pro)functorial in $\Glo^{\op}$, but this is a subtlety which we choose to gloss over in this introduction. To exhibit this functoriality, we build a global version of the category $\OR_G$ and the algebra object $S_{G}$, which we denote by $\ORgl$ and $\Sgl$, see Definition~\ref{definition:ORgl}. The $\infty$-category $\ORgl$ is naturally fibred over $\Glo^{\op}$ and has objects $(G,V)$, where $G$ is a compact Lie group and $V$ is a $G$-representation, and $S_{gl}\colon \ORgl\rightarrow \Spcp$ sends $(G,V)$ to the pointed space $(S^V)^G$.  

There is a precise sense in which the pair $(\ORgl,\Sgl)$ contain all of the functoriality of the pairs $(\OR_G,S_G)$ in $\Glo$. For the group direction this stems from the fact that the surjections and injections form an orthogonal factorization system on $\Glo$, while for the representation direction this follows from the observation that $\ORgl$ is a cocartesian fibration over $\Glo^{\op}$ classifying the functor $\Rep(-)\colon \Glo^{\op}\rightarrow \Cat_\infty$ which sends a compact Lie group $G$ to its category of $G$-representations, with functoriality given by restriction.
These observations allow us to prove the following result, see Proposition~\ref{proposition:functoriality-of-prespectra}.
\begin{Step}
There exists a functor \[\PSp_\bullet \colon \Glo^{\op}\rightarrow \Cat_\infty^\otimes,\quad G\mapsto \PSp_G.\] Furthermore the partially lax limit of $\PSp_\bullet$ over $(\Glo^{\op})^\dag$ is given by $\Mod_{S_{gl}}\Fun(\ORgl,\Spcp)$.
\end{Step}
We have shown in Step 1 that $\Sp_G$ is a Bousfield localization of $\PSp_G$. We call a map in $\PSp_G$ a stable equivalence if it is inverted by the functor $\PSp_G\rightarrow \Sp_G$. 
\begin{Step}
The diagram $\PSp_\bullet$ preserves stable equivalences, and therefore induces a diagram $\Sp_\bullet$. Furthermore, as indicated by the notation, this diagram is equivalent to the functoriality of equivariant spectra built at the beginning of this section using the universal property of $\Sp_G$.
\end{Step}
In particular, on morphisms this diagram gives the standard restriction-inflation functors on equivariant spectra, see Corollary~\ref{cor:funct-Sp_bullet}. The following result follows formally from this.
\begin{Step}
The partially lax limit of $\Sp_\bullet$ is given by an explicit Bousfield localization of the $\infty$-category 
\[
\Mod_{S_{gl}}\Fun(\ORgl,\Spcp).
\]
\end{Step}

Finally, we compare this $\infty$-category to Schwede's model of global spectra, $\Spgl$. Once again we do this by first translating his construction into one internal to $\infty$-categories. We define an $\infty$-category $\ORfgl$ as the subcategory of $\ORgl$ spanned by the objects $(G,V)$, where $V$ is a \textit{faithful} $G$-representations. Restricting $\Sgl$ we obtain a commutative algebra object $\Sfgl$ in $\Fun(\ORfgl,\Spcp)$. We then show:
\begin{Step}
$\Spgl$ is an explicit Bousfield localization of the category $\Mod_{\Sfgl}(\Fun(\ORfgl,\Spcp))$.
\end{Step}
The precise statement is obtained by combining Proposition~\ref{prop-glspectra-local-object} and Corollary \ref{cor:PSpgl=PSpfgl}. Finally we show in Section \ref{sec:global-spectra-par-lax-lim} that the canonical inclusion $j\colon \ORfgl \to \ORgl$ induces an adjunction 
\[
j_! \colon \Mod_{\Sfgl}(\Fun(\ORfgl,\Spcp)) \leftrightarrows \Mod_{\Sgl}(\Fun(\ORgl,\Spcp))\cocolon j^*
\]
on prespectrum objects. Then we show that this adjunction descends to an adjunction on the corresponding Bousfield localizations of Steps 4 and 5. Finally we prove that the fibrancy conditions imposed by these localizations cancel out the difference between all and faithful representations, so that we obtain an equivalence 
\[
\Spgl \simeq\laxlimdag \Sp_\bullet,
\]
concluding the proof of Theorem \ref{thm-laxlim-global-spectra}.

Finally let us note that to fill in all of the details of this argument requires a long list of technical results about the relationship between various constructions applied to model categories and $\infty$-categories, Day convolution monoidal structures induced by promonoidal categories, and partially lax limits of symmetric monoidal categories. We have included these in Part I to make the paper self-contained, and because we failed to find a convenient reference for many of these facts.

\subsection*{Related work}

There are many models of global unstable homotopy theory. The first was given in \cite{GH}, and since then others have been obtained in \cite{Schwede18} and \cite{Schwede20}. The second of these papers, together with \cite{korschgen}, proves that all these models induce the same $\infty$-category. Finally, we would like to mention the unpublished manuscript \cite{Rezk}, which contains many of the ideas we exploit in Section~\ref{sec-global-spaces}.

There has been a lot of work towards finding a good framework for the study of global stable homotopy theory, see~\cites{Bohmann14, GreenleesMay97} and~\cite{Lewis}*{Chapter II}. Schwede's model~\cite{Schwede18} has so far being the most successful one, in part because of its numerous applications to equivariant stable homotopy theory, see for example~\cite{Schwede2017} and~\cite{Markusannals}. Hausmann~\cite{Markus2} gave a model for global homotopy theory for the family of finite groups by endowing the category of symmetric spectra with a global model structure. There is also a model for $G$-global homotopy theory~\cite{LenzGglobal} which is a synthesis between classical equivariant homotopy theory and Schwede's global homotopy theory. This specializes to global homotopy theory by setting $G$ to be the trivial group.
Recently, Lenz~\cite{Lenz2022} gave an $\infty$-categorical model for global stable homotopy theory for the family of finite groups using spectral Mackey functors. However to the best of our knowledge, our model is the first $\infty$-categorical model for global stable homotopy theory for the family of all compact Lie groups and not just the finite ones. 

\subsection*{Future directions}
In this paper we focused only on global and proper equivariant homotopy theory, but it is quite natural to wonder if we can recover our two results as a special case of a more general one. For any Lie group $G$, we can in fact consider $G$-global homotopy theory which is a generalization of global and $G$-equivariant homotopy theory. We conjecture that $G$-global stable homotopy theory is equivalent to the partially lax limit of the functor $\Sp_\bullet$ restricted along the canonical functor $\Glo^{\op}_{/G} \to \Glo^{\op}$.

\subsection*{Organization of the paper} 
The paper is divided into three main parts. 

In the first part we first discuss the relationship between model and $\infty$-categories. Then we recall the concept of a promonoidal $\infty$-category and use this to define the Day convolution product on functor categories. We then introduce the notions of partially lax (co)limits and collect various useful results that we will need throughout the paper. We finish Part I by describing the lax limits of symmetric monoidal $\infty$-categories in terms of the operadic norm functor. 

The second part of the paper contains the proofs of our main results. In Section~\ref{sec-global-spaces} we introduce the $\infty$-category of global spaces and prove Theorem~\ref{thm:unstablelaxlim}. This is an unstable version of Theorem~\ref{thm-laxlim-global-spectra}, and serves as a warm up for the considerably more involved proof of the stable case. We therefore recommend the reader to read this section before moving forward. In Section~\ref{sec-cat-of-eq-prespectra} we recall various model structures on the categories of orthogonal $G$-spectra for a Lie group $G$, and hence define the underlying $\infty$-categories of proper $G$-spectra and of global spectra. In Section~\ref{sec:eq_prespectra} we apply a variant of Elmendorf's theorem and use this to provide specific models for the $\infty$-categories of proper $G$-prespectra and global prespectra. In Section~\ref{sec:funct_prespectra} we construct the functor $\Sp_\bullet$ from the introduction, and in Section~\ref{sec:global-spectra-par-lax-lim} we identify the partially lax limits with the $\infty$-category of global spectra. Finally in Section~\ref{sec-proper-section}, we apply the same techniques to describe the $\infty$-category of proper $G$-spectra as a limit, proving Theorem~\ref{thm-lim-proper-spectra}.

The third part of the paper contains an appendix on the tensor product of modules in an $\infty$-category.

\subsection*{Acknowledgements}
The authors thank Stefan Schwede for first suggesting that a description of global spectra as a partially lax limit should be possible. We would also like to thank Stefan Schwede, Lennart Meier, Markus Hausmann, Branko Juran and Bastiaan Cnossen for numerous helpful conversations on this topic. We also thank the referees for helpful suggestions which improved the paper.

During the preparation of this text the first author was a member of the Hausdorff Center for Mathematics at the University of Bonn funded by the German Research Foundation (DFG). This material is based upon work supported by the Swedish Research Council under grant no. 2016-06596 while the third author was in residence at Institut Mittag-Leffler in Djursholm, Sweden during the semester \textit{Higher algebraic structures in algebra, topology, and geometry}. The first author was supported by the DFG Schwerpunktprogramm 1786 ``Homotopy Theory and Algebraic Geometry'' (project ID SCHW 860/1-1).  The third author was supported by the SFB 1085 Higher Invariants in Regensburg.

\part{Partially lax limits, promonoidal \texorpdfstring{$\infty$}{infty}-categories and  Day convolution}

In this part of the paper we introduce the necessary machinery to state and prove our main 
results. In the first section we give references for the passage from topological/model categories to $\infty$-categories. We then discuss the Day convolution 
product for functor $\infty$-categories, where the source is only assumed to be a promonoidal $\infty$-category. Finally we recall the notion of partially lax limits of $\infty$-categories 
and symmetric monoidal $\infty$-categories, and proof some useful properties about them.

\section{From topological/model categories to \texorpdfstring{$\infty$}{infty}-categories} 

In this paper we will often need to pass from topological categories (or operads) and (symmetric monoidal) model categories to $\infty$-categories. In this section we recall how this is done, and provide relevant references. After this section we will largely leave these identifications implicit for the rest of the paper.

\subsection{Topological categories and operads}

We can promote a topological category $\CC$ to an $\infty$-category by first applying 
the singular functor to the mapping spaces (see~\cite{HTT}*{Section 1.1.4}) and then applying
the coherent (also called simplicial) nerve functor~\cite{HTT}*{Corollary 1.1.5.12}. This defines a functor 
\[
\mathrm{TopCat} \to \Cat_\infty
\]
from topological categories to $\infty$-categories. Importantly, applying this functor to a topologically enriched category $\CC$ preserves the set of objects and the weak homotopy type of the mapping space between any two objects, see~\cite{HTT}*{Theorem 1.1.5.13}. Throughout this paper we will not distinguish between the topological category and its $\infty$-categorical counterpart. 

There is a similar functorial construction between topological operads and $\infty$-operads, which we now recall. Given a topological coloured operad $\CO$, we let $\CO^\otimes$ denote the topological 
 category whose objects are pairs $(I_+, (C_i)_{i \in I})$ where $I_+ \in \Finp$ and $C_i$ are 
 colours in $\CO$. Given a pair of objects $C=(I_+, \{C_i\}_{i \in I})$ and 
 $D=(J_+, \{D_j\}_{j\in J})$ 
 in $\CO^\otimes$, the morphism space $\CO^\otimes(C,D)$ is given by 
 \[
 \coprod_{\alpha \colon I_+ \to J_+} \prod_{j \in J} \CO(\{C_i\}_{\alpha(i)=j}, D_j).
 \]
 Composition is defined in the obvious way. This is the topological analogue 
 of~\cite{HA}*{Notation 2.1.1.22}. Note that $\CO^\otimes$ admits a functor to $\Finp$. By the process before, this induces a functor of $\infty$-categories $\CO^\otimes \rightarrow \Finp$.

\begin{lemma}\label{ORG_operad}
    Let $\CO$ be a topological coloured operad. Then the forgetful functor 
    $p\colon \CO^\otimes \to \Finp$ defines an $\infty$-operad. Moreover this construction is functorial in the sense that it sends maps of topological coloured operads to maps 
    of $\infty$-operads. 
\end{lemma}

\begin{proof}
 Recall that a topological category is seen as 
 an $\infty$-category by applying the singular functor on mapping spaces and then by applying the coherent nerve functor to the resulting simplicial category. 
 Since the singular functor preserves products and sends every object to a 
 fibrant one, it sends the topological coloured operad $\CO$ to a fibrant\footnote{Recall that a simplicial operad is fibrant if each multispace is a fibrant simplicial set, see \cite{HA}*{Definition 2.1.1.26}.} simplicial operad $\CO_s$. Moreover by direct inspection, the singular functor sends the topological category $\CO^\otimes$ defined above to $\CO^\otimes_s$ as defined in~\cite{HA}*{Notation 2.1.1.22}. Applying the coherent nerve to $\CO^\otimes_s \to \Fin_*$ we obtain an $\infty$-operad by~\cite{HA}*{Proposition 2.1.1.27}, proving the first claim. A simple check shows that the formation of the topological category $\CO^\otimes$ is functorial in maps of topological operads. Applying the singular functor and the coherent nerve then gives a functor of $\infty$-categories over $\Finp$. Furthermore the cocartesian edges over inert edges are explicitly constructed in the proof of ~\cite{HA}*{Proposition 2.1.1.27}, and the functor constructed clearly preserves these edges.
\end{proof}

\subsection{Model categories and $\infty$-categories}
We will very often pass from model categories to $\infty$-categories. Therefore we explain and give references for this passage.

Let $\mathcal{M}$ be a model category with class of weak equivalences denoted by $W$. We always assume that $\mathcal{M}$ has functorial factorizations. The model category $\mathcal{M}$ presents an $\infty$-category which we denote by $\mathcal{M}[W^{-1}]$. We may define $\mathcal{M}[W^{-1}]$ as the Dwyer-Kan localization of $N(\mathcal{M})$ at the weak equivalences of $\mathcal{M}$, i.e. as the initial $\infty$-category with a functor from $\mathcal{M}$ which inverts the morphisms in $W$.  Write $\mathcal{M}^f$, $\mathcal{M}^c$, and $\mathcal{M}^\circ$ for the full subcategories of $\mathcal{M}$ spanned by the fibrant, cofibrant and bifibrant objects respectively. The composite
\[
N(\mathcal{M}^f)\rightarrow N(\mathcal{M})\rightarrow \mathcal{M}[W^{-1}] 
\]
is a Dwyer-Kan localization at the restriction of $W$ to $\mathcal{M}^f$, and similarly for the case of cofibrant and bifibrant objects. See for example the discussion in~\cite{HA}*{Remark 1.3.4.16}. 

If $\mathcal{M}$ is a topological model category, then the enriched structure gives another construction of $\mathcal{M}[W^{-1}]$. In this case, $\mathcal{M}[W^{-1}]$ is equivalent to the $\infty$-category associated to the topologically enriched category $\mathcal{M}^\circ$ as in the previous section, see~\cite{HA}*{Theorem 1.3.4.20}. Throughout our paper it will be necessary to use all these different constructions of $\mathcal{M}[W^{-1}]$. 

We note that if the model category $\mathcal{M}$ is cofibrantly generated and the underlying category is locally presentable, then $\mathcal{M}[W^{-1}]$ is a presentable $\infty$-category, see \cite{HA}*{Proposition 1.3.4.22}. Also we note that any Quillen adjunction of model categories $F \colon \mathcal{M}_0 \rightleftarrows\mathcal{M}_1\cocolon G$ induces an adjunction of underlying $\infty$-categories $F \colon \mathcal{M}_0[W_0^{-1}] \rightleftarrows \mathcal{M}_1[W_1^{-1}]\cocolon G$ by~\cite{Hinich}*{Proposition 1.5.1}.

Next we may consider symmetric monoidal model categories. By ~\cite{HA}*{Proposition 4.1.7.6}, if $\mathcal{M}$ is a symmetric monoidal model category then the $\infty$-category $\mathcal{M}[W^{-1}]$ admits a symmetric monoidal structure such that the localization functor $\CM^c\rightarrow \CM[W^{-1}]$ is strong monoidal, and if $F$ is a symmetric monoidal left Quillen functor then $F$ is again symmetric monoidal.

Once again we obtain a different construction of the symmetric monoidal $\infty$-category $\mathcal{M}[W^{-1}]$ when $\mathcal{M}$ is topological. Namely one can first restrict to bifibrant objects and then form the topological coloured operad $N^{\otimes}(\mathcal{M})$ with colors $X\in \mathcal{M}^{\circ}$ and multi-morphism spaces 
\[\Mul_{N^{\otimes}(\mathcal{M}^\circ)}(\{X_1,\dots,X_n\}, Y) = \Map_{\mathcal{M}^\circ}(X_1\otimes \dots \otimes X_n,Y).\] This then gives an $\infty$-operad by Lemma \ref{ORG_operad}. By \cite{HA}*{Proposition 4.1.7.10} this is in fact a symmetric monoidal $\infty$-category whose underlying $\infty$-category is equivalent to $\mathcal{M}[W^{-1}]$. Furthermore, by \cite{HA}*{Corollary 4.1.7.16}, these two methods of obtaining a symmetric monoidal structure on $\mathcal{M}[W^{-1}]$ are equivalent.

\subsection{Pointed categories}\label{subsec-pointed}
Many of the typical constructions one applies to model categories admit an analogue internally to $\infty$-categories. Furthermore, in many cases these constructions are not only analogous but in fact equivalent. 

For example we may consider the formation of pointed objects. Given a model category $\CM$ with final object $*$, we can equip the slice category $\CM_* = \CM_{\ast/}$ with a model structure in which fibrations, cofibrations and weak equivalences are detected by the forgetful functor $\CM_* \to \CM$, see ~\cite{Hovey}*{Proposition 1.1.8}. 
If $\CM$ is cofibrantly generated with set of generating cofibrations $I$ and set of generating acyclic cofibrations $J$, then $\CM_*$ is also cofibrantly generated by the sets $I_+$ and $J_+$, see~\cite{Hovey}*{Lemma 2.1.21}. If $\CM$ is symmetric monoidal with cofibrant unit given by $*$, then the slice category $\CM_*$ with the smash product is again a symmetric monoidal model category with cofibrant unit, see~\cite{Hovey}*{Proposition 4.2.9}. 

Let us now discuss the same construction for $\infty$-categories. Given a presentable symmetric monoidal $\infty$-category $(\CC,\otimes)$, we can endow the slice $\CC_* = \CC_{*/}$ with a symmetric monoidal structure $\wedge_\otimes$ given as follows: for all $(* \to C),(* \to D) \in \CC_*$, we define $C\wedge_\otimes D$ by the following pushout in $\CC$: 
\[
\begin{tikzcd}
	{C\otimes * \sqcup *\otimes D} & {C\otimes D} \\
	{*\otimes *} & {C\wedge_\otimes D}.
	\arrow[from=2-1, to=2-2]
	\arrow[from=1-1, to=1-2]
	\arrow[from=1-2, to=2-2]
	\arrow[from=1-1, to=2-1]
	\arrow["\lrcorner"{anchor=center, pos=0.125, rotate=180}, draw=none, from=2-2, to=1-1]
\end{tikzcd}
\]
The existence of such symmetric monoidal structure on $\CC_*$ is a formal consequence of~\cite{HA}*{Proposition 4.8.2.11} as we now explain. Indeed the cited reference shows that the functor $(-)_*\colon \mathrm{Pr}^{\mathrm{L}} \to \mathrm{Pr}^{\mathrm{L}}_*$ from presentable $\infty$-categories to pointed presentable $\infty$-categories is a smashing localization, so it induces a functor on commutative algebras $\CAlg( \mathrm{Pr}^{\mathrm{L}})\to \CAlg(\mathrm{Pr}_*^{\mathrm{L}})$ showing that a symmetric monoidal structure on $\CC_*$ exists. 
Furthermore ~\cite{HA}*{Proposition 4.8.2.11} implies that this symmetric monoidal structure is uniquely determined by the condition that the tensor product on $\CC_*$ commutes with colimits on each variable and makes the functor $(-)_+ \colon \CC \to \CC_*$ into a symmetric monoidal functor. From this one obtains the concrete description of $\wedge_\otimes$ as given above. 

\begin{example}
	Applying this construction to $\Spc$ with the cartesian product returns $\Spcp$, the category of pointed spaces with the smash product. We write $\Spc^\times$ for the $\infty$-operad giving the former, and $\Spcp^\wedge$ for the latter.
\end{example}

We now give a result that connects these two constructions. 

\begin{proposition}\label{prop-ptd-obj-model}
	Let $\CM$ be a symmetric monoidal model category with cofibrant final object, which is also the monoidal unit. 
	Suppose that the underlying $\infty$-category $\CM[W^{-1}]$ is presentable. 
	Then the functor $(-)_+ \colon \CM \to \CM_*$ induces a symmetric monoidal equivalence
	\[
	 (\CM[W^{-1}])_*\simeq \CM_*[W^{-1}] .
	\]
\end{proposition}

\begin{proof}
	First we note that the underlying $\infty$-category $\CM_*[W^{-1}]$ models the $\infty$-categorical slice $(\CM[W^{-1}])_*$, see for example \cite{Cisinski}*{Corollary 7.6.13}. Note also that $(-)_+\colon \CM\rightarrow \CM_*$ is left Quillen and strong monoidal, and therefore we obtain a strong monoidal colimit preserving functor 
	\[
	(-)_+\colon \CM[W^{-1}]\rightarrow \CM_*[W^{-1}]
	\]
	which is equivalent to the standard left adjoint $(-)_+$ under the equivalence $\CM_*[W^{-1}]\simeq \CM[W^{-1}]_*$ by inspection. Also, $\CM_*[W^{-1}]$ is automatically presentable and closed monoidal. Now we can conclude the result, because there is a unique closed symmetric monoidal structure on $\CM[W^{-1}]_*$ such that $(-)_+$ is strong monoidal.
\end{proof}

Next we consider the formation of module categories. Recall that given a symmetric monoidal $\infty$-category $\CC$ and a commutative algebra object $S\in \CAlg(\CC)$, the category of $S$-modules in $\CC$, $\Mod_S(\CC)$ is a symmetric monoidal category via the relative tensor product, constructed in \cite{HA}*{Section 4.5.2}. We will always consider $\Mod_S(\CC)$ as symmetric monoidal in this way.

\begin{proposition}\label{prop:Modules_localization_commute}
	Let $\CM$ be a symmetric monoidal and cofibrantly generated model category with weak 
	equivalences $W$, generating cofibrations $I$ and generating acyclic cofibrations $J$, 
	and let $A$ be a commutative algebra object in $\CM$ whose underlying object is cofibrant. Suppose that $\Mod_A(\CM)$ 
	admits a symmetric monoidal and cofibrantly generated model structure where fibrations 
	and weak equivalences are tested on underlying objects, and the sets $A \otimes I$ and $A \otimes J$ 
	form a set of generating cofibrations and generating acyclic cofibrations respectively. 
	Write $W_m$ for the class of weak equivalences in $\Mod_A(\CM)$. Then applying $\Mod_A$ to the functor $\CM^c\rightarrow \CM[W^{-1}]$ induces a symmetric monoidal equivalence
    \[
    \Mod_{A}(\CM)[W_m^{-1}] \simeq \Mod_{A}(\CM[W^{-1}]).
    \] 
\end{proposition}

\begin{proof}
	This is essentially~\cite{HA}*{4.3.3.17}. However since the statement there does not literally apply, let us spell out the argument.  We need to show that there exists a symmetric monoidal 
	equivalence
	\[
	\theta \colon N(\Mod_A(\CM)^c)[W_m^{-1}] \xrightarrow{\simeq} \Mod_A(N(\CM^c)[W^{-1}]).
	\] 
	We start by noting that the forgetful functor $U\colon \Mod_A(\CM) \to \CM$ is left Quillen. One can verify this by observing that $U$ sends the generating (acyclic) cofibrations to (acyclic) cofibrations, using that 
	$A$ is cofibrant and that $\CM$ satisfies the pushout-product axiom. Since a cofibrant $A$-module is then also cofibrant in $\CM$, there 
	exists a symmetric monoidal functor 
	\[
	N(\Mod_A(\CM)^c) \to N(\Mod_A(\CM^c))\simeq \Mod_A(N(\CM^c)).
	\]
	Postcomposing with the symmetric monoidal functor $N(\CM^c)\to N(\CM^c)[W^{-1}]$ 
	and using the universal property of symmetric monoidal localization we obtain a 
	symmetric monoidal functor $\theta$ as claimed. To show that $\theta$ is an equivalence, 
	we apply~\cite{HA}*{4.7.3.16} to the diagram
	\[
	\begin{tikzcd}
		N(\Mod_A(\CM)^c)[W_m^{-1}] \arrow[rr, "\theta"] \arrow[rd, "U"'] & & 
		\Mod_A(N(\CM^c)[W^{-1}]) \arrow[dl, "U'"] \\
		& N(\CM^c)[W^{-1}]. & 
	\end{tikzcd}
	\]
	We need to check:
	\begin{itemize}
		\item[(a)] The $\infty$-categories $N(\Mod_A(\CM)^c)[W_m^{-1}]$ and 
		$\Mod_A(N(\CM^c)[W^{-1}])$ admit geometric realization of simplicial objects. In fact, 
		both categories admit all colimits. For $ N(\Mod_A(\CM)^c)[W_m^{-1}]$ this is ~\cite{BHH2017}*{Theorem 2.5.9}. For $ \Mod_A(N(\CM^c)[W^{-1}])$, we note that 
		$N(\CM^c)[W^{-1}]$ admits all colimits by the previous reference and
		that these can be calculated as homotopy colimits in the model category 
		by~\cite{BHH2017}*{Remark 2.5.7}. 
		Since $A$ is cofibrant, the functor $A\otimes -\colon \CM \to \CM$ is left Quillen and so 
		it induces a colimit preserving functor $N(\CM^c)[W^{-1}]\to N(\CM^c)[W^{-1}]$ by~\cite{Hinich}*{Proposition 1.5.1}.
		Finally, we can invoke~\cite{HA}*{Proposition 4.3.3.9} to deduce the existence of all colimits in 
		$\mathrm{Mod}_A(N(\CM^c)[W^{-1}])$.
		\item[(b)] The functors $U$ and $U'$ admits left adjoints $F$ and $F'$. 
		The existence of a left adjoint to $U$ follows from the fact that $U$ is determined by a 
		right Quillen functor. 
		The existence of a left adjoint to $U'$ follows from~\cite{HA}*{Corollary 4.3.3.14}.
		\item[(c)] The functor $U'$ is conservative and preserves geometric realizations 
		of simplicial objects. This follows from~\cite{HA}*{Corollary 4.3.3.2, Proposition 4.3.3.9}.
		\item[(d)] The functor $U$ is conservative and preserves geometric realizations of 
		simplicial objects. The first assertion is immediate from the definition of the weak  equivalences in $\Mod_A(\CM)$, and the second follows from the fact that $U$ is also a 
		left Quillen functor.
		\item[(e)] The natural map $U'\circ F' \to U \circ F$ is an equivalence. 
		Unwinding the definitions, we are reduced to proving that if $N$ is a cofibrant object of 
		$\CM$, then the natural map $N \to A \otimes N$ induces an equivalence 
		$F'(N)\simeq A \otimes N$. This follows from the explicit description of $F'$ given 
		in~\cite{HA}*{Corollary 4.3.3.13}.
	\end{itemize}
\end{proof}

\begin{remark}
Suppose $\mathcal{M}$ is a symmetric monoidal cofibrantly generated model category. If $\mathcal{M}$ is locally presentable, then the existence of the model structure on $\Mod_A(\mathcal{M})$ as in Proposition~\ref{prop:Modules_localization_commute} holds by ~\cite{SchwedeSchipley}*{Remark 4.2}.
\end{remark}

\section{Promonoidal \texorpdfstring{$\infty$}{infty}-categories and Day convolution}
We start this section by recalling the notion of a promonoidal $\infty$-category. We recall the definition of the operadic norm functor and use this to define the Day convolution product on a functor category. We then collecting various important results about the Day convolution product which will be important later. We finish the section by giving a symmetric monoidal recognition criteria for presheaf categories, inspired by Elmendorf's theorem.

We start off by recalling the following useful notion from~\cite{Ayala-Francis}*{Definition 0.7}.

\begin{definition} 
A functor $p \colon \CC\to\CB$ between $\infty$-categories is an \emph{exponentiable fibration} if the pullback functor $p^*\colon {\Cat_\infty}_{/\CB}\to {\Cat_\infty}_{/\CC}$ admits a right adjoint $p_*$, which we call the pushforward. 
\end{definition}

\begin{example}\label{ex-(co)cartesian-are-exponentiable}
Both cocartesian and cartesian fibrations are exponentiable, see~\cite{Ayala-Francis}*{Lemma 2.15}.
\end{example}

\begin{example}\label{exp-stable-under-pullback}
Exponentiable fibrations are stable under pullbacks, see~\cite{Ayala-Francis}*{Corollary 1.17}
\end{example}

For any $\infty$-operad $\CO^\otimes$, we let $\CO^\otimes_{act}\coloneqq\CO^\otimes\times_{\Finp}\Fin\subseteq \CO^\otimes$ denote the subcategory of active arrows. We recall the following definition from~\cite{Exp2}*{Definition~10.2}.

\begin{definition}\label{def-promonoidal}
Let $\CO^\otimes$ be an $\infty$-operad. A map of $\infty$-operads $p\colon \CC^\otimes\to \CO^\otimes$ defines a $\CO^\otimes$-\emph{promonoidal} $\infty$-category if the restricted functor $p_{act}\colon \CC^\otimes_{act}\to \CO^\otimes_{act}$ is exponentiable. A functor of $\CO^\otimes$-promonoidal $\infty$-categories is simply a map of $\CO^\otimes$-operads. 
\end{definition}

\begin{example}
Any $\CO^\otimes$-symmetric monoidal $\infty$-category is $\CO^\otimes$-promonoidal by Example~\ref{ex-(co)cartesian-are-exponentiable}.
\end{example}

\begin{example}\label{example:cocartesian-is-promonoidal}
	Let $\CC$ be an $\infty$-category. Then the $\infty$-operad $\CC^\amalg\to\Finp$ of \cite{HA}*{Construction~2.4.3.1} is a symmetric promonoidal $\infty$-category. In fact
	\[\CC^\amalg\times_{\Finp}\Fin\to\Fin\]
	is the cartesian fibration which classifies the functor sending $I$ to $\Fun(I,\CC)$.
\end{example}

\begin{example}\label{example:cart-over-cocart-pro}
	Consider a cartesian fibration $p\colon \CC\rightarrow \CI$. Similarly to Example~\ref{example:cocartesian-is-promonoidal}, one can show that the induced map $p^\amalg\colon \CC^\amalg \rightarrow \CI^\amalg$ exhibits $\CC^\amalg$ as a $\CI^\amalg$-promonoidal $\infty$-category.
\end{example}

The key property of promonoidal $\infty$-categories is that they induce operadic norm functors. 

\begin{definition}\label{def-norm}
  Let $p\colon \CC^\otimes\rightarrow \CO^\otimes$ be a $\CO^\otimes$-promonoidal 
  $\infty$-category. Then the functor 
  \[
  p^*\colon (\Op_\infty)_{/\CO^\otimes}\to(\Op_\infty)_{/\CC^\otimes}
  \] 
  has a right adjoint by~\cite{Exp2}*{Theorem/Construction 10.6}, which we denote by 
  $N_p$ and call the \emph{norm} along $p$. Note that $p^*$ also has a left 
  adjoint $p_!$ which is given by postcomposition with $p$.
\end{definition}

The norm interacts well with pullbacks along maps of $\infty$-operads.

\begin{lemma}\label{lemma:norm-and-pullbacks}
 Let $p\colon \CC^\otimes\to \CO^\otimes$ be a $\CO^\otimes$-promonoidal 
 $\infty$-category and let $f\colon \CP^\otimes\to \CO^\otimes$ be a map of $\infty$-
 operads. Write $p'\colon \CC^\otimes\times_{\CO^\otimes}\CP^\otimes\to \CP^\otimes$ 
 and $f'\colon \CC^\otimes\times_{\CO^\otimes}\CP^\otimes\to \CO^\otimes$ for the functors obtained via basechange. Then there is a natural equivalence of functors 
    \[
    f^*N_p\simeq N_{p'}(f')^* \colon (\Op_\infty)_{/\CC^\otimes}\to (\Op_\infty)_{/\CP^\otimes}\,.
    \]
   In other words, for every $\CD^\otimes \in (\Op_\infty)_{/\CC^\otimes}$ there is an equivalence of $\infty$-operads over $\CP^\otimes$
    \[N_p(\CD^\otimes)\times_{\CO^\otimes}\CP^\otimes \simeq N_{p'}(\CD^\otimes\times_{\CO^\otimes}\CP^\otimes)\,.\]
\end{lemma}
\begin{proof}
    To check that two right adjoint functors are equivalent it is enough to check that the left adjoints are equivalent. But the left adjoint of $f^*$ is just postcomposition with $f$, so the thesis is equivalent to the fact that for every $\CE^\otimes\in (\Op_\infty)_{/\CP^\otimes}$, there is a natural equivalence
    \[\CE^\otimes\times_{\CO^\otimes}\CC^\otimes\simeq \CE^\otimes\times_{\CP^\otimes}(\CP^\otimes\times_{\CO^\otimes}\CC^\otimes)\]
    and this is clear.
\end{proof}

\begin{remark}
In a similar vein we observe that because $q^*p^* \simeq (pq)^*$, also $N_{pq} \simeq N_pN_q$.
\end{remark}

\begin{remark} \label{rem-norm-underlying}
	Recall that passing to underlying $\infty$-categories gives a functor 
	$U\colon \Op_{\infty}\to \Cat_\infty$ which admits a left adjoint $F$ with 
	essential image precisely those $\infty$-operads 
	$q\colon \CP^\otimes \to \Fin_\ast$ such that the functor $q$ factors through
	$\mathrm{Triv}\subseteq \Fin_\ast$, see~\cite{HA}*{Proposition 2.1.4.11}. 
	In particular for any $\infty$-operad $\CO^\otimes$, we obtain an adjunction on overcategories:
	\[
	F\colon (\Cat_\infty)_{/\CO}\to 
	(\Op_\infty)_{/\CO^\otimes}\cocolon U,
	\]
	see \cite{HTT}*{Proposition 5.2.5.1}
	Let $p\colon \CC^\otimes \to \CO^\otimes$ be a $\CO^\otimes$-promonoidal 
	$\infty$-category; we will now describe the effect of $N_p$ on underlying 
	$\infty$-categories. Observe that the underlying map $U(p)$ on $\infty$-categories is exponentiable,
	as it can be described as the pullback of $p$ along $\CO \subseteq \CO^\otimes$,
	compare with Example~\ref{exp-stable-under-pullback}.
	One can compute that the following diagram of left adjoints
\[\begin{tikzcd}
	{(\Op_\infty)_{/\CC^\otimes}} & {(\Op_\infty)_{/\CO^\otimes}} \\
	{(\Cat_\infty)_{/\CC}} & {(\Cat_\infty)_{/\CO}}
	\arrow["F"',from=2-2, to=1-2]
	\arrow["{p^*}"', from=1-2, to=1-1]
	\arrow["{U(p)^*}", from=2-2, to=2-1]
	\arrow["F",from=2-1, to=1-1]
\end{tikzcd}\]
commutes. Therefore the associated diagram of right adjoints
\[\begin{tikzcd}
	{(\Op_\infty)_{/\CC^\otimes}} & {(\Op_\infty)_{/\CO^\otimes}} \\
	{(\Cat_\infty)_{/\CC}} & {(\Cat_\infty)_{/\CO}}
	\arrow["U",from=1-2, to=2-2]
	\arrow["{N_p}", from=1-1, to=1-2]
	\arrow["{U(p)_*}"', from=2-1, to=2-2]
	\arrow["U"',from=1-1, to=2-1]
\end{tikzcd}\] also commutes, and we conclude that on underlying categories $N_p$ is given by the pushforward $U(p)_*$.
\end{remark}

We can now define the Day convolution functor.

\begin{definition}\label{def-day-conv}
	Let $p\colon \CC^\otimes\rightarrow \CO^\otimes$ be a $\CO^\otimes$-promonoidal 
	$\infty$-category. The \emph{Day convolution functor}
	\[
	\Fun_\CO(\CC,-)^{Day}\colon (\Op_\infty)_{/\CO^\otimes}\to 
	(\Op_\infty)_{/\CO^\otimes}
	\] 
	is the right adjoint of the functor 
	\[
	p_!p^*= -\times_{\CO^\otimes}\CC^\otimes\colon (\Op_\infty)_{/\CO^\otimes}\to(\Op_\infty)_{/\CO^\otimes} .
	\]
	This is a composite of right adjoints, and so we conclude that $\Fun_\CO(\CC,-)^{Day} \simeq N_p\, p^*(-)$. This also shows the existence of $\Fun_\CO(\CC,-)^{Day}$. When $\CO = \Finp$, we will omit it from the notation.
\end{definition}

\begin{remark}
Recall that $\Alg_{\CC^\otimes}(\CD^\otimes)$ is defined to be the full subcategory of $\Fun_{/\Finp}(\CC^\otimes,\CD^\otimes)$ spanned by the maps of operads, and that taking the maximal sub $\infty$-groupoid of this category gives the mapping spaces $\Op_{\infty}(\CC^\otimes,\CD^\otimes)$. Therefore we may view $\Alg_{({-})}({-})$ as constituting an enrichment of $\Op_{\infty}$ in $\Cat_\infty$. A standard argument shows that the adjunction equivalence \[\Op_{\infty}(\CP^\otimes,\Fun(\CJ^\otimes,\CC^\otimes)^{Day})\simeq \Op_{\infty}(\CP^\otimes\times_{\Finp}\CJ^\otimes,\CC^\otimes)\] improves to an equivalence
\[\Alg_{\CP^\otimes}(\Fun(\CJ^\otimes,\CC^\otimes)^{Day})\simeq \Alg_{\CP^\otimes\times_{\Finp}\CJ^\otimes}(\CC^\otimes).\]
\end{remark}

\begin{example}\label{ex-cocart-day-convolution}
Recall from Example \ref{example:cocartesian-is-promonoidal} that for any $\infty$-category $\CC$, the $\infty$-operad $\CC^\amalg\rightarrow \Finp$ is promonoidal. For every $\infty$-operad $\CD^\otimes$, the Day convolution $\infty$-operad $\Fun(\CC^\amalg,\CD^\otimes)^{Day}$ is equivalent to the pointwise operad structure on $\Fun(\CC,\CD)$. Indeed they corepresent the same functor by \cite{HA}*{Theorem~2.4.3.18}.
\end{example}
 
 The description of Day convolution combined with Remark \ref{rem-norm-underlying} implies that on underlying categories $\Fun_\CO(\CC^\otimes,-)^{Day}$ is given by $U(p)_*U(p)^*$. We can describe the fibres of this category explicitly.
 
 \begin{construction}\label{con-exp-fibration-fibres}
 	Let $p\colon \CC\to \CB$ be an exponentiable fibration of $\infty$-categories and 
 	$q\colon \CD\to \CB$ any functor. Fix an arrow $f\colon b_0\to b_1$ in $\CB$ and let us write 
 	$\CC_{b_i}$ and $\CD_{b_i}$ for the fibres of $p$ and $q$ over $b_i$. 
 	The unit of the adjunction $(p^*,p_*)$ gives a canonical functor 
 	$p_*p^*\CD \to \CB$ whose fibre over $b_i$ can be identified with
 	\begin{equation}\label{eq-fibre}
 		(p_*p^*\CD)_{b_i}\simeq \Fun_\CB(\{b_i\},p_*p^*\CD)\simeq \Fun_{\CC}(\CC\times_\CB\{b_i\},\CC\times_\CB\CD)\simeq \Fun(\CC_{i},\CD_{i})\,.
 	\end{equation}
 \end{construction}
 
\begin{remark}\label{rem-day-conv-underlying} 
One should be careful to note that, if the underlying $\infty$-category $\CO$ of $\CO^\otimes$ is not contractible, then the underlying $\infty$-category of $\Fun_\CO(\CC^\otimes,\CD^\otimes)^{Day}$ is not the same as the $\infty$-category of functors over $\CO$. Rather, it is a fibration over $\CO$ whose global sections are $\Fun_{/\CO}(\CC,\CD)$. Compare also with the previous construction. 
\end{remark}

We would like to have a formula for the multimapping spaces for the Day convolution. 
We will achieve this in Lemma~\ref{lem:mapping-spaces-Day-convolution} below. In preparation for this result, we compute the mapping spaces in a pushforward. To state the result we recall the definition of twisted arrow $\infty$-categories, and the notion of coends.

\begin{definition}\label{def-twisted-arrow}
	Let $\epsilon \colon \Delta \to \Delta$ be the functor 
	$[n]\mapsto [n]\star [n]^{\op} \simeq [2n+1]$. Let $\CI$ be an $\infty$-category. The twisted arrow $\infty$-category $\Tw(\CI)$ is the associated $\infty$-category of the simplicial set $\epsilon^*N\CI$. By definition, we have 
	\[
	\Tw(\CI)_n=\Map(\Delta^n\star (\Delta^n)^{\op}, \CI).
	\]
	The natural transformations $\Delta^\bullet$ and $(\Delta^\bullet)^{\op}\to \Delta^\bullet \star (\Delta^\bullet)^{\op}$ induce a functor $(s,t)\colon \Tw(\CI) \to \CI \times \CI^{\op}$.
\end{definition}

\begin{remark}
	There are two possible conventions for defining $\Tw(-)$. In this paper we follow that of Lurie~\cite{HA}*{Section 5.2.1}. This is the opposite of the convention used in ~\cite{Barwick}.
\end{remark}

\begin{example}
	The objects of $\Tw(\CI)$ are given by edges of $\CI$.
	An edge from $f\colon x \to y$ to $f' \colon x' \to y'$ in $\Tw(\CI)$ is represented by a diagram
	\[
	\begin{tikzcd}
		x \arrow[d, "f"'] \arrow[r]& x' \arrow[d, "f'"] \\
		y  & y' \arrow[l]
	\end{tikzcd}
	\]
\end{example}

\begin{remark}\label{rem-op-twisted-cat}
The twisted arrow category is insensitive to taking opposites, meaning that $\Tw(\CI^{\op})\simeq\Tw(\CI)$. However under this equivalence $(s,t)$ is sent to $(t,s)$.
\end{remark}

\begin{definition}
Given a functor $F\colon \CC\times \CC^{\op} \rightarrow \Spc$, we define the coend $\int^{x\in \CC} F(x,x)$ to equal the colimit of the functor 
\[\Tw(\CC)\xrightarrow{(s,t)} \CC\times \CC^{\op} \xrightarrow{F} \Spc.\] Dually for a functor $F\colon \CC^{\op}\times \CC \rightarrow \CC,$ we define the end $\int_{x\in \CC} F(x,x)$ to be the limit of the functor
\[\Tw(\CC)^{\op}\xrightarrow{(s,t)^{\op}} \CC^{\op}\times \CC \xrightarrow{F} \Spc.\]
\end{definition}

We are now ready to state the formula for multimapping spaces in the Day convolution.

\begin{lemma}\label{lem:mapping-spaces-pushforward}
    Suppose we are in the setting of Construction~\ref{con-exp-fibration-fibres}.
    Let $F_i\colon \CC_i\to \CD_i$ be two objects of $(p_*p^*\CD)_{b_i}$, viewed as such via the equivalence~(\ref{eq-fibre}). Then there is an equivalence
    \begin{equation}\Map^f_{p_*p^*\CD}(F_0,F_1)\simeq \int_{(x_0,x_1)\in\CC_0^{\op}\times\CC_1} \Map(\Map_\CC^f(x_0,x_1),\Map_{\CD}^f(F_0x_0,F_1x_1)) \label{eq:maps-in-exponential}\end{equation}
    where the left hand side denotes the fibre over $f$ of the canonical map 
    $\Map_{p_*p^*\CD}(F_0, F_1) \to \Map_{\CB}(b_0,b_1)$.
\end{lemma}
\begin{proof}
    Let us write $f$ as a map $\Delta^1\to\CB$. Then, by the definition of $p_*$ there is an equivalence
    \[\Map_\CB(\Delta^1,p_*p^*\CD)\simeq \Map_\CC(\Delta^1\times_\CB\CC,\CC\times_\CB\CD)\simeq \Map_{\Delta^1}(\Delta^1\times_\CB\CC,\Delta^1\times_\CB\CD)\,.\]
    Therefore we have an equivalence
    \[\begin{split}\Map^f_{p_*p^*\CD}(F_0,F_1)
                                    &\simeq \{(F_0,F_1)\}\times_{\Map_\CB(\partial\Delta^1,p_*p^*\CD)} \Map_\CB(\Delta^1,p_*p^*\CD)\\
                                    &\simeq \{(F_0,F_1)\}\times_{\Map(\CC_0,\CD_0)\times\Map(\CC_1,\CD_1)} \Map_{\Delta^1}(\Delta^1\times_\CB \CC,\Delta^1\times_\CB \CD)\,.\end{split}\]
    Now from the proof of \cite{Ayala-Francis}*{Lemma~4.2} it follows that the map
    \[{\Cat_\infty}_{/\Delta^1}\to \Cat_\infty\times\Cat_\infty\qquad [\CC\to \Delta^1]\mapsto(\CC\times_{\Delta^1}\{0\},\CC\times_{\Delta^1}\{1\})\,,\]
    is a right fibration classified by the functor $\Cat_\infty\times\Cat_\infty\to\Spc$ sending $(\CC_0,\CC_1)$ to $\Map(\CC_0^{op}\times\CC_1,\Spc)$. Therefore
    \[\begin{split}\{(F_0,F_1)\}\times_{\Map(\CC_0,\CD_0)\times\Map(\CC_1,\CD_1)} \Map_{\Delta^1}(\Delta^1\times_\CB \CC,\Delta^1\times_\CB \CD)\mkern-256mu\\
        &\simeq\Map^{(F_0,F_1)}_{\Cat_\infty{}_{/\Delta^1}}(\Delta^1\times_\CB \CC,\Delta^1\times_\CB \CD)\\
        &\simeq\Map_{(\Cat_\infty{}_{/\Delta^1})_{(\CC_0,\CC_1)}}(\Delta^1\times_\CB \CC, (F_0,F_1)^*(\Delta^1\times_\CB \CD))\\
        &\simeq\Map_{\Fun(\CC_0^{\op}\times\CC_1,\Spc)}(\Map_\CC^f(-,-), \Map^f_\CD(F_0-,F_1-))\,.\end{split}\]
   But this is exactly the thesis, thanks to \cite{GHN}*{Proposition~5.1}.
\end{proof}

\begin{remark}\label{rem:pushforward-in-exponential}
    In the setting of Lemma~\ref{lem:mapping-spaces-pushforward}, suppose that $q$ is equal to the projection $\CD\times\CB\rightarrow \CB$ and that $\CD$ is cocomplete. Then we can interpret formula~\ref{eq:maps-in-exponential} as saying that $p_*p^*\CD$ is a cocartesian fibration and that given $f\colon i\rightarrow j$, the induced functor \[f_!\colon \Fun(\CC_i,\CD)\rightarrow \Fun(\CC_j,\CD)\] evaluated on a functor $F\colon \CC_i\rightarrow \CD$ gives the functor
    \[\CC_j\rightarrow \CD,\, x_j \mapsto \int^{x_i\in \CC_i}\Map_{\CC_{ij}}(x_i,x_j)\times F(x_i),\]
    where $\CC_{ij}:=\CC\times_{\CB,f}[1]$. That is, $f_!F$ is computed by left Kan extending $F$ along the inclusion $\CC_i\subseteq \CC_{ij}$ and then restricting to $\CC_j\subseteq \CC_{ij}$. In particular, if $\CC_{ij}\to [1]$ is a cartesian fibration we have $f_!F\simeq F\circ f^*$ where $f^*\colon \CC_1\to\CC_0$ is the pullback.
\end{remark}

Recall the following notion of multimapping spaces. 

\begin{definition}\label{def:multimapping-spaces}
    Let $\CC^\otimes\to\CO^\otimes$ be a map of $\infty$-operads and let $\phi\colon\{x_i\}\to y$ be an active morphism of $\CO^\otimes$ with target in $\CO\coloneqq (\CO^\otimes)_{1_+}$. For every $\{c_i\}\in (\CC^\otimes)_{\{x_i\}}\simeq \prod_i\CC_{x_i}$ and $d\in \CC_{y}$, objects of $\CC^\otimes$ over the source and target of $\phi$, we define the $\phi$-multimapping space in $\CC^\otimes$ as the space of morphisms $\{c_i\}\to d$ above $\phi$:
    \[\Mul^\phi_{\CC^\otimes}(\{c_i\},d)\coloneqq \Map_{\CC^\otimes}(\{c_i\},d)\times_{\Map_{\CO^\otimes}(\{x_i\},y)}\{\phi\}\,.\]
    We say that $\CC^\otimes$ is \emph{representable} if for every active morphism $\phi$ and objects $\{c_i\}$, the functor \[\Mul^\phi_{\CC^\otimes}(\{c_i\},-)\colon \CC\to \Spc\] is corepresentable. In this case we write $\bigotimes_\phi\{c_i\}$ for the corepresenting object and we call it the $\phi$-tensor product of $\{c_i\}$. This is equivalent to the functor $\CC^\otimes \rightarrow \CO^\otimes$ being a locally cocartesian fibration.
\end{definition}

We are ready to prove the formula for the multimapping spaces in the Day convolution. 

\begin{lemma}\label{lem:mapping-spaces-Day-convolution}
    Let $\CO^\otimes$ be an $\infty$-operad, $\CC^\otimes$ be an $\CO^\otimes$-promonoidal $\infty$-category and $\CD^\otimes$ be an $\infty$-operad over $\CO^\otimes$. Then the multimapping spaces in $\Fun_\CO(\CC^\otimes,\CD^\otimes)^{Day}$ are given by the following natural equivalence
    \[\Mul_{\Fun_{\CO}(\CC^\otimes,\CD^\otimes)^{Day}}^\phi(\{F_i\}, G)\simeq \int_{c'\in\CC_{y}}\int_{\{c_i\}\in (\prod_i\CC_{x_i})^{\op}} \Map\left(\Mul_\CC^\phi(\{c_i\},c'),\Mul_{\CD}^\phi(\{F_ic_i\},Gc')\right)\]
    for all active morphisms $\phi\colon \{x_i\}\to y$, and objects $\{F_i\}\in\prod_i\Fun(\CC_{x_i},\CD_{x_i})$, $G\in \Fun(\CC_{y},\CD_{y})$.
\end{lemma}
\begin{proof}
    We will use \cite[Proposition~2.2.6.6]{HA}. However the cited result has the hypothesis that $\CC^\otimes$ is a $\CO^\otimes$-monoidal $\infty$-category. We note that this is only used to ensure the existence of the norm (after replacing the appeal to \cite[Proposition~3.3.1.3]{HTT} with \cite[Proposition~B.3.14]{HA}). Therefore, in view of \cite[Theorem-Construction~10.6]{Exp2} we can safely apply this result when $\CC^\otimes$ is only $\CO^\otimes$-promonoidal.
    
    Then, arguing as in the proof of \cite[Proposition~2.2.6.11]{HA}, we obtain an equivalence
    \[\Mul_{\Fun_{\CO}(\CC^\otimes,\CD^\otimes)^{Day}}^\phi(\{F_i\},G)\simeq \{(F,G)\}\times_{\Fun_{/\CO^\otimes}(\partial\Delta^1\times_{\CO^\otimes}\CC^\otimes,\CD^\otimes)} \Fun_{/\CO^\otimes}(\Delta^1\times_{\CO^\otimes}\CC^\otimes,\CD^\otimes)\]
    where $\Delta^1\to \CO^\otimes$ picks out the active arrow $\phi$ and $F\colon\CC^\otimes_{\{x_i\}}\to \CD^\otimes_{\{x_i\}}$ is the functor sending $\{c_i\}$ to $\{F_ic_i\}$. Let $\CC^{act}\coloneqq\CC^\otimes\times_{\CO^\otimes}\CO^{act}$ and $\CD^{act}\coloneqq\CD^\otimes\times_{\CO^\otimes}\CO^{act}$ be the subcategories of active arrows. Since $\Delta^1\to \CO^\otimes$ factors through $\CO^{act}$, we have an equivalence
    \[\begin{split}
            \Mul_{\Fun(\CC^\otimes,\CD^\otimes)^{Day}}(\{F_i\}_{i\in I},G) &\simeq \{(F,G)\}\times_{\Fun_{/\CO^{act}}(\partial\Delta^1\times_{\CO^{act}}\CC^{act},\CD^{act})} \Fun_{/\CO^{act}}(\Delta^1\times_{\Fin}\CC^{act},\CD^{act})\\
            &\simeq\Map_{(p^{act})_*(p^{act})^*\CD^{act}}(F,G)
    \end{split}\]
    where the last equality makes sense since $p^{act}$ is an exponentiable fibration. Therefore the thesis follows from Lemma~\ref{lem:mapping-spaces-pushforward}.
\end{proof}

\begin{definition}\label{definition:monoidal-structure-compatible-with-colimits}
	We say that an $\CO^\otimes$-monoidal $\infty$-category $\CD^\otimes\rightarrow \CO^\otimes$ is \emph{compatible with colimits} if for every object $x\in \CO$ the fibre $\CD_{x}$ has all small colimits, and if for every active arrow $\phi$, the $\phi$-tensor product commutes with all small colimits separately in each variable (see \cite[Definition~3.1.1.18]{HA} for a more precise formulation). If each fibre is moreover presentable, then we say $\CD^\otimes$ is a presentably $\CO^\otimes$-monoidal $\infty$-category. 
\end{definition}

\begin{example}
 The underlying $\infty$-category of a symmetric monoidal model category is compatible with 
 colimits as the tensor product is a left Quillen bifunctor by the pushout-product axiom. 
\end{example}

\begin{remark}
Recall that every cocomplete $\infty$-category $\CC$ is canonically tensored over $\Spc$. Namely for every $X\in\Spc$ and $C\in \CC$, we define $X\times C$ to equal $\colim (\mathrm{const_C}\colon X \rightarrow \CC)$, the colimit over $X$ of the constant functor at $C$.
\end{remark}

\begin{corollary}\label{cor:tensor-in-Day-convolution}
    Fix an $\infty$-operad $\CO^\otimes$. Let $\CC^\otimes$ be a small $\CO^\otimes$-promonoidal $\infty$-category and let $\CD^\otimes$ be a $\CO^\otimes$-monoidal $\infty$-category which is compatible with colimits. 
    \begin{itemize}
    \item[(a)] Then $\Fun_\CO(\CC^\otimes,\CD^\otimes)^{Day}$ is an $\CO^\otimes$-monoidal $\infty$-category which is again compatible with colimits.
    \end{itemize}
    Suppose furthermore that $\CO^\otimes\simeq\Finp$ is the commutative $\infty$-operad. 
    \begin{itemize}
    \item[(b)] The unit of $\Fun(\CC^\otimes,\CD^\otimes)^{Day}$ is given by $1_{Day}\coloneqq\Mul_\CD(\varnothing,-)\times 1_\CD$, and the tensor product is given by
\[(F\otimes^{Day} G)(-)\simeq \int^{(c_1,c_2)\in\CC^2} \Mul_\CC(\{c_1,c_2\},-)\times (F(c_1)\otimes G(c_2))\,.\]
In particular, when $\CD$ is the $\infty$-category of spaces with the cartesian symmetric monoidal structure, we have 
\[\Map_\CC(x,-)\otimes^{Day}\Map_\CC(y,-)\simeq \Mul_\CC(\{x,y\},-)\]
for every $x,y\in\CC$.
\end{itemize}
\end{corollary}

\begin{proof}
    If $\CD^\otimes$ is $\CO^\otimes$-monoidal, it follows from the formula of Lemma~\ref{lem:mapping-spaces-Day-convolution} that $\Fun(\CC^\otimes,\CD^\otimes)^{Day}$ is representable and that the $\phi$-tensor product is given by
    \[\bigotimes_\phi \{F_i\}_{i\in I}\simeq \int^{\{c_i\}\in \prod_{i\in I} \CC_{o_i}} \Mul^\phi_{\CC^\otimes}(\{c_i\}_{i\in I},-)\times \bigotimes_\phi \{F_i(c_i)\}_{i\in I}\,.\]
    This shows the existence of locally cartesian edges in $\Fun(\CC^\otimes,\CD^\otimes)^{Day}$. Because the tensor product functors in $\CD^\otimes$ commutes with colimits in each variable, one can calculate that the composite of locally cartesian edges is locally cartesian, and therefore $\Fun_{\CO}(\CC^\otimes,\CD^\otimes)^{Day}$ is a $\CO^\otimes$-monoidal $\infty$-category. The fibres are clearly cocomplete, and from the formula for the tensor product it follows that the tensor in $\Fun(\CC,\CD)^\otimes$ commutes with colimits in each variable.
    
    Finally the statement for the tensor product of corepresentable functors follows from the formula above and the Yoneda lemma.
\end{proof}

\begin{notation}
Suppose we are in the situation of the previous corollary, and suppose that $\CO^\otimes\simeq \Finp$. In the case that both $\CC^\otimes$ and $\CD^\otimes$ are canonically (pro)monoidal, then we write $\CC{-}\CD$ for the symmetric monoidal category given by the $\infty$-operad $\Fun(\CC^\otimes,\CD^\otimes)^{Day}$. The two examples which will arise constantly are $\DgSpc{\CC}$ and $\DgSpcp{\CC}$, where $\Spc$ is symmetric monoidal via the cartesian product, and $\Spcp$ via the smash product. Nevertheless when we refer to the $\infty$-operad inducing the symmetric monoidal structure on $\CC{-}\CD$, we will continue to write $\Fun(\CC^\otimes,\CD^\otimes)^{Day}$. While this distinction is mathematically meaningless, we find it notationally convenient.
\end{notation}

We next turn to the functoriality of Day convolution.

\begin{construction}
	Let $\CO^\otimes$ be an $\infty$-operad and suppose $f\colon\CI^\otimes\to \CJ^\otimes$ is a map of $\CO^\otimes$-promonoidal $\infty$-categories. Then for every two $\infty$-operads $\CC^\otimes$ and $\CP^\otimes$ over $\CO^\otimes$ we have a natural transformation
    \[\Alg_{\CP^\otimes/\CO^\otimes}(\Fun_{\CO^\otimes}(\CJ^\otimes,\CC^\otimes)^{Day})\simeq \Alg_{\CP^\otimes\times_{\CO^\otimes}\CJ^\otimes}(\CC^\otimes)\to \Alg_{\CP^\otimes\times_{\CO^\otimes}\CI^\otimes}(\CC^\otimes)\simeq \Alg_{\CP^\otimes/\CO^\otimes}(\Fun_{\CO^\otimes}(\CI^\otimes,\CC^\otimes)^{Day})\,,\]
    given by precomposition along $\CP^\otimes\times_{\CO^\otimes}\CI^\otimes\to \CP^\otimes\times_{\CO^\otimes}\CJ^\otimes$. Since this is natural in $\CP^\otimes$, it induces a map in $(\Op_\infty)_{/\CO^\otimes}$
    \[f^*\colon \Fun_{\CO^\otimes}(\CJ^\otimes,\CC^\otimes)^{Day} \to \Fun_{\CO^\otimes}(\CI^\otimes,\CC^\otimes)^{Day}\,.\]
\end{construction}

\begin{definition}\label{definition:operadic-adjunction}
    Consider $\CC^\otimes,\CD^\otimes\in(\Op_\infty)_{/\CO^\otimes}$. An \emph{operadic adjunction} between $\CC^\otimes$ and $\CD^\otimes$ is a relative adjunction over $\CO^\otimes$ in the sense of~\cite[Definition~7.3.2.2]{HA} such that both functors are maps of $\infty$-operads. This notion is equivalent to an adjunction in the $(\infty,2)$-category of $\infty$-operads, see ~\cite[Observation~4.3.2]{Riehl-Verity-adjunctions}.
\end{definition}

\begin{remark}
Note that if $\CC^\otimes$ and $\CD^\otimes$ are both $\CO^\otimes$-monoidal then an operadic left adjoint $f\colon \CC^\otimes\rightarrow \CD^\otimes$ is automatically $\CO^\otimes$-monoidal by \cite[Proposition 7.3.2.6]{HA}.
\end{remark}

\begin{proposition}\label{proposition:functoriality-of-presheaves}
    Let $\CO^\otimes$ be an $\infty$-operad and let $f:\CI^\otimes\to \CJ^\otimes$ a map of $\CO^\otimes$-promonoidal $\infty$-categories. Suppose $\CC^\otimes$ is a presentably $\CO^\otimes$-monoidal $\infty$-category. Let us consider the lax $\CO^\otimes$-monoidal functor
    \[f^*\colon \Fun_{\CO^\otimes}(\CJ^\otimes,\CC^\otimes)^{Day} \to \Fun_{\CO^\otimes}(\CI^\otimes,\CC^\otimes)^{Day}\,.\]
    \begin{itemize}
        \item[(a)] Suppose that for every active arrow $\phi:\{t_i\}_i\to t$ in $\CO^\otimes$ the natural map
        \[(f_t)_!\Mul_\CI^\phi(\{x_i\}_i,-)\to \Mul_\CJ^\phi(\{f_{t_i}x_i\}_i,-)\]
        adjoint to 
        \[\Mul_\CI^\phi(\{x_i\}_i,-)\to \Mul_\CJ^\phi(\{f_{t_i}x_i\}_i,f_t(-))\]
        is an equivalence for every family of objects $\{x_i\}_i$. Then $f^*$ has a left operadic adjoint $f_!$ that is $\CO^\otimes$-monoidal;
        \item[(b)] Suppose $f$ has an operadic right adjoint $g\colon \CJ^\otimes\to \CI^\otimes$. Then there is a natural equivalence of maps of $\infty$-operads $f_!\simeq g^*$, and moreover this functor is $\CO^\otimes$-monoidal.
    \end{itemize}
\end{proposition}
\begin{proof}
    We will use \cite[Proposition~7.3.2.11]{HA} applied to the functor $f^*$ over $\CO^\otimes$. Since on the fibre over $t_i\in\CO$ this is just given by precomposition by $f_{t_i}$, the functor on the fibre over $\{t_i\}_i$ 
    \[\prod_i \Fun(\CJ_{t_i},\CC_{t_i})\to \prod_i \Fun(\CI_{t_i},\CC_{t_i})\]
    has a left adjoint, given by the left Kan extension $(f_{t_i})_!$ on every component. In particular, this collection of left adjoints commutes with the pushforwards along inert maps. So it suffices to show that this collection of left adjoints commute with the pushforwards along active maps. Let $\phi:(t_i)_i\to t$ be an active map. Then we need to show that the map
    \[(f_t)_!\left(\bigotimes_i^\phi F_i\right)\to \bigotimes_i^\phi (f_{t_i})_!F_i\]
    is an equivalence.
    But then this follows from our hypothesis together with the description of Corollary~\ref{cor:tensor-in-Day-convolution}.
    
    Suppose now that $f$ has an operadic right adjoint $g$. Since $g^*$ is an operadic left adjoint to $f^*$, it follows immediately that $f_!=g^*$. So it remains only to check the two final conditions. But we have
    \[(f_t)_!\Mul_\CI^\phi(\{x_i\}_i,-)\simeq \Mul_{\CI}^\phi(\{x_i\}_i,g_t-)\simeq \Mul_\CJ^\phi(\{f_{t_i}x_i\}_i,-)\]
    since $g$ is an operadic right adjoint of $f$.
\end{proof}

\begin{remark}
Note that if $\CO^\otimes=\Finp$ and $\CI^\otimes$ and $\CJ^\otimes$ are both symmetric monoidal, then the conditions ensuring the symmetric monoidality of $f_!$ are equivalent to $f$ being a symmetric monoidal functor (since $f_!$ restricts to $f$ on representables). Thus the above proposition gives an alternative proof of \cite[Proposition~3.6]{ben-moshe-schlank}.
\end{remark}

\subsection{Symmetric monoidal structures on copresheaf categories}
 We finish this section by classifying all possible closed symmetric monoidal structures on the copresheaf $\infty$-category $\Fun(\CI,\Spc)$ in terms of symmetric promonoidal structures on $\CI$, see Theorem~\ref{thm-I-promonoidal}.

\begin{lemma}\label{lemma:sub-of-closed-is-promonoidal}
    Let $\CI$ be a small $\infty$-category and let us suppose that the presheaf category $\Fun(\CI,\Spc)$ is equipped with a symmetric monoidal structure $\Fun(\CI,\Spc)^\otimes$ which is compatible with colimits. Equip $\CI$ with the full suboperad structure $\CI^\otimes$ induced by the Yoneda embedding $\CI\subseteq \Fun(\CI,\Spc)^{\op}$. Then $\CI^\otimes$ is symmetric promonoidal.
\end{lemma}

\begin{proof}
    For brevity let us write $\CD^\otimes=\Fun(\CI,\Spc)^\otimes$. Recall from Definition \ref{def-promonoidal} that $\CI^\otimes$ is promonoidal if the functor $\CI^{\otimes}\to \Finp$ is exponentiable over $\Fin \simeq (\Finp)^{act}$. By the characterization of exponentiablility in \cite[Lemma~1.10.(c)]{Ayala-Francis}, we need to show that for every map $f\colon I\to J$ in $\Fin$, every $x\in \CI^I$ and every $z\in\CI$ the map
    \[\int^{y\in \CI^J} \Mul_{\CI}(\{y_j\}_{j\in J},z)\times \prod_{j\in J}\Mul_{\CI}(\{x_i\}_{i\in f^{-1}j},y_j)\to \Mul_{\CI}(\{x_i\}_{i\in I},z)\]
    is an equivalence. Using that $\CI\subseteq \CD^{\op}$ is a full suboperad, this is equivalent to asking that the map
    \[\int^{y\in \CI^J} \prod_{j\in J} \Map_{\CD}(y_j,\bigotimes_{i\in f^{-1}j} x_i)\times \Map_{\CD}(z,\bigotimes_{j\in J} y_j)\to \Map_{\CD}(z,\bigotimes_{i\in I}x_i)\]
    is an equivalence of spaces. But since $\Map_{\CD}(z,-)$ commutes with all colimits (as $z\in\CI$ is tiny) it is enough to show that the map
    \[\int^{y\in \CI^J} \left(\prod_{j\in J}\Map_{\CD}(y_j,\bigotimes_{i\in f^{-1}j} x_i)\right)\otimes \bigotimes_{j\in J} y_j\to \bigotimes_{i\in I}x_i\,,\]
    is an equivalence. Since the tensor product in $\CD$ commutes with colimits in each variable, we can bring all the colimits inside (using that $\Tw(\CI^J)\simeq \Tw(\CI)^J$). We are reduced to proving that the map
    \[\bigotimes_{j\in J} \int^{y_j\in \CC} \Map_{\CD}(y_j,\bigotimes_{i\in f^{-1}j} x_i)\otimes y_j \to \bigotimes_{i\in I} x_i\]
    is an equivalence. But this follows from the fact that for any $j\in J$ and $w\in \CD$, the map
    \[\int^{y_j\in\CC}\Map(y_j,w)\times y_j\simeq \colim_{y_j\in \CC_{/w}} y_j\to w\]
    is an equivalence, which is just another form of the Yoneda lemma.
\end{proof}

We are ready to prove our classification result.

\begin{theorem}\label{thm-I-promonoidal}
    Let $\CI$ be a small $\infty$-category and suppose $\Fun(\CI,\Spc)$ is equipped with a symmetric monoidal structure $\Fun(\CI,\Spc)^\otimes$ which is compatible with colimits. Equip $\CI^\otimes$ with the $\infty$-operad structure induced by the Yoneda embedding $\CI\subseteq\Fun(\CI,\Spc)^{\op}$. Then $\CI^\otimes$ is symmetric promonoidal and the symmetric monoidal structure on $\Fun(\CI,\Spc)$ is equivalent to the one induced by Day convolution with the symmetric promonoidal structure on $\CI^\otimes$.
\end{theorem}
\begin{proof}
    It follows from Lemma~\ref{lemma:sub-of-closed-is-promonoidal} that $\CI^\otimes$ is symmetric promonoidal. Consider the composite 
    \[\CI^\otimes\times_{\Finp}\Fun(\CI,\Spc)^\otimes\to (\Fun(\CI,\Spc)^{\op})^\otimes\times_{\Finp}\Fun(\CI,\Spc)^\otimes\to \Spc^\times\]
    of lax symmetric monoidal functors, where the first functor is induced by the Yoneda embedding and the second is the lax symmetric monoidal enhancement of the mapping space functor constructed in \cite[Section~3]{GlasmanDay}. By the universal property of the Day convolution, we obtain a map of $\infty$-operads
    \[\Fun(\CI,\Spc)^\otimes\to \Fun(\CI^\otimes ,\Spc^\times)^{\mathrm{Day}}\]
    which is the identity on underlying $\infty$-categories. Therefore to prove our thesis it will suffice to show that this functor is symmetric monoidal. Since $\Fun(\CI,\Spc)$ is generated under colimits by the corepresentable functors and both tensor products commute with colimits in each variable, it is enough to check that the maps
    \[\Mul_\CI(\varnothing,-)\simeq 1\to 1^{\mathrm{Day}}\qquad \Mul_\CI(\{x,y\},-)\simeq \Map_\CI(x,-)\otimes \Map_\CI(y,-)\to \Map_\CI(x,-)\otimes^{\mathrm{Day}} \Map_\CI(y,-)\]
    are equivalences for all $x,y\in \CI$. But this follows from Corollary~\ref{cor:tensor-in-Day-convolution}.
\end{proof}

Recall that the $\infty$-category of pointed objects in a presentably symmetric monoidal $\infty$-category is canonically symmetric monoidal. For later use we also record how taking pointed objects in a category of diagram spaces interacts with the Day convolution symmetric monoidal structure.

\begin{proposition}\label{prop:Day_conv_pointed}
Consider a small promonoidal $\infty$-category $\CI$, and a presentably symmetric monoidal $\infty$-category $\CC$. There exists a symmetric monoidal equivalence
\[(\CI{-}\CC)_*  \simeq \CI{-}\CC_*\]
\end{proposition}

\begin{proof}
Consider the lax monoidal functor $\CI{-}\CC \rightarrow \CI{-}\CC_*$ induced by the universal property of Day convolution by the composite
\[\Fun(\CI^\otimes,\CC^\otimes)\times_{\Finp} \CI^\otimes \rightarrow \CC^\otimes \xrightarrow{({-})_+} (\CC_*)^{\wedge_\otimes}.\] Because $({-})_+$ is strong monoidal and colimit preserving, one calculates that this functor is in fact strong monoidal. Therefore by \cite{HA}*{Proposition 4.8.2.11} we obtain an induced strong monoidal functor $(\CI{-}\CC)_* \rightarrow \CI{-}\CC_*$, which is easily seen to be the identity on underlying categories.
\end{proof}

\subsection{A symmetric monoidal Elmendorf's theorem}

In this subsection we give a general $\infty$-categorical version of Elmendorf's theorem. We then enhance this to a symmetric monoidal equivalence.

\begin{theorem}[Elmendorf]\label{thm-elmendorf}
	Let $\CC$ be a cocomplete $\infty$-category and let $i\colon \CC_0\to \CC$ be the 
	inclusion of a small full subcategory satisfying the following conditions:
	\begin{itemize}
		\item[(a)] The objects of $\CC_0$ are tiny: for all $c\in\CC_0$, the functor 
		$\Map_{\CC}(c,-)$ preserves small colimits;
		\item[(b)] The collection of objects $\{c_0 \in \CC_0\}$ is jointly conservative: 
		an arrow $f$ in $\CC$ is an equivalence if and only if $\Map_{\CC}(c_0,f)$ is so 
		for all $c_0\in\CC_0$. 
	\end{itemize}
	Then the restricted Yoneda functor induces an equivalence of $\infty$-categories 
	$j\colon \CC \simeq \CP(\CC_0)$.
\end{theorem}

\begin{proof}
	By the universal property of the category of presheaves~\cite{HTT}*{Theorem 5.1.5.6}, there 
	exists a colimit preserving functor $L\colon \CP(\CC_0)\to \CC$ such that 
	$Lj_0 \simeq i$ where $j_0 \colon \CC_0 \to \CP(\CC_0)$ denotes the Yoneda embedding. 
	By the adjoint functor theorem~\cite{AFT}*{Corollary 4.1.4}, the functor $L$ admits a right 
	adjoint $R\colon \CC \to \CP(\CC_0)$ which is defined via the formula 
	\[
	R_c(c_0)=\Map_{\CC}(Lj_0(c_0), c)\simeq\Map_{\CC}(i(c_0), c)
	\]
	for all $c\in\CC$ and $c_0 \in \CC_0$.
	Therefore $R$ can be identified with the restricted Yoneda functor 
	$j\colon \CC \to \CP(\CC_0)$. We note that the functor $j$ 
	preserves all small colimits since for all $c_0 \in \CC_0$, the functor 
	\[
	\Map_{\CC}(c_0,-)\colon \CC\xrightarrow{j}\CP(\CC_0) 
	\xrightarrow{\mathrm{ev_{c_0}}} \CC_0
	\]
	does so by condition (a). 
	As equivalences in $\CP(\CC_0)$ are detected pointwise, the same argument 
	as above using condition (b) then shows that $j$ is conservative. 
	Note that the unit map $\eta \colon 1 \to j L$ is an equivalence on all objects 
	in the image of $j_0$ as by construction $jLj_0\simeq ji=j_0$. 
	It follows that the unit map is an equivalence on all objects as $\CP(\CC_0)$  
	is generated under colimits by the representable functors and all the functors involved 
	preserve colimits. 
	Using the triangle identities of the adjunction we then find that $j(\epsilon)$ is 
	an equivalence and so the counit map $\epsilon \colon L j\to 1$ is an 
	equivalence by conservativity of $j$. Thus $j$ and $L$ are inverse equivalences. 
\end{proof}

\begin{example}\label{ex-S_G}
	Let $G$ be a topological group and let $G\Top$ be a convenient category of $G$-spaces. 
	There is a model structure on $G\Top$ where a map $f\colon X \to Y$ of $G$-spaces is a 
	weak equivalence (resp., fibration) if $f^H\colon X^H \to Y^H$ is a weak homotopy 
	equivalence (resp., Serre fibration) for all closed subgroups $H\leq G$, 
	see~\cite{Schwede18}*{Proposition B.7}. Let $\Spc_G$ denote the underlying $\infty$-category of this 
	model category, which is cocomplete by~\cite{BHH2017}*{Theorem 2.5.9}. 
	Moreover colimits in $\Spc_G$ of projective cofibrant diagrams can be calculated as 
	homotopy colimits in $G\Top$ by ~\cite{BHH2017}*{Remark 2.5.7}. 
	Let $\OO_G\leq \Spc_G$ be the full subcategory of $G$-spaces spanned by the cosets 
	$G/H$ where $H$ runs over all closed subgroups of $G$. 
	Note that $G/H\in \Spc_G$ corepresents the $H$-fixed points functors so the collection 
	of cosets $\{G/H \mid H \leq G\}$ is jointly conservative by definition of weak
	equivalences in $G\Top$. The fact that $G/H \in \Spc_G$ is tiny then follows 
	from the fact that the $H$-fixed points functor commutes with all small homotopy 
	colimits~\cite{Schwede18}*{Proposition B.1, (i) and (ii)}. 
	Then the theorem above gives an equivalence 
	$\DgSpc{\OO_G^{\op}}\simeq \Spc_G$. Therefore the previous theorem is a generalization of the classical Elmendorf's theorem ~\cite{Elmendorf}. 
\end{example}

Under suitable assumptions we now enhance this to a symmetric monoidal equivalence, where we endow the presheaf category with Day convolution for a promonoidal structure on subcategory of tiny objects.

\begin{corollary}\label{cor-elmendorf}
 Suppose we are in the setting of Theorem~\ref{thm-elmendorf} and that furthermore the 
 following holds:
 \begin{itemize} 
 \item[(a)] $\CC$ admits a symmetric monoidal structure $\CC^\otimes$ which is compatible 
 with colimits;
 \item[(b)] $\CC_0$ admits an $\infty$-operad structure $\CC_0^\otimes$;
 \item[(c)] $i \colon \CC_0 \to \CC$ lifts to a fully faithful functor of $\infty$-operads $i^\otimes \colon \CC_0^\otimes \to \CC^\otimes$. 
 \end{itemize}
 Then $\CC_0^\otimes$ is a symmetric promonoidal $\infty$-category and the restricted Yoneda embedding induces a symmetric monoidal equivalence $\CP(\CC_0)^{Day}\simeq \CC^\otimes$.
\end{corollary}

\begin{proof}
 By Theorem~\ref{thm-elmendorf} there is a commutative diagram 
 \[
 \begin{tikzcd}
  \CC_0 \arrow[r,"i"] \arrow[d,"j_0"'] & \CC \arrow[dl,"j","\sim"'] \\
  \CP(\CC_0)   &.
 \end{tikzcd}
 \]
 We can equip $\CP(\CC_0)$ with a symmetric monoidal structure 
 $\CP(\CC_0)^\otimes$ induced by $\CC^\otimes$ via $j$, and hence obtain a symmetric 
 monoidal equivalence $j^\otimes \colon \CC^\otimes \to \CP(\CC_0)^\otimes$. Combining this  with condition (c) we obtain another commutative diagram
 \[
 \begin{tikzcd}
  \CC_0^\otimes \arrow[r,"i^\otimes"] \arrow[d,"j^\otimes_0"'] & \CC^\otimes 
  \arrow[dl,"j^\otimes","\sim"'] \\
  \CP(\CC_0)^\otimes  &
 \end{tikzcd}
 \]
 of $\infty$-operads. It is only left to note that by Theorem~\ref{thm-I-promonoidal}, the $\infty$-category $\CC_0^\otimes$ is symmetric promonoidal and that the symmetric monoidal structure on $\CP(\CC_0)^\otimes$ coincides with the Day convolution product.
\end{proof}

\section{Partially lax limits}\label{sec-partially-lim}

In this section we recall the necessary background on (partially) lax (co)limits and collect some important properties that we will use throughout the paper. The main references for this material are~\cites{GHN, Berman}.

The notion of a partially lax limit over an $\infty$-category $\I$ is defined with reference to a collection of edges of $\I$. To make this precise we make the following definition.

\begin{definition}
	A marked $\infty$-category is an $\infty$-category $\CC$ along with a collection 
	of edges $\CW\subseteq \Map(\Delta^1, \CC)$ which contains 
	all equivalences and which is stable under homotopy and composition. 
	Given two marked $\infty$-categories $\CC$ and $\CD$, we write 
	$\Fun^{\dagger}(\CC, \CD)$ for the subcategory spanned by marked functors; 
	those functors that preserve marked edges. 
	We write $\Cat_\infty^{\dagger}$ for the 
	$\infty$-category of marked $\infty$-categories. 
	For the existence see~\cite{HA}*{Construction 4.1.7.1}. 
\end{definition}

\begin{example}\label{ex-marking}
	Let $\CC$ be an $\infty$-category.
	\begin{itemize}
		\item[(a)] There is a maximal marking $\CC^{\sharp}$ where all morphisms 
		are marked;
		\item[(b)] There is a minimal marking $\CC^{\flat}$ where only the 
		equivalences are marked;
		\item[(c)] Given a (co)cartesian fibration $p\colon \CC \to \I^\dagger$ over a marked $\infty$-category, there is a  marking $\CC^{p}$ in which the (co)cartesian morphisms living over marked edges are marked.
	\end{itemize}
\end{example}

Partially lax limits in an $\infty$-category $\CC$ are also defined with reference to a cotensoring of $\CC$ by $\Cat_\infty$. For the purposes of this paper, this is nothing but a functor $[{-},{-}]\colon \Cat_{\infty}^{\op}\times\CC\rightarrow \CC$. The following examples are all naturally cotensored over $\Cat_\infty$.

\begin{example}\label{ex-cotensor}
	In the following $\CI$ is an $\infty$-category.
	\begin{itemize}
		\item[(a)] Clearly $\Cat_\infty$ is cotensored over itself with cotensor given by 
		$[\CI,\CC] = \Fun(\CI,\CC)$.
		\item[(b)] The $\infty$-category $\Cat_\infty^{\dagger}$ is cotensored over 
		$\Cat_\infty$ by considering $[\CI,\CC^\dagger] = \Fun(\CI,\CC^\dagger)$, where we mark all those natural 
		transformations whose components are all marked in $\CC^\dagger$.
		\item[(c)] The $\infty$-category of symmetric monoidal categories 
		$\Cat^\otimes_\infty$ is cotensored over 
		$\Cat_\infty$ by endowing the $\infty$-category $\Fun(\CI, \CC)$ with the 
		pointwise symmetric monoidal structure $q\colon \Fun(\CI, \CC)^\otimes\to \Finp$ 
		which is defined as follows. 
		If $p \colon \CC^\otimes \to \Finp$ is the cocartesian fibration witnessing 
		the symmetric monoidal structure of $\CC$, then we construct the following pullback 
		\[
		\begin{tikzcd}
			\Fun(\CI, \CC)^\otimes \arrow[r,"q"] \arrow[d] & \Finp \arrow[d, "\mathrm{const}"] \\
			\Fun(\CI, \CC^\otimes) \arrow[r, "p_*"] & \Fun(\CI,\Finp).
		\end{tikzcd}
		\]
		Note that by construction we have $\Fun(\CI,\CC)^\otimes_{\langle n \rangle}\simeq \Fun(\CI, \CC^\otimes_{\langle n \rangle})$ for all $\langle n\rangle \in \Finp$. 
		From this we immediately see that $q$ satisfies the Segal conditions. 
		The map $p_*$ is a cocartesian fibration by the dual of~\cite{HTT}*{Proposition 3.1.2.1} and so by 
		base-change~\cite{HTT}*{Proposition 2.4.2.3} $q$ is too. 
		Therefore $q$ gives a symmetric monoidal structure on $\Fun(\CI, \CC)$.
		\item[(d)] We can generalize the previous example as follows. 
		Let $\CO^\otimes \to \Finp$ be an $\infty$-operad. 
		The $\infty$-category of $\infty$-operads $\Op_\infty$ is cotensored over $\Cat_\infty$ by endowing the $\infty$-category $\Fun(\CI, \CO)$ with the pointwise operadic structure induced by the map $\Fun(\CI, \CO^\otimes) \times_{\Fun(\CI, \Finp)} \Finp \to \Finp$. 
	\end{itemize}  
\end{example}

Similarly partially lax colimits in $\CC$ are defined with reference to a tensoring of $\CC$ by $\Cat_\infty$. Once again, while more structured tensorings are typically useful, for our purposes it suffices for this to be a functor $({-})\otimes ({-})\colon \Cat_\infty\times \CC\rightarrow \CC$. The most important example will be $\Cat_\infty$, for which the cartesian product gives a tensoring.

We now move on to the definition of partially lax (co)limits. For this we need to recall some categorical constructions. Recall the following result.

\begin{lemma}[\cite{HA}*{Proposition 4.1.7.2}]
	The minimal functor $(-)^{\flat}\colon \Cat_{\infty} \to \Cat_{\infty}^\dagger$ 
	admits a left adjoint denoted by $|-|$. 
\end{lemma}

The $\infty$-category $|\CC^{\dagger}|$ is obtained from $\CC$ by adjoining formal inverses to all the marked morphisms, and so we call $|-|$ the localization functor. 

\begin{example}
Given a model category $\mathcal{M}$, we may view it as a marked $\infty$-category by marking the weak equivalences in $\mathcal{M}$. Then $|\mathcal{M}|\simeq \mathcal{M}[W^{-1}]$.
\end{example} 

Next we define marked slice categories.

\begin{construction}
	Let $\CC$ be an $\infty$-category. There is a functor 
	$\CC_{/-} \colon \CC \to \Cat_\infty$ sending $x\in \CC$ to the slice 
	category $\CC_{/x}$. 
	This is obtained by straightening the cocartesian fibration given by the 
	target map $t\colon \Ar(\CC) \coloneqq \Fun(\Delta^1,\CC) \to \CC$. One checks that a diagram 
	\[
	\begin{tikzcd}
		f_0 \arrow[r]\arrow[d,"f"'] & g_0 \arrow[d,"g"] \\ 
		f_1 \arrow[r] & g_1 
	\end{tikzcd}
	\]
	is a $t$-cocartesian edge if the top horizontal arrow is an equivalence. 
	If $\CC^{\dagger}$ is marked, then $\CC^{\dagger}_{/x}$ has an induced 
	marking where a morphism is marked if its image under the forgetful 
	functor $\CC^{\dagger}_{/x} \to \CC^{\dagger}$ is a marked morphism. 
	It is easy to see that this construction is functorial on $x$, and so we obtain a functor $\CC_{/-}^{\dagger} \colon \CC \to \Cat^\dagger_\infty$. 
\end{construction}

We are finally ready to introduce the notion of partially lax (co)limit. Recall the definition of the twisted arrow $\infty$-category from Definition~\ref{def-twisted-arrow}.

\begin{definition}\label{def-lax-lim}
	Consider a functor $F \colon \CI \to \CC$ and choose a marking 
	$\CI^\dagger$. 
	\begin{itemize}
		\item[(a)] If $\CC$ is cotensored over $\Cat_\infty$, then the 
		partially lax limit of $F$ is the limit of the composite
		\[
		\Tw(\CI)^{\op} \xrightarrow{(s,t)^{\op}} \CI^{\op} \times \CI 
		\xrightarrow{|\CI^\dagger_{/-}|\times F} \Cat^{\op}_\infty \times 
		\CC
		\xrightarrow{[-,-]} \CC.
		\]
		We abbreviate this by $\laxlimdag F$. 
		\item[(b)] If $\CC$ is tensored over $\Cat_\infty$, then the 
		partially lax colimit of $F$ is the colimit of the composite
		\[
		\Tw(\CI) \xrightarrow{(s,t)} \CI \times \CI^{\op} 
		\xrightarrow{F \times |(\CI^{\op})_{/-}^\dagger|} \CC \times\Cat_\infty 
		\xrightarrow{-\otimes -} \CC.
		\]
		We abbreviate this by $\laxcolimdag F$. 
	\end{itemize} 
\end{definition}

\begin{remark}
	If we choose the minimal marking $\CI^\flat$, then we recover the notion 
	of lax (co)limit of~\cite{GHN}. If we choose the maximal marking 
	$\CI^{\sharp}$, then we recover the usual notion of (co)limit, see~\cite{Berman}*{Proposition 3.6}. 
\end{remark}

In some cases we have a concrete description of the partial lax (co)limit.

\begin{theorem}[\cite{Berman}*{Theorem 4.4}]\label{thm-lax-limit-section}
	Let $\CI^\dagger$ be a small marked $\infty$-category and let 
	$F \colon \CI \to \Cat_\infty$ be a functor. Consider the source of the (co)cartesian fibrations $\Unct{F}\rightarrow \CI^{\op}$ and $\Unco{F}\rightarrow \CI$ as marked via
	Example~\ref{ex-marking}(c).  
	\begin{itemize}
		\item[a)] The partially lax limit of $F$ is the $\infty$-category of 
		marked sections of $p\colon \Unco{F} \to \CI^\dagger$.
		In other words 
		we have
		\[
		\laxlimdag F \simeq \Fun^\dagger_{/I^\dagger}(\CI^\dagger,  \Unco{F})
		\]
		\item[(b)] The partially lax colimit of $F$ is given by the localization of $ \Unct{F}$ at the 
		marked edges. In other words we have 
		\[
		\laxcolimdag F = \vert \Unct{F} \vert.
		\]
	\end{itemize}
\end{theorem}

\begin{remark}\label{rem:lax_limit_intuition}
	The previous result gives a more explicit description of the partially 
	lax limit of $F$. Recall that informally the Grothendieck 
	construction $\Unco{F}$ is the $\infty$-category whose objects are pairs 
	$(X, i)$ where $i\in \CI$ and
	$X \in F(i)$. A morphism from $(X, i)$ to $(Y, j)$ 
	is a pair $(\varphi,f)$ where $f \colon i \to j$ is a morphism in $\CI$ and 
	$\varphi \colon F(f)(X) \to Y$ is a morphism in $F(j)$. 
	Then the previous result informally implies that $\laxlimdag F$ is equivalent to the 
	$\infty$-category whose objects are coherent collections of objects
	$(X_i \in F(i))_{i \in \CI}$ together with 
	maps $\varphi_f\colon F(f)(X_i) \to X_j$ for every arrow 
	$f\colon i \to j$ in $\CI$, such that the map $\varphi_{f}$ is an equivalence whenever $f$ is marked. 
\end{remark}

We record some useful properties of partially lax (co)limits.

\begin{proposition}\label{prop:funpartlim}
	Let $\CI^\dagger$ be a marked $\infty$-category and let $F\colon \CI \to \Cat_\infty$ be a functor. Given any other $\infty$-category $\CC$, we have an equivalence 
	\[
	\Fun(\laxcolimdag_{\CI} F , \CC) \simeq \laxlimdag_{\CI^{\op}} \Fun(F(-),\CC).
	\]
\end{proposition}

\begin{proof}
	The partially lax colimit of $F \colon \CI \to \Cat_\infty$ is by definition calculated via the formula
	\[
	\laxcolimdag F=\colim\limits_{\Tw(\CI)} \; F \times |(\CI^{\op})_{/-}^\dagger|
	\]
	Postcomposing by the limit preserving functor $\Fun(-, \CC) \colon \Cat_\infty^{\op} \to \Cat_\infty$, we deduce that the $\infty$-category $\Fun(\laxcolimdag F, \CC)$ is the limit of the diagram
	\begin{equation}\label{eq:lax-lim-Fun(-,C)}
		\Tw(\CI)^{\op}\xrightarrow{(s,t)^{\op}} \CI^{\op} \times \CI \xrightarrow{(F, |(\CI^{\op})_{/-}^\dagger|)^{\op}} \Cat_\infty^{\op}\times \Cat_\infty^{\op} \xrightarrow{-\times -} \Cat_\infty^{\op} \xrightarrow{\Fun(-, \CC)} \Cat_\infty
	\end{equation}
	By adjunction, we find that the composite of the final three functors is equivalent to \[\Fun(-,-)\circ ( |(\CI^{\op})_{/-}^\dagger|, \Fun(F(-), \CC)) \circ \sigma \colon\CI^{\op} \times \CI \rightarrow \Cat_\infty,\] where $\sigma$ is the symmetry isomorphism of the product. As indicated in Remark~\ref{rem-op-twisted-cat}, the following triangle commutes:
	\[\begin{tikzcd}
		{\Tw(\CI)^{\op}} && {\Tw(\CI^{\op})^{\op}} \\
		& {\CI^{\op} \times \CI}
		\arrow["{(s,t)^{\op}}"', from=1-1, to=2-2]
		\arrow["{(t,s)^{\op}}", from=1-3, to=2-2]
		\arrow["\sim", from=1-1, to=1-3]
	\end{tikzcd}\]
	These two observations allow us to rewrite Equation~(\ref{eq:lax-lim-Fun(-,C)}) and conclude that $\Fun(\laxcolimdag F, \CC)$ is the limit of the functor
	\[\Tw(\CI^{\op})^{\op}\xrightarrow{(s,t)^{\op}} \CI \times \CI^{\op} \xrightarrow{(|(\CI^{\op})_{/-}^\dagger|,\Fun(F(-), \CC))} \Cat_\infty^{\op}\times \Cat_\infty \xrightarrow{\Fun(-,-)} \Cat_\infty,\] which is exactly the definition of the partially lax limit of $\Fun(F(-), \CC)\colon \CI^{\op}\rightarrow \Cat_{\infty}$.
\end{proof}

We finish this section by discussing how (partially) lax limits interact with 
localizations. Later on we will use these results to pass from (partially) lax limits of 
prespectra to that of spectra.  

\begin{lemma}\label{lem:adjunction-to-the-lax-limit}
	Let $\CI$ be an  $\infty$-category and let $F\colon \CI\to \Cat_\infty$ be 
	a functor. Suppose that for every $i\in \CI$ we are given a reflexive subcategory 
	$Gi\subseteq Fi$ with left adjoint 
	$L_i\colon Fi\to Gi$. If for every arrow 
	$f\colon i\to j$ of $\CI$, the pushforward functor $f_*\colon  Fi\to Fj$ sends 
	$L_i$-equivalences to $L_j$-equivalences, then there is a functor 
	$G\colon \CI\to\Cat_\infty$ and a natural transformation 
	$L\colon F\Rightarrow G$ whose $i$-th component is given by 
	$L_i \colon Fi\to Gi$. Furthermore, the functor
	\[\laxlim_\CI L\colon \laxlim_\CI F\to \laxlim_\CI G\]
	has a fully faithful right adjoint.
\end{lemma}

\begin{proof}
	Let us consider the Grothendieck construction $\Unco{F}\to \CI$ of $F$. 
	This is the cocartesian fibration classified by $F$ under the straightening-unstraightening equivalence, so in particular the fibre over $i\in \CI$ can be canonically identified with $Fi$. Let $\CE\subseteq \Unco{F}$ be the full subcategory spanned by the objects of $Gi\subseteq \Unco{F}$ for all $i\in \CI$. 
	
	We claim that $\CE \to \CI$ is a cocartesian fibration whose cocartesian edges are those that can be factored in $\Unco{F}$ as a cocartesian edge of $\Unco{F}$ followed by a $L_i$-equivalence in the fibre over $i$. More precisely if $f\colon i \to j$ is an arrow of $\CI$ and $x\in Gi$, then the cocartesian lift of $f$ starting from $x$ is the composition $x\to f_*x\to L_j(f_*x)$ where the first arrow is the cocartesian lift of $f$ in $\Unco{F}$.
	
	Indeed for every $z\in Gj$, we have
	\[\Map^f_\CE(x,z)\simeq\Map_{Fj}(f_*x,z)\simeq\Map_{Gj}(L_j f_*x,z)\]
	and so those edges are locally cocartesian. Furthermore it is easy to see they are stable under composition (using the fact that $L$-equivalences are stable under pushforward), therefore they are cocartesian arrows by \cite{HTT}*{Lemma 2.4.2.7}.
	
	The inclusion $\iota\colon \CE\subseteq \Unco{F}$ has a relative left adjoint which is a map of cocartesian fibrations by \cite{HA}*{Proposition~7.3.2.11}. Therefore there is a functor $G\colon \CI\to \Cat_\infty$ and a natural transformation $L\colon F\Rightarrow G$ such that $\CE$ can be identified with $\Unco{G}$ in such a way that  the induced map $L\colon\Unco{F}\to \Unco{G}$ agrees with $L_i\colon Fi\to Gi$ on each fibre.
	
	Finally by Theorem~\ref{thm-lax-limit-section} the lax limit of $F$ and $G$ are computed by the $\infty$-categories of sections of the respective cocartesian fibrations and $\laxlim_\CI L$ is given by postcomposition with $L$. Therefore postcomposition with $\iota$ gives a fully faithful right adjoint to $\laxlim_\CI L$.
\end{proof}

\begin{lemma}\label{lem:adjunction-to-the-marked-lax-limit}
	Suppose we are in the situation of Lemma~\ref{lem:adjunction-to-the-lax-limit}, and suppose $\CI$ is equipped with a marking $\CI^\dagger$ such that for every marked edge $f\colon i\to j$ the pushforward functor $f_*\colon Fi\to Fj$ sends $Gi$ into $Gj$. Then the functor
	\[\laxlim_{\CI^\dagger} L\colon \laxlim_{\CI^\dagger} F\to \laxlim_{\CI^\dagger} G\]
	has a fully faithful right adjoint. In particular $\laxlim_{\CI^\dagger} L$ is a localization functor.
\end{lemma}

\begin{proof}
	It suffices to show that the right adjoint of Lemma~\ref{lem:adjunction-to-the-lax-limit} sends $\laxlim_{\CI^\dagger} G$ into $\laxlim_{\CI^\dagger}F$.  
	Recall that the partially lax limit can be calculated as the subcategory of sections spanned by those sending marked edges to cocartesian arrows. Thus, we ought to show that the right adjoint preserves cocartesian arrows lying over marked edges. But the right adjoint is given by postcomposing a section with the inclusion $\Unco{G}\to \Unco{F}$, and so by the description of cocartesian edges given in Lemma~\ref{lem:adjunction-to-the-lax-limit} and by our hypothesis, it sends cocartesian arrows over marked edges to cocartesian arrows (here we are implicitly using that an 
	$L_i$-equivalence between objects of $Gi$ is automatically an equivalence in $Fi$ and so in particular a cocartesian arrow).
\end{proof}

For later reference we record the following immediate corollary of Lemma~\ref{lem:adjunction-to-the-lax-limit}.
\begin{corollary}\label{cor:laxlim-monoidal}
	Let $\CI$ be an $\infty$-category and let $F\colon \CI\to \Cat_\infty$ be 
	a functor. Suppose that for every $i\in \CI$, we are given a reflexive subcategory 
	$Gi\subseteq Fi$ with left adjoint $L_i\colon Fi\to Gi$ which is compatible with 
	the symmetric monoidal structure in the sense of \cite{HA}*{Definition~2.2.1.6}. Suppose furthermore that  for every arrow $f\colon i\to j$ in $I$, the pushforward functor $f_*\colon  Fi\to Fj$ sends 
	$L_i$-equivalences to $L_j$-equivalences. Then there exists a functor 
	$G\colon \CI\to\Cat_\infty^\otimes$ and a symmetric monoidal natural 
	transformation $L\colon F\Rightarrow G$ whose $i$th component is given by 
	$L_i \colon Fi\to Gi$.
\end{corollary}

\begin{proof}
	Since $\Cat_\infty^\otimes$ embeds as a subcategory of $\Fun(\Finp,\Cat_\infty)$ we can consider the functor $\tilde F\colon \Finp\times \I\to \Cat_\infty$ induced by $F$, so that $\tilde F (A_+,i)\simeq (Fi)^A$ (the fibre over $A$ of $Fi \to \Fin_*$). If we let $\tilde G(A_+,i)=(Gi)^A\subseteq \tilde F(A_+,i)$, we can apply Lemma~\ref{lem:adjunction-to-the-lax-limit} to $\tilde F$. To see that the pushforwards respect local equivalences, it suffices to prove this separately for maps of the form $(\sigma,\mathrm{id})$ and $(\mathrm{id}, f)$ in $\Finp\times \CI$. However both of these cases are ensured by our assumptions. Therefore there exists a functor
	\[\tilde G\colon \Finp\times\mathcal{I} \to \Cat_\infty\]
	and a natural transformation $\tilde L\colon \tilde F\Rightarrow\tilde G$ as desired. By construction $\tilde G$ satisfies the Segal conditions, and so it induces a functor $G\colon \CI\to \Cat_\infty^\otimes$ with a symmetric monoidal natural transformation $L\colon F\Rightarrow G$ as desired.
\end{proof}

\section{Partially lax limits of symmetric monoidal \texorpdfstring{$\infty$}{infty}-categories}

Recall that $\Op_\infty$ is canonically cotensored over $\Cat_\infty$ by Example~\ref{ex-cotensor}. Therefore we immediately obtain a definition of partially lax limits of diagrams in $\Op_\infty$. In this section we will collect some important properties of partially lax limits of symmetric monoidal $\infty$-categories and $\infty$-operads. In particular the calculations of Proposition~\ref{proposition:laxlimit-is-norm}, and Theorem~\ref{thm:modules-in-laxlimit} are used repeatedly in part two. The first is analogous to the calculation of the (partially) lax limit of a diagram of $\infty$-categories, and as such it is stated in terms of an unstraightening equivalence for symmetric monoidal categories, which we recall in Proposition~\ref{prop:equiv_amalg}. 

\begin{remark}\label{rem-sym-mon-cats-closed-under-laxlimits}
If $\CP^\otimes$ is another $\infty$-operad, it follows from the definition and \cite[Remark~2.1.3.4]{HA} that there is a natural equivalence
\[\Alg_{\CP^\otimes}(\laxlim_{i\in I}\CO^\otimes_i)\simeq \laxlim_{i\in I}\Alg_{\CP^\otimes}(\CO^\otimes)\,.\]  
Such a natural equivalence then also uniquely determines the lax limit. 	
Since $\Cat^\otimes_\infty\subseteq\Op_\infty$ is a subcategory closed under limits and cotensoring, it is also closed under partially lax limits. In particular we conclude that for every family of symmetric monoidal $\infty$-categories $\CC_\bullet$ and every symmetric monoidal $\infty$-category $\CD$, there is a natural equivalence
\[\Fun^\otimes(\CD,\laxlim_{i\in I}\CC_i)\simeq \laxlim_{i\in I}\Fun^\otimes(\CD,\CC_i)\,.\]	
We note that the underlying $\infty$-category functor $U\colon \Op_\infty\rightarrow \Cat_\infty$ preserves limits and commutes with cotensoring, and therefore preserves partially lax limits. Therefore the previous construction equips the partially lax limit of a family of symmetric monoidal $\infty$-categories with a canonical symmetric monoidal structure, which satisfies the expected universal property.
\end{remark}

\begin{remark}\label{rem-partially-lax-underlying}
	Note that there is always a canonical map $\laxlimdag \CO_i^\otimes \rightarrow \laxlim \CO^\otimes_i$. This functor is induced on limits by a natural transformation which is pointwise given by the inclusion of a fully faithful suboperad. Therefore we conclude that the partially lax limit is always a fully faithful suboperad of the lax limit. In practice this means that we can determine which suboperad by considering the induced map on underlying categories. 
\end{remark}

In the second part of the paper we will build diagrams of symmetric monoidal $\infty$-categories indexed on $\Glo^{\op}$. Central to our constructions of these diagrams is an operadic variant of  straightening/unstraightening, which we will recall now. 

\begin{notation}\label{notation-constant}
Recall from \cite[2.4.3.5]{HA} that for every $\infty$-category $\CI$ there is a functor of $\infty$-operads $c\colon \CI\times\Finp\to \CI^\amalg$ sending $(x,A_+)$ to the constant family $\{x\}_{a\in A}\in \CI^\amalg_{A_+}$.
\end{notation}

\begin{construction}
	Let $\CI$ be an $\infty$-category and let $\CC^\otimes$ be an $\CI^\amalg$-monoidal $\infty$-category. Then the commutative diagram of cocartesian fibrations
	\[\begin{tikzcd}[column sep=1em]
		\CC^\otimes\times_{\CI^\amalg}(\CI\times\Finp)\ar[rr,"\mathrm{pr}_2"]\ar[rd,"\mathrm{pr}_{\CI}"']&& \CI\times\Finp\ar[ld,"\mathrm{pr}_1"]\\
		& \CI
	\end{tikzcd}\]
	is classified by a functor $\CC_\bullet:\CI\to (\Cat_\infty)_{/\Finp}$ which lands in $\Cat^\otimes_\infty$. We refer to $\CC_\bullet$ as the family of symmetric monoidal $\infty$-categories classifying $\CC^\otimes$.
\end{construction}

\begin{proposition}\label{prop:equiv_amalg}
	The previous construction furnishes an equivalence between the $\infty$-category of $\CI^\amalg$-monoidal categories and $\Fun(\CI,\Cat_{\infty}^\otimes)$.
\end{proposition}

\begin{proof}
	This is \cite{drew2021universal}*{Corollary A.12}.
\end{proof}	

\begin{definition}\label{definition:operadic-fiber}
	Consider a map of $\infty$-operads $p:\CO^\otimes\to \CI^\amalg$. Any object $i\in \CI$ induces a functor \[\{i\}\times \Finp\hookrightarrow \CI\times \Finp \xrightarrow{c} \CI^\amalg\,,\]
    see Notation~\ref{notation-constant}. Equivalently, the map above can be obtained by applying $(-)^\amalg$ to the map $\Delta^0 \to \CI$ defined by $i\in \CI$. Inspired by the equivalence of Proposition~\ref{prop:equiv_amalg} we will refer to the pullback $\CO^\otimes\times_{\CI^\amalg} \Finp$ as the \emph{operadic fibre} of $p$ at $i\in \CI$. If $p$ is an $\CI^\amalg$-monoidal $\infty$-category, then its operadic fibre at $i$ is a symmetric monoidal $\infty$-category, and corresponds to the value of the functor $\CC_\bullet$ at $i$.
\end{definition}

The following example will be crucial for later applications. 

\begin{example}\label{ex-operadic-fibre}
    Let $p\colon \CC^\otimes\to \CI^\amalg$ be a $\CI^\amalg$-promonoidal $\infty$-category and  let $\CD^\otimes\to \CI^\amalg$ be a map of $\infty$-operads which is compatible with colimits. Then the operadic fibre of the Day convolution $\Fun_\CI(\CC^\otimes,\CD^\otimes)^{Day}$ over $i\in \CI$ is given by the symmetric monoidal $\infty$-category $\CC_i{-}\CD_i$, where $\CC_i,\CD_i$ are the operadic fibres over $i$ of $\CC^\otimes$ and $\CD^\otimes$ respectively. To see this, first recall that $\CC_i{-}\CD_i$ is defined to be $\Fun(\CC_i,\CD_i)$ with the Day convolution symmetric monoidal structure. Then the claim follows from the following computation using Lemma~\ref{lemma:norm-and-pullbacks}
    \[(N_pp^*\CD^\otimes)\times_{\CI^\amalg}\Finp\simeq N_{p_i}(p^*\CD^\otimes \times_{\CC^\otimes} \CC_i^\otimes)\simeq N_{p_i}p_i^*\CD_i^\otimes= \CC_i{-}\CD_i.\] 
\end{example}

Recall that the lax limit of a diagram of $\infty$-categories was calculated by taking sections of the associated cocartesian fibration. Similarly, we can describe the lax limit of $\CC_\bullet$ in terms of (suitable) sections of the $\infty$-operad $\CC^\otimes$.

\begin{proposition}\label{proposition:laxlimit-is-norm}
		Let $\CC^\otimes\to \CI^\amalg$ be a $\CI^\amalg$-monoidal $\infty$-category, and write $\CC_\bullet\colon \CI\rightarrow \Cat_\infty^\otimes$ for the associated diagram of symmetric monoidal $\infty$-categories. Then there is a natural equivalence of symmetric monoidal $\infty$-categories
	\[\laxlim\CC_\bullet \simeq N_{\CI^\amalg}\CC^\otimes\]
	where the right hand side is the norm along $\CI^\amalg\to \Finp$, which is well defined by Example~\ref{example:cocartesian-is-promonoidal}.
\end{proposition}

\begin{proof}
	We will show that the right hand side has the universal property of the lax limit. By the universal property of the norm, for any $\infty$-operad $\CP^\otimes$ we have an equivalence 
	\[\Alg_{\CP^\otimes} (N_{\CI^\amalg}\CC^\otimes)\simeq \Alg_{\CP^\otimes\times_{\Finp} \CI^\amalg/ \CI^\amalg}(\CC^\otimes)\,.\]
	By \cite[Theorem~2.4.3.18]{HA}, we can write
	\begin{align*}
	\Alg_{\CP^\otimes\times_{\Finp} \CI^\amalg/ \CI^\amalg}(\CC^\otimes)& \simeq \Alg_{\CP^\otimes\times_{\Finp}\CI^\amalg}(\CC^\otimes)\times_{\Alg_{\CP^\otimes\times_{\Finp}\CI^\amalg}(\CI^\amalg)} \{\pr_2\}\\
	& \simeq \Fun(\CI,\Alg_{\CP^\otimes}(\CC^\otimes)) \times_{\Fun(\CI,\Alg_{\CP^\otimes}(\CI^\amalg))} \{\pr_2\}\,.
	\end{align*}
	where $\pr_2\colon \CP^\otimes\times_{\Finp}\CI^\amalg\to \CI^\amalg$ is the projection. In other words, we have shown that $\Alg_{\CP^\otimes}(N_{\CI^\amalg}\CC^\otimes)$ is the $\infty$-category of sections of the functor
	\[\Alg_{\CP^\otimes}(\CC^\otimes)\times_{\Alg_{\CP^\otimes}(\CI^\amalg)} \CI \to \CI\]
	which is exactly the cocartesian fibration classified by $i\mapsto \Alg_{\CP^\otimes}(\CC_i^\otimes)$. Our thesis then follows from Theorem~\ref{thm-lax-limit-section}.
\end{proof}

\begin{remark}\label{rem-norm-extend-sec}
	Let $p\colon \CC^\otimes\rightarrow \CI^\amalg$ be an $\CI^\amalg$-monoidal $\infty$-category, and write $\CC_\bullet\colon \CI\rightarrow \Cat_\infty^\otimes$ for the associated diagram of symmetric monoidal $\infty$-categories. Then by the discussion in Remark \ref{rem-norm-underlying}, the underlying category of $N_{\CI^\amalg} \CC^\otimes$ is given by $\Fun_{/\CI}(\CI,\CC)$. Therefore the proposition above is an operadic analogue of Theorem~\ref{thm-lax-limit-section}(b). Since we know that the partially lax limit of a diagram of $\infty$-operads is a fully faithful suboperad of the lax limit, the previous result also allows us to calculate the partially lax limit of $\CC_\bullet$. Namely it is the fully faithful symmetric monoidal subcategory of $N_{\CI^\amalg}\CC^\otimes$ determined by the fully faithful subcategory $\laxlimdag \CC_\bullet \subset \laxlim \CC_\bullet$. 
\end{remark}

We finish this section by proving that the formation of (partially) lax limits of symmetric monoidal categories commutes with taking modules, in a precise sense. This will be a key observation for the second part of the paper, and crucially uses the equivalence $N_{\CI^\amalg}\CC^\otimes\simeq \laxlim \CC_\bullet$. 

\begin{theorem}\label{thm:modules-in-laxlimit}
	Let $\CC^\otimes\to \CI^\amalg$ be a $\CI^\amalg$-monoidal $\infty$-category which is compatible with colimits, and write $\CC_\bullet\colon \CI\rightarrow \Cat_\infty^\otimes$ for the associated diagram of symmetric monoidal $\infty$-categories. Let $S\in\CAlg(\laxlim \CC_\bullet)$ be an algebra in the lax limit, which corresponds to a (partially lax) family of algebras $S_i\in \CC_i$. Then there is a functor 
	\[
	\Mod_{S_\bullet}(\CC_\bullet)\colon \CI\to \Cat_\infty^\otimes, \quad i \mapsto \Mod_{S_i}(\CC_i)
	\]
	and an equivalence of symmetric monoidal $\infty$-categories
	\[\laxlim\Mod_{S_\bullet}(\CC_\bullet)\simeq \Mod_S(\laxlim \CC_\bullet )\,.\]	
	Moreover, there is a natural transformation $\CC_\bullet\to \Mod_{S_\bullet}(\CC_\bullet)$ sending $x\in \CC_i$ to the free $S_i$-module $S_i\otimes x$, which induces the functor $S\otimes -$ on the lax limit.
\end{theorem}

The proof of the previous result will require some preparation and some results from Appendix A. For this reason we recommend the reader to skip this part on a first reading.

We start our journey by studying how the lax limit interacts with the tensor product of algebras. 

\begin{construction}
By \cite{HA}*{Proposition~3.2.4.6} there is an equivalence of $\infty$-operads $\CI^\amalg\otimes_{BV} \Finp\simeq \CI^\amalg$, where $\otimes_{BV}$ is the Boardmann-Vogt tensor product, and so there exists a unique bifunctor of $\infty$-operads $\CI^\amalg \times\Finp\to \CI^\amalg$. For any $\infty$-operad $\CO^\otimes$ we obtain a bifunctor of $\infty$-operads $m_\CO$, which is given by the composition
\[\CI^\amalg \times \CO^\otimes\to \CI^\amalg \times \Finp\to \CI^\amalg\,.\]    
Thus, for every map of $\infty$-operad $\CC^\otimes\to \CI^\amalg$ \cite[Construction~3.2.4.1]{HA} produces a map of $\infty$-operads
\[\Alg_{\CO^\otimes/\CI^\amalg}(\CC)^\otimes\to \CI^\amalg\]
whose operadic fibre over $i\in I$ is given by $\Alg_{\CO^\otimes}(\CC_i)^\otimes$. Suppose that $\CC^\otimes$ is a $\CI^\amalg$-monoidal category. Then by \cite[Proposition~3.2.4.3.(3)]{HA} $\Alg_{\CO^\otimes}(\CC_i)^\otimes$ is also a $\CI^\amalg$-monoidal $\infty$-category. In this case, Proposition~\ref{prop:equiv_amalg} gives a functor $\CI\to \Cat_\infty^\otimes$ sending $i\in\CI$ to $\Alg_{\CO^\otimes}(\CC_i)^\otimes$. We will now compute the lax limit of this functor.
\end{construction}

\begin{lemma}\label{lemma:algebras-in-laxlimit}
	Let $\CI$ be an $\infty$-category, $\CC^\otimes\to \CI^\amalg$ a map of $\infty$-operads and $\CO^\otimes$ an $\infty$-operad. Then there is a natural equivalence of $\infty$-operads
	\[\Alg_{\CO^\otimes}(N_{\CI^\amalg}\CC^\otimes)^\otimes \simeq N_{\CI^\amalg}\Alg_{\CO^\otimes/\CI^\amalg}(\CC)^\otimes\,.\]
	In particular if $\CC^\otimes$ is $\CI^\amalg$-symmetric monoidal we have a natural equivalence of $\infty$-operads
	\[\Alg_{\CO^\otimes}(\laxlim_{i\in \CI}\CC_i)^\otimes\simeq \laxlim_{i\in \CI} \Alg_{\CO}(\CC_i)^\otimes\,.\]
\end{lemma}
\begin{proof}
	We will prove that both sides represent the same functor in the $\infty$-category of $\infty$-operads. Let $\CP^\otimes$ an $\infty$-operad. Then
	\begin{equation*}\begin{split}
			\Alg_{\CP^\otimes}N_{\CI^\amalg}\Alg_{\CO^\otimes/\CI^\amalg}\CC^\otimes
			&\simeq \Alg_{\CP^\otimes\times_{\Finp}\CI^\amalg/I^\amalg}\left(\Alg_{\CO^\otimes}(\CC)^\otimes\times_{\Alg_{\CO^\otimes}(\CI)^\otimes}\CI^\amalg\right)\\
			&\simeq \Alg_{\CP^\otimes\times_{\Finp}\CI^\amalg/\Alg_{\CO^\otimes}(\CI)^\otimes}(\Alg_{\CO^\otimes}(\CC)^\otimes)\\
			&\simeq \Alg_{\CP^\otimes\times_{\Finp}\CI^\amalg}(\Alg_{\CO^\otimes}(\CC)^\otimes)\times_{\Alg_{\CP^\otimes\times_{\Finp}\CI^\amalg}(\Alg_{\CO^\otimes}(\CI)^\otimes)} \{pr_2\}\\
			&\simeq \Alg_{(\CP^\otimes\otimes_{BV} \CO^\otimes)\times_{\Finp} \CI^\amalg}(\CC^\otimes)\times_{\Alg_{(\CP^\otimes\otimes_{BV} \CO^\otimes)\times_{\Finp} \CI^\amalg}(\CI^\amalg)}\{pr_2\}\\
			&\simeq \Alg_{(\CP^\otimes\otimes_{BV} \CO^\otimes)\times_{\Finp} \CI^\amalg / \CI^\amalg}(\CC)^\otimes\\
			&\simeq \Alg_{\CP^\otimes}\left(\Alg_{\CO^\otimes}(N_{\CI^\amalg}\CC)^\otimes\right)\,.
	\end{split}\end{equation*}
	Here $\otimes_{BV}$ is the Boardman-Vogt tensor product of $\infty$-operads of \cite[Section~2.2.5]{HA}.
\end{proof}

We are ready to prove the main result of this section.

\begin{proof}[Proof of Theorem~\ref{thm:modules-in-laxlimit}]
	Note that by the definition of the norm we have an equivalence
	\[\CAlg(N_{\CI^{\amalg}}\CC^\otimes)\simeq \Alg_{\CI^\amalg/\CI^\amalg}(\CC^\otimes)\simeq \Alg_{\CI^\amalg}(\Alg_{\Finp/\CI^\amalg}(\CC)^\otimes)\]
	therefore we can also consider $S$ as a section of $\Alg_{\Finp/\CI^\amalg}(\CC)^\otimes\to \CI^\amalg$ in $\Op_\infty$.
	
	By Theorem~\ref{theorem:modules-are-pullbacks} and Lemma~\ref{lemma:algebras-in-laxlimit} there is an equivalence
	\begin{equation*}\begin{split}
			\Mod_S(N_{\CI^\amalg}\CC^\otimes)^\otimes&\simeq \Alg_{\CMod}(N_{\CI^\amalg}\CC)^\otimes\times_{\CAlg(N_{\CI^\amalg}\CC)^\otimes} \Finp\\
			&\simeq N_{\CI^\amalg}\left(\Alg_{\CMod/\CI^\amalg}(\CC)^\otimes\times_{\Alg_{\Finp/\CI^\amalg}(\CC)^\otimes} \CI^\amalg\right)
	\end{split}\end{equation*}
	where $\CI^\amalg\to \Alg_{\Finp/\CI^\amalg}(\CC)^\otimes$ is the section corresponding to $S$. Moreover note that by Lemma~\ref{lemma:families-of-commutative-algebras}
	\[\Alg_{\CMod/\CI^\amalg}(\CC)^\otimes\times_{\Alg_{\Finp/\CI^\amalg}(\CC)^\otimes} \CI^\amalg\to \CI^\amalg\]
	is an $\CI^\amalg$-monoidal $\infty$-category. Then Theorem~\ref{theorem:modules-are-pullbacks} shows that the corresponding family of symmetric monoidal $\infty$-categories is exactly
	\[i\mapsto\Mod_{S_i}(\CC_i)\,,\]
    and so our thesis follows from Proposition~\ref{proposition:laxlimit-is-norm}.
	
	Finally let us construct the symmetric monoidal functor $\CC_i^\otimes\to \Mod_{S_i}(\CC_i)^\otimes$. There is a map of $\CI^\amalg$-monoidal $\infty$-categories
	\[\Alg_{\CMod^\otimes/\CI^\amalg}(\CC)^\otimes\to \Alg_{\Finp/\CI^\amalg}(\CC)^\otimes \times_{\CI^\amalg}\CC^\otimes\]
	induced by the map of $\infty$-operads $\Finp\boxplus \operatorname{Triv}^\otimes\to \CMod^\otimes$ picking the algebra and the underlying object of the module. By \cite[Corollary~4.2.4.4]{HA} this has a left adjoint on every fibre which is compatible with the pushforwards by \cite[Corollary~4.2.4.8]{HA}, and so by \cite[Corollary~7.3.2.12]{HA} it has a relative left adjoint $F$ which is an $\CI^\amalg$-monoidal functor. Then the functor we want is the composite
	\[\CC^\otimes\xrightarrow{(S,\operatorname{id})}\Alg_{\Finp/\CI^\amalg}(\CC)^\otimes\times_{\CI^\amalg}\CC^\otimes \xrightarrow{F}\Alg_{\CMod^\otimes/\CI^\amalg}(\CC)^\otimes\,.\]
    This induces the desired functor on the lax limit since applying $N_{\CI^\amalg}$ preserve operadic adjunctions.
\end{proof}

\part{\texorpdfstring{$\infty$}{infty}-categories of global objects as partially lax limits}

In this second part of the paper we prove that various $\infty$-categories of global objects
admit a description using (partially lax) limits. In Theorem~\ref{thm:unstablelaxlim}, we show that the $\infty$-category of global spaces is equivalent to the partially lax limit of the functor sending a compact Lie group $G$ to the $\infty$-category of $G$-spaces. Our main result is Theorem~\ref{thm-laxlim-global-spectra} 
which describes the $\infty$-category of global spectra as a partially lax limit of $G$-spectra where $G$ runs over all compact Lie groups $G$. Finally, the techniques employed 
in the previous cases allow us to prove that for any Lie group $G$, the $\infty$-category of proper 
$G$-spectra is a limit of $H$-spectra for $H$ running over all compact subgroups of $G$.  
The precise statement can be found in Theorem~\ref{thm-lim-proper-spectra}.

\begin{remark*}\label{rem-works-for-other-families}
 To not burden the notation even more, we have decided to state 
 Theorem~\ref{thm:unstablelaxlim} and Theorem~\ref{thm-laxlim-global-spectra} for the family 
 of all compact Lie groups. However, the proofs hold verbatim for any family of compact 
 Lie groups which is closed under isomorphisms, finite products, passage to subgroups and 
 passage to quotients (i.e., any multiplicative global family in the language of~\cite{Schwede18}). If the family is not closed under finite products, then the equivalences of the two theorems still hold without symmetric monoidal structures. This is due to the fact that the model structure constructed in~\cite{Schwede18} is only shown to be symmetric monoidal for a multiplicative global family. We note that our result in fact allows us to define a symmetric monoidal structure on global spectra with respect to any global family, as a partially lax limit of symmetric monoidal categories is automatically symmetric monoidal.
\end{remark*}

\section{Global spaces as a  partially lax limit}\label{sec-global-spaces}

In this section we show that the $\infty$-category of global spaces is equivalent to a certain partially lax limit of the functor which sends a group $G$ to the $\infty$-category of 
$G$-spaces $\Spc_G$. This is an unstable version of our main result, and serves as a warm up for the considerable more details involved in that proof. We start off by recalling a few relevant definitions. 

\begin{definition}\label{def:glo}
	The \emph{global category} $\Glo$ is the $\infty$-category associated to the topological category whose objects are compact Lie groups and whose mapping spaces are given by
	\[\Map_{\Glo}(H,G):=|\Hom(H,G)\sslash G|\]
	the geometric realization of the action groupoid of $G$ acting on the space of continuous group homomorphisms $\Hom(H,G)$ by conjugation. Composition is induced by the composition of group homomorphisms.
	
	We define $\Orb$ and $\Glo^{\sur}$ to be the wide subcategory of $\Glo$ whose hom-spaces are given by those path-components of $\Map_{\Glo}(H,G)$ spanned by the injective and surjective group homomorphisms respectively. For later applications it will be convenient to mark all the edges in the full subcategory $\Orb \subseteq \Glo$; we denote this marking by $\Glo^\dagger$. Finally, we let $\Rep$ denote the homotopy category of $\Glo$, that is the category whose objects are compact Lie groups 
	and whose morphisms are given by conjugacy classes of continuous group homomorphisms. 
\end{definition} 

\begin{remark}
	The definition of $\Glo$ agrees with the definition given in Section 4 of \cite{GH} restricted to compact Lie groups, up to one difference. We apply thin geometric realization to the action groupoids to obtain a topologically enriched category, while the original definition uses fat geometric realization. Up to a technical condition, the two conventions define Dwyer-Kan equivalent topological categories. See \cite{korschgen}*{Remark 3.10} for a more detailed discussion. Note as well that \cite{GH} uses the name $\Orb$ for both $\Glo$ and what we call $\Orb$. 
\end{remark}

Key to the main properties of $\Glo$ is the following description of the mapping spaces.

\begin{proposition}\label{prop:glo_hom_spaces}
	Let $G,H$ be two compact Lie groups. Then \[\Hom(H,G) \simeq \coprod\limits_{[\alpha]\in \Rep(H,G)} \alpha G  \quad \text{and}\quad \Glo(H,G) \simeq \coprod\limits_{[\alpha]\in \Rep(H,G)} BC(\alpha)\]
	where $\alpha G$ denotes the orbit of $\alpha$ under the $G$-conjugation action,  
	and $C(\alpha)$ denotes the centraliser of the image of $\alpha$.
\end{proposition}

\begin{proof}
	See \cite{korschgen}*{Proposition 2.4, 2.5} for a proof of the first and second statement respectively.
\end{proof}

\begin{proposition}\label{prop:Glo_composition}
	Let $f\colon H\rightarrow G$ be a map in $\Glo$. The induced map on mapping spaces
	$f_*\colon \Glo(K,H)\rightarrow \Glo(K,G)$ correspond under the equivalences of Proposition~\ref{prop:glo_hom_spaces} to the composite of the map \[\coprod\limits_{[\alpha]\in \Rep(H,G)} Bf \colon \coprod\limits_{[\alpha]\in \Rep(K,H)} BC(\alpha) \rightarrow \coprod\limits_{[\alpha]\in \Rep(K,H)} BC(f\alpha)\] with the map
	\[\coprod\limits_{[\alpha]\in \Rep(K,H)} BC(f\alpha) \rightarrow  \coprod\limits_{[\beta]\in \Rep(K,G)} BC(\beta)\] which is the identity on individual path-components, and acts on $\pi_0$ by the map $f_*\colon \Rep(K,H)\rightarrow \Rep(K,G)$.
\end{proposition}

\begin{proof}
	The statement on $\pi_0$ follows from the fact that $\Rep$ is the homotopy category of $\Glo$. Therefore, it suffices to restrict to one path component, and analyse the effect of $f$. The relationship $f_*(c_h \alpha) = c_{f(h)} f \alpha$ implies that $f_*$ acts as $f$ when restricted to a map $\alpha H\rightarrow f\alpha G$. This implies that the induced map $BC(\alpha)\rightarrow BC(f\alpha)$ equals $Bf$.
\end{proof}

\begin{definition}\label{def-global-spaces}
	The $\infty$-category of \emph{global spaces} $\Spcgl$ is the category of functors from 
	$\Glo^{\op}$ to $\Spc$. This admits a symmetric monoidal structure by 
	pointwise product. This is equivalent to the symmetric monoidal category $\DgSpc{(\Glo^{\op})^\amalg}$.
\end{definition}

\begin{remark}\label{rem-models-are-equivalent}
	In~\cite{Schwede20}, the author proves that the underlying $\infty$-category of orthogonal spaces equipped with the positive global model structure of~\cite{Schwede18}*{Proposition 1.2.23} is equivalent to presheafs on a topologically enriched category $\OO_{gl}$. Furthermore, in \cite{korschgen} it is shown that $\OO_{gl}$ is Dwyer-Kan equivalent to $\Glo$. Therefore the two models of global spaces define the same $\infty$-category. In fact, the two $\infty$-categories are symmetric monoidal equivalent since they are both endowed with the cartesian monoidal structure, see~\cite{Schwede18}*{Theorem 1.3.2}.
\end{remark}

Before stating and proving the main result of this section, we need some preparation.
In the following we fix an $\infty$-category $\CC$ with an orthogonal factorization system $(\CC^L,\CC^R)$. For a detailed discussion and a definition of orthogonal factorization systems on $\infty$-categories, the reader may consult \cite{HTT}*{Section 5.2.8}. We write $\CC^L$ for the left class of maps and $\CC^R$ for the right class. We will denote edges in $\CC^L$ by $\twoheadrightarrow$ and edges in $\CC^R$ by $\rightarrowtail$. 

\begin{proposition}\label{prop:ArRcocart} 
	Let $\CC$ be an $\infty$-category equipped with an orthogonal factorization $(\CC^L, \CC^R)$. 
	Write $\Ar_R(\CC)$ for the full subcategory of the arrow category of $\CC$ spanned by the edges 
	in $\CC^R$. 
	Then the target projection $t\colon\Ar_R(\CC)\rightarrow \CC$ is a cocartesian fibration. Furthermore an edge in $\Ar_R(\CC)$ is $t$-cocartesian if and only if it is of the form 
	\begin{equation}\label{cocartesian}
		\begin{tikzcd}
			X \arrow[d, tail] \arrow[r, two heads] & Y \arrow[d, tail] \\
			X' \arrow[r]                           & Y'  .             
		\end{tikzcd}
	\end{equation}
\end{proposition}

\begin{proof}
	Consider an edge in $\Ar_R(\CC)$:
	\begin{center}
		\begin{tikzcd}
			X \arrow[d, tail] \arrow[r] & Y \arrow[d, tail] \\
			X' \arrow[r]                           & Y'    .           
		\end{tikzcd}
	\end{center}
	This is cocartesian if and only if, given a 2-simplex in $\CC$ and a (2,0)-horn in $\Ar_R(\CC)$, there is a contractible choice of extensions. This corresponds to showing that given a diagram in $\CC$
	\begin{center}
		\begin{tikzcd}
			X \arrow[d, tail] \arrow[r] \arrow[rr, bend left] & Y \arrow[d, tail] \arrow[rd] & Z \arrow[d, tail] \\
			X' \arrow[r]                                                 & Y' \arrow[r]                 & Z',              
		\end{tikzcd}
	\end{center} its extensions to a 2-simplex in $\Ar_R(\CC)$ form a contractible space. However, completing this diagram is equivalent to supplying an edge $Y\rightarrow Z$ which makes the diagram 
	\begin{center}
		\begin{tikzcd}
			X \arrow[d] \arrow[r] & Z \arrow[d, tail] \\
			Y \arrow[r] \arrow[ru, dotted]   & Z'               
		\end{tikzcd}
	\end{center} commute. There is a contractible choice of such factorizations if and only if $ X\rightarrow Y$ is in $\CC^L$. This shows that an edge is $t$-cocartesian if and only if it is of the form of Equation (\ref{cocartesian}).
	Next, fix an edge in $\CC$ and a lift of its source in $\Ar_R(\CC)$. This corresponds to a diagram 
	\begin{center}
		\begin{tikzcd}
			X \arrow[d, tail] \arrow[rd] &    \\
			X' \arrow[r]                 & Y'.
		\end{tikzcd}
	\end{center} Factorizing the composite $X\rightarrow Y'$ extends this to an edge 
	\begin{center}
		\begin{tikzcd}
			X \arrow[d, tail] \arrow[rd] \arrow[r, two heads] & Y \arrow[d, tail] \\
			X' \arrow[r]                                      & Y'               
		\end{tikzcd}
	\end{center} in $\Ar_R(\CC)$, which is $t$-cocartesian.
\end{proof}

We record the following fact for later reference.

\begin{lemma}\label{lem:adjointsource}
	The constant functor $s_0\colon \CC\rightarrow \Ar_R(\CC)$ is a fully faithful left adjoint to the source functor $s\colon\Ar_R(\CC)\rightarrow \CC$.
\end{lemma}

\begin{construction}\label{cons-functoriality_ARR}
	Suppose we are in the setting of Proposition~\ref{prop:ArRcocart}. Straightening the cocartesian fibration $t\colon\Ar_R(\CC)\rightarrow \CC$ gives a functor 
	\[
	\CC^R_{/-}\colon\CC\rightarrow \Cat_\infty.
	\]
	To justify our notation let us unravel the effect of this functor. By definition, the evaluation of $\CC^R_{/-}$ at an object $X\in \CC$ is given by $\Ar_R(\CC)_X$; the fibre of $t$ at $X$. By construction this is the full subcategory of $\CC_{/X}$ on the objects $C\rightarrowtail X$ in $\CC^R$. A priori an edge in this full subcategory is given by a diagram
	\begin{center}
		\begin{tikzcd}
			X \arrow[r] \arrow[rd, tail] & X' \arrow[d, tail] \\
			& Y.                 
		\end{tikzcd}
	\end{center} However the edge $X\rightarrow X'$ is necessarily also in $\CC^R$ by \cite{HTT}*{Proposition 5.2.8.6(3)}, and therefore $\Ar_R(\CC)_X$ is in fact equivalent to $\CC^R_{/X}$. Next consider an edge $f\colon Y\rightarrow Y'$. Then the induced map $f_*\colon\CC^R_{/Y}\rightarrow \CC^R_{/Y'}$ sends an object $X\rightarrowtail Y$ to an object $X'\rightarrowtail Y'$ such that the following diagram commutes:
	\begin{center}
		\begin{tikzcd}
			X \arrow[d, tail] \arrow[r, two heads] & X' \arrow[d, tail] \\
			Y \arrow[r, "f"]                       & Y' .        
		\end{tikzcd}
	\end{center}
	In particular if $f\in \CC^R$ this is nothing but the standard functoriality of the slices $\CC^R_{/-}$. Therefore the functor $\CC^R_{/-}\colon\CC\rightarrow \Cat_\infty$ extends the functoriality of the slices of $\CC^R$ to all of $\CC$.
\end{construction}

\begin{proposition}\label{prop:partially-lax-colimit-is-C}
	Let $\CC$ be an $\infty$-category equipped with a factorization system $(\CC^L, \CC^R)$.
	The partially lax colimit of $(-)^{\op}\circ \CC^R_{/-} \colon \CC\rightarrow \Cat_\infty$ with respect to the marking $\CC^R\subset \CC$ is equivalent to $\CC^{\op}$.
\end{proposition}

\begin{proof}
	Recall that the partially lax colimit of a functor $F\colon\CC\rightarrow \Cat_\infty$ is the localization of $\Unct{F}$ at the cartesian edges which live above marked edges, see Theorem~\ref{thm-lax-limit-section}(b). In the case $F=(-)^{\op}\circ \CC^R_{/-}$, we observe that $\Unct{F} \simeq \Unco{\CC^R_{/-}}^{\op}$ and so we conclude that the partially lax colimit of $F$ is equal to the opposite of $\Ar_R(\CC)$ localized at the edges of the form
	\begin{center}
		\begin{tikzcd}
			X \arrow[d, tail] \arrow[r, two heads] & X' \arrow[d, tail] \\
			Y \arrow[r, tail]                      & Y'.       
		\end{tikzcd}
	\end{center}
	However note that because edges in $\CC^R$ are left cancellable, $X\rightarrow X'$ is not only in $\CC^L$ but also in $\CC^R$. Therefore $X\rightarrow X'$ is in fact an equivalence. We will write $M$ for this collection of edges. We claim that localizing at the edges of $M$ is equivalent to localizing at the larger class of edges $M'$ of the form 
	\begin{center}
		\begin{tikzcd}
			X \arrow[d, tail] \arrow[r, "\sim"] & X' \arrow[d, tail] \\
			Y \arrow[r]                      & Y',               
		\end{tikzcd}
	\end{center} where we do not impose any conditions on the edge $Y\rightarrow Y'$. To see this note that such an edge in $M'$ fits into the following diagram:
	\begin{center}
		\begin{tikzcd}
			X \arrow[r, "\sim"] \arrow[d, "\sim"']          & X' \arrow[d, tail] \arrow[r, equals] & X' \arrow[d, tail] \\
			X' \arrow[r, tail] \arrow[rr, tail, bend right] & Y \arrow[r]                         & Y' .
		\end{tikzcd}
	\end{center}
	Both the first edge and the composite are in $M$, and so therefore $M'$ is contained in the two-out-of-three closure of $M$. So it is enough to calculate the localization of $\Ar_R(\CC)$ at $M'$. Note that the source functor $s\colon\Ar_R(\CC)\rightarrow \CC$ sends an edge to an equivalence if and only if it is in $M'$. Then Lemma~\ref{lem:adjointsource} implies that $\CC$ is a Bousfield colocalization of $\Ar_R(\CC)$ at $M'$. So we conclude that the partially lax colimit of $^{\op}\circ \CC^R_{/-}$ is equivalent to $\CC^{\op}$, finishing the proof.
\end{proof}

\begin{example}
	There are two extreme cases of the previous result. If $\CC^R = \CC,\, \CC^L = \iota \CC$, then
	\[
	\mathrm{colim}((\CC_{/-})^{\op}\colon\CC\rightarrow \mathrm{Cat}_\infty) \cong \CC^{\op}
	\]
	If $\CC^R = \iota \CC,\, \CC^L = \CC$, then 
	\[
	\laxcolim(\iota \CC_{/-} \colon\CC\rightarrow \mathrm{Cat}_\infty) \cong \CC^{\op}.
	\]
\end{example}

Now that we have introduced the main tools we need, we can build our functor and compute its partially lax limit. This relies on two important observations. The first key insight is the following, which was first stated in \cite{GH} and originally proven as ~\cite{Rezk}*{Example 3.5.1}.

\begin{lemma}\label{lem:Orbslices}
	For all compact Lie group $G$, the assignment $G/K \mapsto (K \hookrightarrow G)$ defines 
	an equivalence $\OO_G\simeq \Orb_{/G}$.
\end{lemma}

\begin{proof}
	We observe that the spaces $\OO_G(G/H,G/K)$ are homeomorphic to the space $\{g\in G\mid c_g(H)\subseteq K\}/K$. The latter space is equivalent to the homotopy orbits $\{g\in G\mid c_g(H)\subseteq K\}_{\h K}$ as the $K$-space is free, see for example \cite{korschgen}*{Theorem A.7}. Therefore we can define a functor $F'\colon \OO_G\rightarrow \Glo$, which sends $G/H$ to $H$, and on mapping spaces acts as homotopy orbits of the $K$-equivariant inclusion
	\[\{g\in G\mid c_g(H)\subseteq K\} \rightarrow \hom(H,K),\quad g\mapsto [c_g\colon H\rightarrow K].\] Note that the $\infty$-category $\OO_G$ has a final object $G/G$, and therefore $F'$ induces a functor $ \OO_G\rightarrow \Glo_{/G}$, which in fact factors through $\Orb_{/G}$. We claim that the induced functor $F\colon \OO_G \to \Orb_{/G}$ is an equivalence of $\infty$-categories. First note that $F$ is clearly essentially surjective. To deduce that the functor is fully faithful pick two objects $G/H$ and $G/K$ which we identify with inclusions $i\colon H\hookrightarrow G$ and $j\colon K\hookrightarrow G$. Recall that the mapping space between $G/H$ and $G/K$ is empty if and only if $H$ is not subconjugate to $K$. In this case the mapping space in $\Orb/G$ between $i$ and $j$ is also empty. Now suppose that this is not the case. Consider the square
	\[\begin{tikzcd}
		\{g\in G\mid c_g(H)\subseteq K\}_{\h K}\ar[r]\ar[d] & \Hom(H,K)_{\h K}\ar[d]\\
		\ast\ar[r] & \Hom(H,G)_{\h G}.
	\end{tikzcd}\]
	To prove $F$ is fully faithful it suffices to prove that this square is homotopy cartesian. For every $K$-space $X$, $(G\times_K X)_{\h G}\simeq X_{\h K}$, so that the above square is equivalent to
	\[\begin{tikzcd}
		\left(G\times_K \{g\in G\mid c_g(H)\subseteq K\}\right)_{\h G}\ar[r]\ar[d] & \left(G\times_K\Hom(H,K)\right)_{\h G}\ar[d]\\
		G_{\h G}\ar[r] & \Hom(H,G)_{\h G}.
	\end{tikzcd}\]
	Because taking homotopy orbits preserves homotopy pullback diagrams, it suffices to show that the square
	\[\begin{tikzcd}
		G\times_K \{g\in G\mid c_g(H)\subseteq K\}\ar[r]\ar[d] & G\times_K\Hom(H,K)\ar[d]\\
		G\ar[r] & \Hom(H,G)
	\end{tikzcd}\]
	is homotopy cartesian. In fact it is easily shown to be a pullback square of topological spaces, and the bottom horizontal arrow is a Serre fibration. To see this we note that the map $G\rightarrow \Hom(H,G)$ factors through one component of the decomposition of Proposition \ref{prop:glo_hom_spaces}, and therefore is equivalent to the quotient map $G\rightarrow G/C(H)$ which is a fibration by \cite{korschgen}*{Theorem A.9}.
\end{proof}
The second insight is the following, which was also observed in \cite{Rezk}. 

\begin{proposition}\label{prop:Glofactsystem}
	The subcategories $\Glo^{\sur}$ and $\Orb$ are the left and right classes respectively of an orthogonal factorization system on $\Glo$.
\end{proposition}

\begin{proof}
	We will apply \cite{HTT}*{Proposition 5.2.8.17} to the subcategories $\Glo^{\sur}$ and $\Orb$. Clearly these subcategories contain all the equivalences and are closed under equivalences in $\mathrm{Ar}(\Glo)$. Therefore it suffices to prove that given a diagram:
	\begin{center}
		\begin{tikzcd}
			H \arrow[r] \arrow[d, "f"', two heads] & J \arrow[d, "g", hook] \\ 
			G \arrow[r] \arrow[ru, dotted] & K,
		\end{tikzcd}
	\end{center}
	the space of dotted diagonal fillers is contractible. As noted in \cite{HTT}*{Remark 5.2.8.3}, this is equivalent to the map
	$$\Map_{\Glo_ {H/}}(H \overset{f}{\rightarrow} G, H\rightarrow J) \overset{g} {\longrightarrow}\Map_{\Glo_ {H/}}(H \overset{f}{\rightarrow} G, H\rightarrow K)$$ being a weak homotopy equivalence for every lift of $g$ to a map in $\Glo_{H/}$ from $H\rightarrow J$ to $H\rightarrow K$. Proposition~\ref{prop:Glo_composition} shows that when $f$ is surjective the map 
	$$\Map_{\Glo}(G,J) \overset{f^*}{\rightarrow} \Map_{\Glo}(H,J)$$ is an inclusion of path-components for every $J$.
	
	Therefore the space $\Map_{\Glo_ {H/}}(H \overset{f}{\rightarrow} G, H\rightarrow J)$, being the homotopy fibre of this map, is either empty or contractible. Translating back this reduces our task to simply proving the existence of a lift in the square above. This is a simple exercise in group theory.
\end{proof}

\begin{remark}
	When we restrict to finite groups, $\Glo$ is equivalent to the full subcategory of $\Spc$ given by the connected 1-truncated spaces. In this case the orthogonal factorization system constructed above is a restriction of the standard mono/epi factorization system of any $\infty$-topos. However in the generality of compact Lie groups no such description applies.
\end{remark}

We are finally ready to construct the functor.

\begin{construction}\label{const:Orbslices}
	Applying Construction~\ref{cons-functoriality_ARR} to the orthogonal factorization system $(\Glo^{\sur},\Orb)$ yields a functor $\Orb_{/-}\colon\Glo\rightarrow \Cat_\infty$. Post-composing the opposite of this functor with $\mathrm{Fun}((-)^{\op},\Spc)\colon\Cat_\infty^{\op}\rightarrow \Cat_\infty$ gives the desired functor \[\Spc_\bullet\colon \Glo^{\op}\rightarrow \Cat_\infty.\] Also note that $\Spc_\bullet$ clearly factors through product preserving functors, and so enhances to a functor 
	\[\Spc_\bullet\colon \Glo^{\op}\rightarrow \Cat_{\infty}^\otimes,\] where each category $\DgSpc{(\Orb_{/G})^{\op}}$ is given the cartesian monoidal structure.
\end{construction} 

Lemma \ref{lem:Orbslices} and Elmendorf's theorem for $G$-spaces, see \Cref{ex-S_G}, imply that the value of $\Spc_\bullet$ at the object $G$ is equivalent to the $\infty$-category of $G$-spaces $\Spc_G$. However we owe the reader the following consistency check, which implies that the functor $\Spc_\bullet$ also has the expected functoriality.

\begin{proposition}\label{prop-restriction}
	Let $\alpha\colon H\rightarrow G$ be a continuous group homomorphism. 
	Then the following diagram commutes:
	\[
	\begin{tikzcd}
		\Fun((\Orb_{/G})^{\op}, \Spc) \arrow[r, "\simeq"] \arrow[d, "\Spc_{\alpha}"'] &\Spc_{G}  
		\arrow[d, "\alpha^*"] \\
		\Fun((\Orb_{/H})^{\op}, \Spc) \arrow[r, "\simeq"]            & \Spc_H.                 
	\end{tikzcd}
	\]
	Here the horizontal equivalences are obtained by applying Lemma~\ref{lem:Orbslices} 
	and Example \ref{ex-S_G}.
\end{proposition}

\begin{proof}
	It is enough to check that the analogous diagram where the vertical maps are replaced with left adjoints commutes. For this, let us denote by $L_{\alpha}$ and $\alpha_!$ the left adjoints of $\Spc_{\alpha}$ and $\alpha^*$ respectively.
	Note that the inclusion $\iota_H\colon \Orb_{/H}\hookrightarrow \Glo_{/H}$ has a left adjoint $L_H$ which on objects sends $K\xrightarrow{\beta} H$ to $\beta(K) \hookrightarrow H$. By the universal property of the presheaf categories there exists a unique cocontinuous functor (the left Kan extension along $\iota_H$)
	\[
	(\iota_H)_!\colon \Fun((\Orb_{/H})^{\op},\Spc) \to \Fun((\Glo_{/H})^{\op},\Spc) 
	\]
	which agrees with $\iota_H$ on representables. In a similar fashion, we define functors $(L_G)_!$ and $(\alpha_*)_!$ where $\alpha_*\colon \Glo_{/H} \to \Glo_{/G}$ is postcomposition by $\alpha$. We claim that the following diagram commutes:
	\[
	\begin{tikzcd}[column sep=large]
		{\Fun((\Orb_{/H})^{\op}, \Spc)} \arrow[r, hook,"{(\iota_H)_!}"] \arrow[d, "L_\alpha"'] & {\Fun((\Glo_{/H})^{\op}, \Spc)} \arrow[d, "{(\alpha_*)_!}"] \\
		{\Fun((\Orb_{/G})^{\op}, \Spc)}                                       & {\Fun((\Glo_{/G})^{\op}, \Spc).} \arrow[l, "{(L_G)_!}"]       
	\end{tikzcd}
	\] This is easily seen by comparing the result on generators, and using that all the functors in the diagram commute with all colimits. Using this diagram we can reduce to a statement on the level of model categories. Namely all three functors which make up the long way around in the diagram above can be modelled by left Quillen functors between enriched functor categories with the projective model structure. Indeed, the right adjoint of $(\iota_H)_!$ is given by restriction along $\iota_H$ which is clearly a right Quillen functor. A similar argument also works for $(L_G)_!$ and $(\alpha_*)_!$.
	After pre-composing and post-composing with the equivalences 
	\[
	\Top_H \simeq \Fun^{\mathrm{top}}((\Orb_{/H})^{\op}, \Top) \quad \mathrm{and} \quad\Fun^{\mathrm{top}}((\Orb_{/G})^{\op}, \Top)\simeq \Top_G
	\]
	constructed in \cite{Rezk}*{Proposition 3.5.1} (which agree with the equivalences constructed by \cite{gepner2020equivariant} by inspection), we can apply the explicit description for $(L_G)_!$ and $(\iota_H)_!$ given in \cite{Rezk}*{Section 5.3} (where $(L_G)_!$ is denoted by $\Pi_G$ and $(\iota_H)_!$ by $\Delta_H$) to deduce that the functor $L_\alpha \colon \Top_H\rightarrow \Top_G$ is equivalent to induction of $H$-spaces.
\end{proof}

We have now constructed our functor. Therefore we are left to prove that the partially lax limit is given by the $\infty$-category of global spaces.

\begin{theorem}\label{thm:unstablelaxlim}
	Let $\Glo^\dagger$ denote the marked $\infty$-category from Definition~\ref{def:glo}. 
	Then the partially lax limit over $(\Glo^\dagger)^{\op}$ of the diagram from Construction~\ref{const:Orbslices}
	\[
	\Glo^{\op}\rightarrow \Cat_\infty^\otimes,\quad G\mapsto \Spc_G
	\] 
	is equivalent to the $\infty$-category of global spaces, equipped with the cartesian monoidal structure.
\end{theorem}

\begin{proof}
	Recall that $\Spc_G=\Fun(\OO_G^{\op}, \Spc)$ and that $\OO_G \simeq \Orb_{/G}$.
	First we prove the result on underlying categories. Proposition \ref{prop:funpartlim} implies that it suffices to prove an equivalence between the partially lax colimit of $(\Orb_{/-})^{\op}$ and $\Glo^{\op}$. However this follows from Proposition \ref{prop:partially-lax-colimit-is-C} applied to the factorization system $(\Glo^{\sur},\Orb)$ on $\Glo$. Now we deduce the symmetric monoidal statement. First observe that the equivalence constructed before trivially lifts to a symmetric monoidal equivalence, where both sides are given the cartesian symmetric monoidal structure. Then note that the subcategory of $\Op_\infty$ spanned by the cartesian operads is closed under partially lax limits. This implies that $\Spc_{gl}$ is equivalent to the partially lax limit of the diagram $\Spc^\bullet\colon \Glo^{\op}\rightarrow \Cat_\infty^\otimes$, but now taken in symmetric monoidal $\infty$-categories.
\end{proof}

\section{\texorpdfstring{$\infty$}{infty}-categories of equivariant prespectra}\label{sec-cat-of-eq-prespectra}

In this section we define the $\infty$-categories of $G$-(pre)spectra for a Lie group $G$, and we introduce the $\infty$-category of global (pre)spectra. We will do this by first defining the relevant level model structures, which present the $\infty$-categories of prespectra objects, and then defining the stable model category as a Bousfield localization. This will then present the $\infty$-categories of spectra objects. The material in this section is classical, and largely well-known. Nevertheless we include the details of the model structures, mainly to emphasize that the level model structure on $\Sp^O_G$ is induced formally from the level model structure on $\DgTop{\I}{G}$. While not a deep statement, it is crucial to our proof strategy. In particular this observation will allow us to interpret the construction of the level model structure $\infty$-categorically, as will be explained in this section.

\begin{definition}
	Let $\I$ denote the topological category whose objects are finite dimensional 
	inner product spaces $V$, and morphisms space $\I(V,W)$ given by the space 
	of linear isometric isomorphisms from $V$ to $W$. 
\end{definition}

\begin{definition}
	Let $G$ be a Lie group (not necessarily compact). We write $\DgTop{\I}{G}$ for 
	the enriched category of continuous functors from $\I$ into $G$-spaces, and call this 
	the category of $\I$-$G$-spaces.
	When $G$ is the trivial group, we simply write $\DgTop{\I}{}$ and refer to 
	it as the category of $\I$-spaces.
\end{definition}

\begin{remark}
	As discussed in~\cite{Bohmann14}*{Section ~5}, the category of $\I$-$G$-spaces 
	(as defined above) is equivalent as a topological category to the category of 
	$\I_G$-spaces as defined by 
	Mandell-May in~\cite{MandellMay}*{Chapter II, Definition 2.3}. 
\end{remark}

\begin{remark}\label{rem-day-convolution}
	The category $\DgTop{\I}{G}$ has a symmetric monoidal structures given by enriched Day convolution, see~\cite{MandellMay}*{Chapter II, Proposition 3.7}. Given $X,Y\in \DgTop{\I}{G}$ we have the formula
	\[
	(X\otimes Y)(V):=\int^{(W,W')\in \I\times \I} \Li(W\oplus W',V)\times X(W)\times Y(W').
	\]
\end{remark}

\begin{remark}\label{rem-G-action}
	Given any $\I$-$G$-space $X$ and an inner product space $V$, the value $X(V)$ 
	admits a $G\times O(V)$-action. If $V$ is given the structure of an $H$-representation 
	$\rho \colon H \to O(V)$, then we can equip $X(V)$ with an $H$-action by restricting along 
	\[
	H \xrightarrow{\Delta} H \times H \xrightarrow{i \times \rho}G \times O(V).
	\] 
\end{remark}

We will always consider the value $X(V)$ with this $H$-action in the following. 

\begin{construction}[Free $\I$-$G$-space]\label{con-free-IG-space}
	For every $H$-representation $V$, there is an evaluation functor 
	\[
	\mathrm{ev}_V\colon \DgTop{\I}{G} \to H\Top, \qquad X \mapsto X(V).
	\] 
	This functor admits a left adjoint $G\times_H\I_{V}$, given by the formula 
	\[G\times_H\I_{V}A=G\times_H (\I(V,-)\times A).\] 
	When $A=*$, we simply write $G\times_H\I_{V}$ and when $G = H$, we write $\I_{V}(-)$. 
	By construction, the $\I$-$G$-space 
	$G\times_H\I_{V}$ corepresents the functor $X \mapsto X(V)^H$.  
	
	For all compact subgroups $H$ and $K$ of $G$, all $H$-representations $V$ and all 
	$K$-representations $W$, there is an isomorphism of $\I$-$G$-spaces 
	\begin{equation}\label{eq-tensor}
		(G\times_H \I_V)\otimes (G\times_K \I_W)\cong \Delta^*(G\times G \times_{H\times K} \I_{V\oplus W})
	\end{equation}
	where $\Delta\colon G \to G\times G$ is the diagonal embedding. 
	This can be checked directly by applying the formula of the Day convolution product from 
	Remark~\ref{rem-day-convolution} and using that induction commutes with colimits.  
\end{construction}

We will now proceed to equip the category of $\I$-$G$-spaces with the level model structure. The following will be the weak equivalences, fibrations and cofibrations of this model structure.

\begin{definition}
	Let $G$ be a Lie group and let $f\colon X\rightarrow Y$ be a morphism in $\DgTop{\I}{G}$.
	\begin{itemize}
		\item[(a)] We say $f$ is a \emph{level equivalence} if for any compact subgroup 
		$H\leq G$ and any $H$-representation $V$, the map 
		$f(V)^H\colon X(V)^H\rightarrow  Y(V)^H$ is a weak homotopy equivalence of spaces.
		\item[(b)]We say $f$ is a \emph{level fibration} if for any compact subgroup 
		$H\leq G$ and any $H$-representation $V$, the map 
		$f(V)^H\colon X(V)^H\rightarrow  Y(V)^H$ is a Serre fibration.
		\item[(c)] We say $f$ is a \emph{level cofibration} if for every $m\geq 0$, the map 
		$f(\mathbb{R}^m)\colon X(\mathbb{R}^m)\to Y(\mathbb{R}^m)$ is a $Com$-cofibration of 
		$(G \times O(m))$-spaces, see~\cite{Proper}*{Definition 1.1.2}, and moreover the 
		$O(m)$-action is free away from the image of $f(\mathbb{R}^m)$.
	\end{itemize}
\end{definition}

For all $m\geq 0$, we let $\CC_G(m)$ denote the family of compact subgroups $\Gamma$ of $G\times O(m)$ such that $\Gamma \cap (1\times O(m))$ consists only of the neutral element.  These are 
precisely the graph subgroups of a continuous homomorphism to $O(m)$ defined on some compact subgroup of $G$.
The category of $G \times O(m)$-spaces admits a $\CC_G(m)$-projective model structure by~\cite{Schwede18}*{Proposition B.7}. We have the following useful characterization of the 
level equivalences, cofibrations and fibrations.

\begin{lemma}\label{lem-same}
 	Let $G$ be a Lie group and let $f\colon X\rightarrow Y$ be a morphism in $\DgTop{\I}{G}$.
    The following are equivalent:
    \begin{itemize}
    \item[(a)] the map $f\colon X \to Y$ is a level equivalence (resp., level fibration);
    \item[(b)] the map $f(\mathbb{R}^m)\colon X(\mathbb{R}^m)\to Y(\mathbb{R}^m)$ is a weak 
    equivalence (resp., fibration) in the $\CC_G(m)$-projective model structure for all $m\geq 0$.
    \end{itemize}    
    Furthermore, the following are equivalent:
    \begin{itemize}
     \item[(c)] the map $f\colon X \to Y$ is a level cofibration;
    \item[(d)] the map $f(\mathbb{R}^m)\colon X(\mathbb{R}^m)\to Y(\mathbb{R}^m)$ is a 
    cofibration in the $\CC_G(m)$-projective model structure for all $m\geq 0$.
    \end{itemize}
\end{lemma}

\begin{proof}
 Let $H \leq G$ be a compact subgroup and let $V$ be an $H$-representation. Choose a linear 
 isometric isomorphism $\varphi\colon V \cong \mathbb{R}^m$ and define a group homomorphism
 \[
 \rho \colon G \to O(m), \qquad g\mapsto \varphi \circ (g\cdot -)\circ \varphi^{-1}.
 \] 
 The homeomorphism $X(\varphi)\colon X(V) \simeq X(\mathbb{R}^m)$ restricts to a homeomorphism 
 \[
 X(V)^H \simeq X(\mathbb{R}^m)^{\Gamma(\rho)}
 \]
 where $\Gamma(\rho)=\{(h,\rho(h))\in H \times O(m)\}$ by the definition of the $H$-action given 
 in~Remark~\ref{rem-G-action}. From this description, it is clear that (b) implies (a). 
 Conversely given $\Gamma\in \CC_G(m)$, we can always find a continuous group 
 homomorphism $\alpha\colon H \to O(m)$ for $H\leq G$ compact such that $\Gamma=\Gamma(\alpha)$. 
 By definition of the $H$-action, we have $X(\mathbb{R}^m)^H = X(\mathbb{R}^m)^{\Gamma}$ showing that (a) implies (b).  
 Finally, that (c) and (d) are equivalent follows from (the topological version of) 
 \cite{Stephan}*{Proposition 2.16}.

\end{proof}

\begin{theorem}
	Let $G$ be a Lie group. The category $\DgTop{\I}{G}$ admits a cofibrantly generated 
	and topological model structure in which the weak equivalences are the level equivalences, 
	the fibrations are the level fibrations and the cofibrations are the level cofibrations.  
	The set of generating cofibrations $I_G$ and acyclic cofibrations $J_G$ are given by 
	\begin{align*}
		I_G= & \{G\times_H \I_{V} \partial D^n \to G\times_H\I_{V} D^n 
		\mid H \leq G, \;\; n\geq 0 \}\\
		J_G =& \{ G\times_H \I_{V} (D^n\times \{0\}) \to G\times_H\I_{V}( D^n \times [0,1])
		\mid H \leq G, \;\; n\geq 0\}
	\end{align*}
	where $H$ runs over all compact subgroups of $G$ and $V$ runs over all $H$-representations. 
	We call this the \emph{(proper) level model structure}.
\end{theorem}

\begin{proof}
    We observe that the category $\DgTop{\I}{G}$ is equivalent to $\prod_{m\geq 0} (G \times O(m))\Top$. We can endow this latter category with the product of the $\mathcal{C}_G(m)$-projective model structures on $G \times O(m)$-spaces. By Lemma~\ref{lem-same}, the induced model structure on $\DgTop{\I}{G}$ has weak equivalences, fibrations and cofibrations as in the theorem. Also we note that the right lifting property against the sets $I_G$ and $J_G$ detect the level fibrations 
	and level acyclic fibrations respectively, by the adjunction isomorphism
	\[
	\Hom_{\DgTop{\I}{G}}(G\times_H\I_{V}A, X)\simeq \Hom_{\Top}(A, X(V)^H)
	\]
	for $A$ a non-equivariant space. Finally we observe that resulting model structure is again topological by \cite{Schwede18}*{Proposition B.5}.
\end{proof}

As discussed in~\cite{Proper}*{Proposition 1.1.6}, a continuous homomorphism 
$\alpha\colon K \to G$ between Lie groups gives rise to adjoint functors between the 
associated category of equivariant spaces
\[
\begin{tikzcd}
	G\Top  \arrow[r,"\alpha^*"] & K\Top \arrow[l, bend right, "G\times_\alpha -"'] 
	\arrow[l, bend left, "{\Map^{\alpha}(G,-)}"]  \\
\end{tikzcd}
\]
which by levelwise application gives rise to an adjoint triple
\[
\begin{tikzcd}
	\DgTop{\I}{G}  \arrow[r,"\alpha^*"] & \DgTop{\I}{K} \arrow[l, bend right, "G\times_\alpha -"'] 
	\arrow[l, bend left, "{\Map^{\alpha}(G,-)}"] . \\
\end{tikzcd}
\]

\begin{proposition}\label{prop-change-of-groups-functors}
	Let $\alpha\colon K \to G$ be a continuous group homomorphism between 
	Lie groups. 
	\begin{itemize}
		\item[(a)] Then $\alpha^*$ preserves level 
		fibrations and level equivalences. Thus the adjoint pair $(G\times_\alpha-, \alpha^*)$ 
		is Quillen.
		\item[(b)] If $\alpha$ has closed image and compact kernel, then the adjoint pair $(\alpha^*, \Map^\alpha(G,-))$ is also Quillen with respect to the level model structure. 
	\end{itemize}
\end{proposition}

\begin{proof}
 Part (a) follows from~\cite{Proper}*{Proposition 1.1.6(ii)}. 
 Suppose that $\alpha$ has closed image and compact kernel and note that by (a), 
 it suffices to check that $\alpha^*$ preserves level cofibrations. We start by noting 
 that the image of $\alpha\times O(m)$ is closed in $G\times O(m)$ since the image of $\alpha$ is 
 closed in $G$. 
 Moreover, the kernel of $\alpha\times O(m)$ is $\ker(\alpha)\times 1$, which is compact by 
 hypothesis. So restriction along $\alpha\times O(m)$ takes $Com$-cofibrations of 
 $(G\times O(m))$-spaces to $Com$-cofibrations of $(K \times O(m))$-spaces 
 by~\cite{Proper}*{Proposition 1.1.6(iii)}.  
 Now let $i\colon A \to B$ be a level cofibration of $\I$-$G$-spaces so that $ i(\mathbb{R}^m)$ 
 is a $Com$-cofibration of $(G\times O(m))$-spaces. By the previous discussion, $\alpha^*(i(\mathbb{R}^m))$ is a $Com$-cofibration of $(K\times O(m))$-spaces. Moreover, the $O(m)$-action is unchanged, so it still acts freely off the image of $\alpha^*i$. This shows that 
	$\alpha^*$ preserves cofibrations as required.
\end{proof}

\begin{proposition}\label{prop-level-sym-mon}
	The level model structures on $\DgTop{\I}{G}$ is symmetric monoidal with 
	cofibrant unit object. 
\end{proposition}

\begin{proof}
	Let us show that the pushout-product axiom holds. By a standard reduction~\cite{Hovey}*{4.2.5}, it suffices to check 
	that the pushout product $f \square g$ is 
	\begin{itemize}
		\item[(i)] a cofibration if $f$ and $g$ belong to the set of generating cofibrations;
		\item[(ii)] an acyclic cofibration if furthermore $f$ or $g$ is a generating acyclic cofibration. 
	\end{itemize}  
	In this case we may assume $f=G\times_H \I_{V}f'$ and 
	$g=G\times_K \I_{W}g'$ and so $f \square g=\Delta^*(G\times G \times_{H\times K}\I_{V\oplus W} f'\square g')$ by Equation~(\ref{eq-tensor}). Since $\Top$ is a symmetric monoidal model category, the 
	pushout-product $f'\square g'$ satisfies conditions (i) and (ii) above. 
	By Proposition~\ref{prop-change-of-groups-functors} we see that the functors
	\[
	\Delta^* \colon \DgTop{\I}{(G\times G)} \to \DgTop{\I}{G}
	\]
	are left Quillen. Moreover, it is clear from the definition of the model structures 
	that $\mathrm{ev}_{V\oplus W}\colon \I-(G\times G)\Top \to (H\times K)\Top $ is
	right Quillen, and therefore $(G\times G)\times_{H\times K}\I_{V\oplus W}$ is left Quillen. 
	From these observations it follows that the pushout-product axiom holds for $\DgTop{\I}{G}$ too. 
	Finally, the unit axiom holds since the unit object $*=G\times_G \I_{0}$ is cofibrant. 
\end{proof}

In Section~\ref{subsec-pointed} we discussed how to induce a model structure on pointed objects. We will apply these results to the category $\DgTop{\I}{G}$ with the level model structure. Note first that the category of pointed objects in $\DgTop{\I}{G}$ is equivalent to $\DgTp{\I}{G}$, the category of continuous functors from $\I$ to $G\Tp$, the category of based $G$-spaces.

\begin{proposition}
	Let $G$ be a Lie group. The category $\DgTp{\I}{G}$ admits a \emph{proper level model 
	structure} in which the weak equivalences, fibrations and cofibrations are detected by 
	the forgetful functor $\DgTp{\I}{G} \to \DgTop{\I}{G}$. This model structure 
	is topological, cofibrantly generated by the sets $(I_G)_+$ and $(J_G)_+$, symmetric monoidal 
	and the unit object is cofibrant. Moreover, there exists a 
	symmetric monoidal equivalence of $\infty$-categories 
	\[\DgTp{\I}{G}[W^{-1}_{lvl}] \simeq (\DgTop{\I}{G}[W^{-1}_{lvl}])_*.\]
\end{proposition}

\begin{proof}
	The first part follows from the discussion in Section~\ref{subsec-pointed} and~\cite{Schwede18}*{Proposition B.5}. For the final claim apply Proposition \ref{prop-ptd-obj-model} together with the fact 
	that $\DgTop{\I}{G}[W_{lvl}^{-1}]$ is presentable by 
	Theorem~\ref{thm-presheaf-I-G-spaces}.
\end{proof}

We now change gears and consider the global analogue of the previous discussion. Recall that for any $G$-representation $V$ and $\I$-space $X$, the value $X(V)$ admits a natural $G$-action by restricting along the canonical morphism $G\to O(V)$, see Remark~\ref{rem-G-action}.

\begin{definition}\label{def-global-level}
	Let $f\colon X \to Y$ be a morphism in $\DgTop{\I}{}$.
	\begin{itemize}
		\item[(a)] We say $f$ is a \emph{faithful level equivalence} if for every compact Lie group 
		$G$ and every faithful $G$-representation $V$, the map $f(V)\colon X(V)\to Y(V)$ is 
		a $G$-weak equivalence: for all closed subgroups $H \leq G$, the induced map $f(V)^H \colon X(V)^H \to Y(V)^H$ is a weak homotopy equivalence of spaces. 
		\item[(b)] We say $f$ is a \emph{faithful level fibration} if for every compact Lie group 
		$G$ and every faithful $G$-representation $V$, the map $f(V)\colon X(V)\to Y(V)$ is a fibration in the projective model structure of $G$-spaces.
	\end{itemize}
\end{definition}

The following result is a reformulation of~\cite{Schwede18}*{Lemmas 1.2.7, 1.2.8} to our context.

\begin{lemma}\label{lem-faithful-level}
 Let $f \colon X \to Y$ be a morphism in $\DgTop{\I}{}$. 
 Then the following are equivalent:
 \begin{itemize}
 \item[(a)] the map $f(V)\colon X(V)^G\to Y(V)^G$ is a weak homotopy equivalence 
 (resp., Serre fibration) for every compact Lie group $G$ and every $G$-representation $V$;
 \item[(b)] the map $f\colon X \to Y$ is a faithful level equivalence (resp., faithful level fibration);
 \item[(c)] the map $f(\mathbb{R}^m)\colon X(\mathbb{R}^m)\to Y(\mathbb{R}^m)$ is a 
 $O(m)$-weak equivalence (resp., $O(m)$-fibration) for every $m\geq 0$.
 \end{itemize}
\end{lemma}

\begin{proof}
 It is clear that (a) implies (b), which implies (c). Suppose that (c) holds and let $V$ be a 
 $G$-representation. As in the proof of Lemma~\ref{lem-same} we can choose a linear isometric 
 isomorphism $\varphi\colon V \simeq \mathbb{R}^m$ and define a group homomorphism 
 $\rho \colon G \to O(m)$ such that 
 \[
 X(V)^G \simeq X(\mathbb{R}^m)^{\rho(G)}
 \]
 showing that (c) implies (a).
\end{proof}

\begin{construction}[semifree $\I$-space]\label{con-semifree}
	For every $G$-representation $V$, there is an evaluation functor 
	\[
	\mathrm{ev}_{G,V}\colon \DgTop{\I}{}\to G\Top, \qquad X\mapsto X(V)
	\]
	which admits a left adjoint $\Li_{G,V}$ given by the formula 
	$\Li_{G,V}(A)=\Li(V,-)\times_G A$.  
	When $A=*$, we simply write $\Li_{G,V}$. For all $H$-representations 
	$V$ and $K$-representations $W$, there is an isomorphism of $\I$-$G$-spaces
	\begin{equation}\label{tensor-product-semifree}
		\Li_{H,V} \otimes \Li_{K,W}\cong \Li_{H\times K,V\oplus W}.
	\end{equation}
	One can check this using the formula in Remark~\ref{rem-day-convolution} or by mimicking 
	the proof of~\cite{Schwede18}*{Example 1.3.3}.
\end{construction}

The next result is an analogue of~\cite{Schwede18}*{Proposition 1.2.10}, adapted to our context.

\begin{theorem}\label{thm-level-global}
	The category $\DgTop{\I}{}$ admits a topological, cofibrantly generated  
	model structure in which the weak equivalences are the faithful level equivalences 
	$W_{f-lvl}$ and the fibrations are the faithful level fibrations. 
	The set of generating cofibrations $I$ and acyclic cofibrations $J$ are given by 
	\begin{align*}
		I= & \{\Li_{G,V} ( \partial D^n) \to \Li_{G,V}( D^n) \}\\
		J =& \{ \Li_{G,V} ( D^n\times \{0\}) \to \Li_{G,V}  ( D^n \times [0,1]) \}
	\end{align*}
	where $G$ runs over all compact Lie groups, $V$ over all faithful $G$-representations 
	and $n\geq 0$. This is a symmetric 
	monoidal model category with cofibrant unit object. 
	We call this the \emph{faithful level model structure}. 
\end{theorem}

\begin{proof}
    We can identify $\DgTop{\I}{}$ with the category $\prod_{m\geq 0}O(m)\Top$ and endow the latter category with the product of the standard model structures on $O(m)$-spaces. The induced model structure on $\DgTop{\I}{}$ has weak equivalences and fibrations as in the theorem by Lemma~\ref{lem-faithful-level}.
	We note that the right lifting property against the sets $I$ and $J$ detect the level fibrations 
	and level acyclic fibrations respectively, by the adjunction isomorphism
	\[
	\Hom_{\DgTop{\I}{}}(\I_{H,V}A, X)\simeq \Hom_{\Top}(A, X(V)^H)
	\]
	for $A$ a non-equivariant space. 	
	Let us next show that the pushout-product axiom holds. 
	As explained in the proof of Proposition~\ref{prop-level-sym-mon}, it suffices to check 
	that the pushout product $f \square g$ is an (acyclic) cofibration if $f$ and $g$ belong to 
	the set of generating (acyclic) cofibrations. In any case we have $f=\Li_{G,V}f'$ and 
	$g=\I_{H,W}g'$. But then $f \square g=\Li_{G\times H, V\oplus W}f'\square g'$ 
	by Equation~(\ref{tensor-product-semifree}).
	Since $G\Top$ is a symmetric monoidal model category, it suffices to check that the 
	functor $\I_{G\times H, V\oplus W}$ is left Quillen. This is clear since 
	$\mathrm{ev}_{G\times H, V\oplus W}$ is right Quillen by definition of the faithful level 
	model structure. The pushout-product axiom then follows. 
	Finally, the unit axiom holds since the unit object $*=\I_{e,0}$ is cofibrant and 
	the model structure is topological by~\cite{Schwede18}*{Proposition B.5}. 
\end{proof}
As before we obtain an induced model structured on pointed objects.

\begin{proposition}
	The category $\DgTp{\I}{}$ admits a faithful level model 
	structure in which the weak equivalences, fibrations and cofibrations are detected 
	by the forgetful functor $\DgTp{\I}{}\to \DgTop{\I}{}$. This model structure is topological,
	cofibrantly generated by the set $I_+$ and $J_+$, symmetric monoidal and the unit object 
	is cofibrant. Finally, there exists a symmetric 
	monoidal equivalence of $\infty$-categories 
	\[\DgTp{\I}{}[W^{-1}_{f-lvl}] \simeq (\DgTop{\I}{}[W^{-1}_{f-lvl}])_*.\]
\end{proposition}

\begin{proof}
	The first two claims follows from the discussion in Section~\ref{subsec-pointed} and
	\cite{Schwede18}*{Proposition B.5}. For the final claim apply Proposition~\ref{prop-ptd-obj-model}, using the fact that $\DgTop{\I}{}[W^{-1}_{f-lvl}]$ is presentable. We will show this in Theorem~\ref{thm-presheaf-I-spaces}.
\end{proof}

We now pass from pointed objects to pre-spectrum objects. Observe that the category of pointed $\I$-$G$-spaces has a commutative algebra object $S_G$ given by the functor sending $V$ to its one-point compactification $S^V$ equipped with the trivial $G$-action. If we are thinking of the category of $\I$-spaces with the faithful level model structure, we will write $S_{fgl}$ for $S_e$, to emphasize that the sphere should be thought of as evaluated on all faithful representations of all groups ($fgl$ stands for faithful global). 

\begin{definition}\label{def-prespectra}
	Let $G$ be a Lie group. Following \cite{MandellMay}*{Chapter II Proposition 3.8}, we define 
	the topological category $\Sp^O_G$ of orthogonal $G$-spectra to be the category of 
	$S_G$-modules in $\DgTp{\I}{G}$. These categories inherit induced model structures:
	\begin{itemize}
		\item[(a)] The category of orthogonal $G$-spectra admits a \emph{(proper) level model structure} whose weak 
		equivalences and fibrations are created by the forgetful functor $\Sp_G^O \to \DgTp{\I}{G}$ where the target is endowed with the level model structure. This is a cofibrantly generated, proper, topological model category, see the proof of~\cite{Proper}*{Theorem 1.2.22}. We also obtain that a set of generating cofibrations and acyclic cofibrations are given by the maps $S_G \otimes I_G$ and $S_G \otimes J_G$ where $S_G \otimes-$ denotes the left adjoint to the forgetful functor $\Sp_G^O \to \DgTp{\I}{G}$.
		
		\item[(b)] The category of orthogonal spectra admits a \emph{faithful level model structure} 
		whose weak equivalences and fibrations are created by the forgetful functor $\Sp^O \to \DgTp{\I}{}$ where the target is endowed with the faithful level model structure, see ~\cite{Schwede18}*{Propositions 4.3.5}. From this result we obtain that the faithful level model structure is cofibrantly generated and topological, with a set of generating cofibrations and acyclic cofibrations given by $\Sfgl \otimes I$ and $\Sfgl \otimes J$, where $\Sfgl \otimes-$ denotes the left adjoint to the forgetful functor $ \Sp^O \to \DgTp{\I}{}$.
		\end{itemize}
\end{definition}

\begin{remark}
By combining straightforward generalizations of ~\cite{MandellMay}*{Theorem 4.3} and ~\cite{Schwede18}*{Remark 3.1.8} to Lie groups, we conclude that $\Sp^O_G$ is equivalent to the category of orthogonal spectra defined in \cite{Proper}*{Definition 1.1.9}. 
\end{remark}

As discussed in~\cite{MandellMay}*{Chapter II Section 3}, the category of orthogonal $G$-spectra admits a closed symmetric monoidal structure.

\begin{proposition}
	Let $G$ be a Lie group.
	\begin{itemize}
		\item[(a)]
		The level model structure on $\Sp_G^O$ is symmetric monoidal.
		\item[(b)] The faithful level model structure on $\Sp^O$ is symmetric monoidal. 
	\end{itemize}
\end{proposition}

\begin{proof}
	The proof that the pushout product axiom holds for $\Sp_G^O$ is similar to that given in Proposition~\ref{prop-level-sym-mon} for 
	$\I$-$G$-spaces. The explicit argument for cofibrations can be found 
	in~\cite{Proper}*{Proposition 1.2.28(i)} and we note that a slight modification of that argument then 
	also gives the statement for acyclic cofibrations. 
	The argument that the faithful level model structure satisfies the pushout-product axiom is similar 
	to that given in Theorem~\ref{thm-level-global}. The argument for cofibrations 
	can also be found in~\cite{Schwede18}*{Proposition 4.3.23} and a slight modification of that argument also 
	gives the statement for acyclic cofibrations.
\end{proof}

\begin{definition}
	We define the $\infty$-category $\PSp_G$ of \emph{$G$-prespectra} to be the symmetric monoidal $\infty$-category associated to the symmetric monoidal model category $\Sp^O_G$ with the level model structure.
	Similarly, we define the $\infty$-category $\PSpfgl$ of \emph{faithful global prespectra} to be the symmetric monoidal $\infty$-category associated to the symmetric monoidal model category $\Sp^O$ with the faithful level model structure.
\end{definition}

We have emphasized how the level model structures on $\Sp_G^O$ and $\Sp^O$ are induced by the level model structure on $\DgTp{\I}{G}$ and $\DgTp{\I}{}$ respectively by taking modules. This allows us to reinterpret the passage to modules internally to $\infty$-categories.

\begin{proposition}\label{cor:PSp_first}
	There are symmetric monoidal equivalences 
	\[
	\PSp_G \simeq \Mod_{S_G}(\DgTop{\I}{G}[W_{lvl}^{-1}]_{*}) \quad \mathrm{and} \quad 
	\PSpfgl\simeq \Mod_{\Sfgl}(\DgTop{\I}{}[W_{f-lvl}^{-1}
	]_{*}).
	\]
\end{proposition}

\begin{proof}
	Apply Proposition \ref{prop:Modules_localization_commute}.
\end{proof}

Finally we pass from the level model structure to the stable model structure, which will present the categories of global and genuine $G$-spectra. Fix a complete $G$-universe $\mathcal{U}_G$ and write $s(\mathcal{U}_G)$ for   
the poset, under inclusion, of finite dimensional $G$-subrepresentations of 
$\mathcal{U}_G$. The $G$-equivariant homotopy groups of an orthogonal 
$G$-spectrum $X$ are given by
\[
\pi^G_k(X) = 
\begin{cases}
	\displaystyle\operatorname*{colim}_{V \in s(\mathcal{U}_G)} \;\; [S^{k+V}, X(V)]_*^G & \mathrm{for} \;\; k \geq 0 \\
	\displaystyle\operatorname*{colim}_{{V \in s(\mathcal{U}_G)}} \;\; [S^{V}, X(\mathbb{R}^{-k} \oplus V)]_*^G & \mathrm{for}\;\; k \leq 0
\end{cases}
\]
where the connecting maps in the colimit system are induced by the structure 
maps, and $[-,-]_*^G$ means $G$-equivariant homotopy classes of 
based $G$-maps. 
Note that the same definition works even if $X$ is an 
orthogonal spectrum since the value $X(V)$ admits a $G$-action as discussed before 
Definition~\ref{def-global-level}. Moreover, everything is functorial with respect to 
morphisms of orthogonal ($G$-)spectra. We finally note that the definition above a priori depends on a choice of complete $G$-universe. However the functors associated to different complete $G$-universes are naturally isomorphic, and so the choice is immaterial.

\begin{definition}\label{def-SpG-Spgl}
	Let $G$ be a Lie group. 
	\begin{itemize}
		\item A morphism $f\colon X \to Y$ of orthogonal $G$-spectra is a 
		$\underline{\pi}_*$-\emph{isomorphism} if $\pi_*^H(f)\colon \pi_*^H(X)\to \pi_*^H(Y)$ 
		is an isomorphism for all compact subgroups $H\leq G$. 
		The $\underline{\pi}_*$-isomorphisms are part of a cofibrantly generated, 
		topological, stable and symmetric monoidal model structure on the category of 
		orthogonal $G$-spectra~\cite{Proper}*{Theorem 1.2.22}, called the $G$-\emph{stable 
			model structure}. 
		
		\item A morphism $f\colon X \to Y$ of orthogonal spectra is a 
		\emph{global equivalence} if $\pi_*^H(f)\colon \pi_*^H(X)\to \pi_*^H(Y)$ 
		is an isomorphism for all compact Lie groups $H$. 
		The global equivalences are part of a cofibrantly generated, 
		topological, proper, stable and symmetric monoidal model structure on the category of 
		orthogonal spectra~\cite{Schwede18}*{Theorem 4.3.17, Proposition 4.3.24}, called 
		the \emph{global model structure}.
	\end{itemize}
\end{definition}

\begin{definition}
	We define the symmetric monoidal $\infty$-category $\Sp_G$ of \emph{$G$-spectra} to be the underlying $\infty$-category of orthogonal $G$-spectra with the $G$-stable model structure.
	Similarly, we define the symmetric monoidal $\infty$-category $\Spgl$ of \emph{global spectra} to be the underlying $\infty$-category of orthogonal spectra with the global model structure. 
\end{definition}

We now make precise the observation that $\Sp_G$ and $\Spgl$ are Bousfield localizations of $\PSp_G$ and $\PSpfgl$ respectively at an explicit collection of weak equivalences. We begin with global spectra.

\begin{construction}\label{cons-lambda-maps}
	Given a compact Lie group $G$ and a $G$-representation $V$, we can consider the adjoint 
	pairs
	\[
	\begin{tikzcd}[column sep = large]
		\Sp^O \arrow[r,shift left,"\mathrm{forget}" ] & \arrow[l, shift left,"S_{gl}\otimes-"] \DgTp{\I}{}
		\arrow[r, shift left, "\mathrm{ev}_{G,V}"]& \arrow[l, shift left, "\Li_{G,V}"] G\Tp \,.
	\end{tikzcd}
	\]
	Following~\cite{Schwede18}*{Construction 4.1.23}, we denote the composite 
	$S_{gl}\otimes \Li_{G,V}$ by $F_{G,V}$. Note that the adjoint pairs above are Quillen with 
	respect to the global level structure and so they yield corresponding adjoint pairs of 
	underlying $\infty$-categories.
	As discussed before~\cite{Schwede18}*{Theorem 4.1.29}, there are maps in $\Sp^O$
	\[
	\lambda_{G,V,W}\colon F_{G,V\oplus W}S^V \to F_{G,W}S^0
	\]
	for all compact Lie groupw $G$ and $G$-representations $V$ and $W$. Note that we can view 
	these maps in $\PSp_{gl}$ since the domain and codomain of $\lambda_{G,V,W}$ are bifibrant. Consider the following diagram
	\[
	\begin{tikzcd}
		{G\Tp(S^0,X(W))} \arrow[r, "{\tilde{\sigma}_{G,V,W}}"] & {G\Tp(S^V,X(V\oplus W))} \arrow[d, "\sim"] \\
		{ \Sp^O(F_{G,W} S^0,X)} \arrow[u,"\sim"] \arrow[r]                     & { \Sp^O(F_{G,V\oplus W} S^V,X)},            
	\end{tikzcd}
	\] where the vertical maps are the adjunction isomorphisms and the top map is the adjoint structure map of $X$. The bottom map is equal to precomposition by $\lambda_{G,V,W}$. In particular, taking $X = F_{G,W} S^0$, we may define $\lambda_{G,V,W}$ as the image of the identity of $F_{G,W}S^0$ under the bottom map. Note also that $\lambda_{G,V,W}$ is equivalent to $F_{G,W} S^0 \otimes \lambda_{G,V,0}$, and that $\lambda_{G,V,0}$ is adjoint to the identity.
\end{construction}

\begin{remark}
	Observe that both characterizations of $\lambda_{G,V,W}$ given above also uniquely specify the map on the level of $\infty$-categories.
\end{remark}

\begin{proposition}\label{prop-glspectra-local-object}
	$\Spgl$ is a Bousfield localization of $\PSpfgl$. Furthermore, an object in $\PSpfgl$ lies in $\Spgl$ if and only if it is local with respect to 
	the morphisms $\{\lambda_{G,V,W}\}$ for all compact Lie groups $G$ and $G$-representations $V$ and $W$ with $W$ faithful. 
\end{proposition}

\begin{proof}
	Let $\Lambda$ denote the set of maps $\lambda_{G,V,W}$ for $G$, $V$ and $W$ as in 
	the proposition. We write $\Sp^O_{lvl}$ and $\Sp^O_{gl}$ for the category of orthogonal spectra endowed with the faithful level model structure and the global stable model structure respectively. We will show that $\Sp^O_{gl}$ is a left Bousfield localization (in the model categorical sense) of $\Sp_{lvl}^O$ at the set $\Lambda$, that is $L_{\Lambda}\Sp^O_{lvl}=\Sp^O_{gl}$. Because both can be checked on underlying homotopy categories, Bousfield localizations of model categories present Bousfield localizations of $\infty$-categories. Therefore the claim in the proposition will follow by passing to underlying $\infty$-categories.   
	By definition $X \in \Sp^O_{lvl}$ is $\Lambda$-local (and so fibrant in the Bousfield localization) if and only if $X$ is fibrant in $\Sp_{lvl}^O$ (which always holds in this case), and the canonical map of homotopy function complexes
	\[
	\lambda^*_{G,V,W}\colon \Map(F_{G,W}S^0, X)\to \Map(F_{G,V\oplus W}S^V, X)
	\]
	is an equivalence for all $\lambda_{G,V,W}\in \Lambda$. By adjunction this is equivalent to asking that $X(W)^G\to \Omega^V(X(V\oplus W))^G$ is an equivalence for all $G,V$ and $W$ as in the proposition. In other words $X$ is a global $\Omega$-spectrum, 
	see~\cite{Schwede18}*{Definition 4.3.8}. By~\cite{Schwede18}*{Theorem 4.3.17} these are precisely the fibrant objects $\Sp^O_{gl}$. Since $L_{\Lambda}\Sp_{lvl}^O$ and $\Sp_{gl}^O$ have the same cofibrations and fibrant objects, the two model structure coincide by~\cite{Joyal}*{Proposition E.1.10}.
\end{proof}

We repeat this analysis for $\Sp_G$ and $\PSp_G$.
\begin{construction}\label{cons-lambda-maps-equivariant}
	Let $H$ be a compact subgroup of a Lie group $G$, and let $V$ be an $H$-representation.
	We have a sequence of adjoint pairs
	\[
	\begin{tikzcd}[column sep =  large]
		\Sp^O_G \arrow[r,shift left,"\mathrm{forget}" ] & \arrow[l, shift left,"S_{G}\otimes-"] \DgTp{\CI}{G}
		\arrow[r, shift left, "\mathrm{ev}_{V}"]& \arrow[l, shift left, "G_+\wedge_H\Li_{V}"] H\Tp 
	\end{tikzcd}
	\]
	which are Quillen with respect to the proper level model structure, and so they 
	define adjoint pairs at the level of underlying $\infty$-categories. The composite 
	$S_G \otimes (G_+\wedge_H \Li_V)$ will also be denoted by $G \ltimes_H F_V$ following~\cite{Proper}*{Example 1.1.15}.  This notation is justified by the fact that 
	$G \ltimes_H F_V$ is also equivalent to the induction of the $H$-prespectrum $F_V$ 
	as one can easily verify.  
	For all pairs of $H$-representations $V$ and $W$, there are 
	maps in $\Sp_G^O$ 
	\[
	G \ltimes_{H}\lambda_{V,W} \colon G \ltimes_H F_{V\oplus W}S^V \to G \ltimes_H F_W,
	\]
	see~\cite{Proper}*{Equation 1.2.19}. We can view these maps in $\PSp_G$ as the domains and 
	codomains are bifibrant. Similarly to before, $G \ltimes_{H}\lambda_{V,W}$ is determined by the property that the map \[{\Sp^O_G(G \ltimes_H F_W,X)} \rightarrow {\Sp_G^O(G \ltimes_H F_{V\oplus W}S^V, X)},\] defined so that the diagram
	\[
	\begin{tikzcd}
		{H\Tp(S^0,X(W))} \arrow[r, "{\mathrm{res}^G_H(\tilde{\sigma}_{V,W})}"] & {H\Tp(S^V,X(V\oplus W))} \arrow[d, "\sim"] \\
		{\Sp^O_G(G \ltimes_H F_W,X)} \arrow[u, "\sim"] \arrow[r] & {\Sp_G^O(G \ltimes_H F_{V\oplus W}S^V ,X)}          
	\end{tikzcd}
	\] commutes, is equal to precomposition by $G \ltimes_{H}\lambda_{H,V,W}$. Also note that $G\ltimes \lambda_{V,W}$ is equal to $G\ltimes_H F_{W}S^0 \otimes \lambda_{V,0}$ and that $\lambda_{V,0}$ is adjoint to the identity on $S^V$.
\end{construction}

\begin{remark}
	Once again, observe that the characterizations of $G\ltimes_H \lambda_{V,W}$ given above also uniquely specify the map on the level of $\infty$-categories.
\end{remark}

\begin{proposition}\label{prop-Sp_G-local}
	Let $G$ be a Lie group. Then $\Sp_G$ is a Bousfield localization of $\PSp_G$. Furthermore, an object in $\PSp_{G}$ lies in $\Sp_G$ if and only if it is local with respect to 
	the morphisms $\{G \ltimes_H \lambda_{V,W}\}$ for all compact subgroups $H\leq G$ 
	and $H$-representations $V$ and $W$. Equivalently, $X\in\PSp_G$ lies in $\Sp_G$ if 
	and only if for all compact subgroups $H\leq G$, the object 
	$\mathrm{res}^G_H X\in\PSp_H$ is local with respect to morphisms 
	$\{\lambda_{V,W}\}$ for all $H$-representations $V$ and $W$.
\end{proposition}

\begin{proof}
	The proof is similar to that of Proposition~\ref{prop-glspectra-local-object} but now we use  the characterization of fibrant objects in the proper stable model structure given in~\cite{Proper}*{Theorem 1.2.22 (v)}. The second claim follows from the first one by adjunction.
\end{proof}

\section{Models for \texorpdfstring{$\infty$}{infty}-categories of equivariant prespectra}\label{sec:eq_prespectra}

In the previous section we introduced the $\infty$-categories of equivariant and global (pre)spectra, and exhibited the spectrum objects as local objects in the relevant category of prespectra with respect to an explicit class of weak equivalences. Furthermore, we observed that the construction of $\PSp_G$ admitted a reinterpretation internal to $\infty$-categories, by first passing to pointed objects in $\DgTop{\I}{G}[W_{lvl}^{-1}]$ and then taking modules over $S_G$. Similarly, we observed that
\[\PSpfgl\simeq \Mod_{\Sfgl}(\DgTop{\I}{}[W_{f-lvl}^{-1}]_{*}).\] 
Furthermore these equivalences were symmetric monoidal.

However this is only part of the story, because the $\infty$-categories $\DgTop{\I}{G}[W_{lvl}^{-1}]$ and $\DgTop{\I}{}[W_{f-lvl}^{-1}]$ are still too inexplicit for our arguments. Luckily we can give explicit models of these $\infty$-categories. Consider the case of $\DgTop{\I}{G}[W_{lvl}^{-1}]$. By construction this $\infty$-category records the fixed point spaces $X(V)^H$ for every (compact) subgroup $H$ of $G$ and every $H$-representation $V$ of an $\I$-$G$-space $X$. By functoriality, these different fixed point spaces are related by subconjugacy relationships in $H$ and equivariant linear isometries in $V$. We will prove that the $\infty$-category $\DgTop{\I}{G}$ is in fact freely generated under these properties. More precisely, we will exhibit an equivalence 
\[\DgTop{\I}{G}[W_{lvl}^{-1}]\simeq \DgSpc{\OR_G},\] where the $\infty$-category $\OR_G$ indexes pairs $(H,V)$, each one of which records one of the fixed point spaces $X(V)^H$ of an $\I$-$G$-space $X$. Similarly we will prove that 
\[\DgTop{\I}{}[W_{f-lvl}^{-1}]\simeq \DgSpc{\ORfgl},\]
where the $\infty$-category $\ORfgl$ indexes pairs $(G,V)$, where $G$ is a compact Lie group and $V$ is a faithful $G$-representation.

In total we will obtain equivalences
\[\PSp_G\simeq \Mod_{S_G}(\DgSpcp{\OR_G})\quad \text{and}\quad \PSpfgl\simeq \Mod_{\Sfgl}(\DgSpcp{\ORfgl}).\] It will be in this guise that we will think of the $\infty$-category of $G$-prespectra and global prespectra for the remainder of the paper. 

Finally, to make future constructions symmetric monoidal it will be important to understand how the symmetric monoidal structures transfer under the equivalences
\[\DgTop{\I}{G}[W_{lvl}^{-1}]\simeq \DgSpc{\OR_G} \quad \text{and}\quad \DgTop{\I}{}[W_{f-lvl}^{-1}]\simeq \DgSpc{\ORfgl}.\] We may immediately apply Theorem \ref{thm-I-promonoidal} to conclude that the monoidal structure on $\DgTop{\I}{G}[W_{lvl}^{-1}]$ and $\DgTop{\I}{}[W_{f-lvl}^{-1}]$ are induced by Day convolution from the restricted promonoidal structure on $\OR_G$. We will make these promonoidal structures explicit.

To show that $\DgTop{\I}{G}[W_{lvl}^{-1}]$ and $\DgTop{\I}{}[W_{lvl}^{-1}]$ are equivalent to categories of copresheafs on an explicit set of generators, we will apply a version of Elmendorf's theorem, see Corollary~\ref{cor-elmendorf}. The application of this theorem to $\DgTop{\I}{G}[W_{lvl}^{-1}]$ and $\DgTop{\I}{}[W_{f-lvl}^{-1}]$ has a similar flavour, but are logically distinct. Therefore we treat each case separately.

\subsection{\texorpdfstring{$\CI$}{CI}-\texorpdfstring{$G$}{G}-spaces and \texorpdfstring{$\OR_G$}{ORG}-spaces} 

We begin with $\DgTop{\I}{G}[W_{lvl}^{-1}]$.

\begin{remark}\label{rem-morphism-O_G}
	Let $G$ be a Lie group and consider a map $\varphi\colon G/K \to G/H$ in the 
	orbit category $\OO_G$. Giving $\varphi$ is equivalent to giving $gH\in (G/H)^K$, that 
	is an element $gH \in G/H$ such that $c_g(K)=g^{-1}K g \subseteq H$. 
	When we need to emphasize this correspondence between $gH$ and $\varphi$ we will use 
	subscripts $\varphi_{g}$ and $g_{\varphi}$. 
	Note that $g_{\psi\circ \varphi}H=g_\varphi g_\psi H$ so composition of maps 
	corresponds to multiplication with reverse order. 
\end{remark}

\begin{definition}\label{def-O_G-proper}
	For a Lie group $G$, the \emph{proper} $G$-\emph{orbit category} $\OO_{G,\pr}$ is 
	the full subcategory of $\OO_G$ spanned by those cosets $G/H$ with $H \leq G$ compact. 
\end{definition}

Let $G$ be a Lie group and $H,K \leq G$ be compact subgroups. 
Given an $H$-representation $V$ and a $K$-representation $W$, we can consider the 
space $G\times_H \Li(V,W)$ where $H$ acts on $G$ 
by right translation, and on $\Li(V,W)$ via $h.\varphi =\varphi h^{-1}$. 
Note that $K$ acts diagonally on $G\times_H \Li(V,W)$ via $G$ and $W$. 
We have the following helpful criterion. 

\begin{lemma}\label{lem-criterion}
	An element $[g,\varphi]\in G\times_H \Li(V,W)$ is $K$-fixed if and only if $c_g(K)\subseteq H$ and $k.\varphi(v)=\varphi(c_g(k)v)$ for all $k\in K$ and $v\in V$.
\end{lemma} 

\begin{proof}
	An element $[g,\varphi]\in G\times_H \Li(V,W)$ is $K$-fixed 
	if and only if $[kg, k.\varphi]=[g,\varphi]$ for all $k\in K$. This means that there 
	exists $h\in H$ such that $kg=gh$ and $k.\varphi =\varphi h$ for all $k\in K$. 
	In other words $g$ is such that $c_g(K)\subseteq H$ and $\varphi$ is $K$-equivariant in 
	the sense that $k.\varphi=\varphi c_g(k)$ for all $k\in K$. 
\end{proof}

\begin{lemma}\label{lem-composition-well-def}
	Let $G$ be a Lie group and $H,K,L \leq G$ be compact subgroups. Let $V$ 
	be an $H$-representation, $W$ a $K$-representation and $U$ an $L$-representation. 
	Then the map 
	\[
	\circ \colon (G\times_K \Li(W,U))^L \times (G\times_H \Li(V,W))^K \to 
	(G\times_H \Li(V,U))^L 
	\] 
	given by $([g',\psi], [g,\varphi]) \mapsto [g'g,\psi\varphi]$ is well-defined and continuous. Furthermore, upon varying the objects, the collection of maps so obtained is associative and unital.
\end{lemma}

\begin{proof}
	Let us first show that the map does not depend on the chosen representatives. For 
	$h\in H$ and $k\in K$ we have $[g,\varphi]=[gh,\varphi h]$ and 
	$[g',\psi]=[g'k,\psi k]$ so we ought to check that 
	$[g'g,\psi\varphi]=[g'kgh,\psi k \varphi h]$. Using that
	$c_g(K)\subseteq H$ and $\varphi$ is $K$-equivariant with respect to 
	the $c_g$-twisted action, we can write 
	\[
	[g'kgh,\psi k \varphi h]=[g'g\underbrace{c_g(k)h}_{\in H}, \psi k \varphi h]=[g'g, \psi k \varphi h (c_{g}(k)h)^{-1}]=[g'g, \psi k \varphi c_{g^{}}(k^{-1})]=[g'g,\psi\varphi]
	\]
	as required. We verify that $[g'g, \psi\varphi]$ is $K$-fixed using the criterion 
	from Lemma~\ref{lem-criterion}. 
	Using that  $c_{g'}(L)\subseteq K$  and $c_g(K)\subseteq H$ we immediately see that 
	$c_{g'g}(L)\subseteq H$. 
	Using the twisted equivariance of $\psi$ and $\varphi$ we see that
	\[
	l. \psi \varphi=\psi \underbrace{c_{g'}(l)}_{\in K}\varphi=\psi \varphi c_{g}(c_{g'}(l))=\psi \varphi c_{g'g}(l)
	\]
	for all $l \in L$. Therefore $\psi \varphi$ is twisted equivariant and 
	$[g'g,\psi\varphi]$ is indeed $K$-fixed. Finally the map is associative, unital 
	and continuous since multiplication and composition maps are so. 
\end{proof}

We now formally define the $\infty$-category $\OR_G$.

\begin{definition}\label{def-OR_G}
	Let $G$ be a Lie group. We define a topological category $\OR_G$ whose objects 
	are pairs $(H,V)$ of a compact subgroup $H\leq G$ and an $H$-representation $V$. 
	The morphism spaces are given by 
	\[
	\OR_G((H,V), (K,W))=(G \times_H \Li(V,W))^K.
	\]
	Composition is given by the maps
	\[
	\circ \colon \OR_G((K,W),(L,U)) \times \OR_G((H,V),(K,W)) \to \OR_G((H,V), (L,U))
	\]
	defined in Lemma~\ref{lem-composition-well-def}. Note that there is a 
	projection map 
	\[
	\OR_G((H,V), (K,W)) \to (G/H)^K= \OO_{G,\pr}(G/K, G/H), \qquad [g, \varphi]\mapsto [gH].
	\]
	which extends to a functor $\pi_G\colon \OR_G \to \OO_{G,\pr}^{\mathrm{op}}$.
\end{definition}

\begin{example}
	Let $G=e$ be the trivial group. Then the topological category $\OR_G$ is equivalent to 
	$\I$.
\end{example}

\begin{example}\label{ex-L}
	By definition $\OR_G((H,V), (e,W))=G \times_H \Li(V,W)$, which is a space with an 
	action of 
	\[
	\OR_G((e,W),(e,W))=G \times O(W).
	\]
	One can identify the functor $\OR_G((H,V), (e, -))\colon \I \to G\Top$ with the free
	$\I$-$G$-space $G\times_H \Li_{V}$.
\end{example}

\begin{definition}
	We let $\DgSpc{\OR_G}$ denote the $\infty$-category of $\OR_G$-spaces, given by the $\infty$-category of functors $\OR_G \to \Spc$. 
\end{definition}

We are finally ready to prove the main result of this subsection.

\begin{theorem}\label{thm-presheaf-I-G-spaces}
	Let $G$ be a Lie group. Then there is an equivalence of $\infty$-categories 
	\[
	\DgTop{\I}{G}[W_{lvl}^{-1}]\simeq \DgSpc{\OR_G}.
	\]
\end{theorem}

\begin{proof}
	The discussion in Example~\ref{ex-L} shows that there exists a functor of topological 
	categories (and so of $\infty$-categories)
	\[
	\OR_G^{\op} \to \DgTop{\I}{G}, \qquad (H,V)\mapsto \OR_G((H,V), (e,-))=G\times_H \Li_V.
	\]
	This is fully faithful by definition of $\OR_G$. Since the $\I$-$G$-spaces 
	$G\times_H \Li_{V}$ are bifibrant in the level model structure, 
	the composite 
	\[
	L \colon \OR_G^{\op}\to \DgTop{\I}{G} \to \DgTop{\I}{G}[W_{lvl}^{-1}], \quad (H,V)\mapsto G \times_H \Li_V
	\]
	is also fully faithful. We apply Theorem~\ref{thm-elmendorf} to the functor $L$.	
	We note that the $\I$-$G$-space $G\times_H \Li_V$ corepresents the functor 
	$X \mapsto X(V)^H$. This functor commutes with 
	small homotopy colimits since:
	\begin{itemize}
	\item the $H$-fixed points functor preserves small homotopy colimits as discussed 
	in Example~\ref{ex-S_G};
	\item  and the evaluation functor $X \mapsto X(V)$ preserves small homotopy colimits. 
	Indeed this functor preserves all colimits (as they are calculated pointwise), level 
	equivalences by definition, and (acyclic) cofibrations (as one can verify by 
	checking on the generating (acyclic) cofibrations).
	\end{itemize}
	Finally, the collection of objects $\{G\times_H \Li_V \mid (H,V)\in\OR_G\}$ is jointly 
	conservative by definition of the level equivalences. Thus the required equivalence 
	follows from Theorem~\ref{thm-elmendorf}.
\end{proof}

Next we explain how to upgrade the equivalence above to an equivalence of symmetric monoidal 
$\infty$-categories.

\begin{construction}
	We enhance the topological category $\OR_G$ to a topological coloured operad as follows. The 
	colours are simply the objects of $\OR_G$, and the space of multi-morphisms from $\{(H_i,V_i)\}_{i\in I}$ to $(K,W)$ is given by
	\[
	\OR_G(\{(H_i,V_i)\}_{i\in I}, (K,W))= \left(\left(\prod_{i \in I} G\right) \times_{(\prod_{i\in I}H_i)} \Li\left(\bigoplus_{i\in I} V_i, W\right) \right)^{K}.
	\]
	By Lemma~\ref{lem-criterion}, a point of this space is equivalent to the following data:
	\begin{itemize}
		\item For all $i \in I$, an element $g_{i}H_i \in G/H_i$ such that 
		$c_{g_{i}}(K)\subseteq H_{i}$;
		\item A linear isometry
		$\varphi=\sum_i \varphi_{i}\colon \bigoplus_i V_i \to W$ such 
		that $k .\varphi_{i}(v)=\varphi_{i}(c_{g_{i}}(k)v)$ for all $v \in V_i$, $k \in K$ and 
		$i\in I$. 
	\end{itemize}
	For every map $I \to J$ of finite sets with fibres $\{I_j\}_{j \in J}$, every finite collections of 
	objects $\{(H_i, V_i)\})_{i \in I}$ and $\{(K_j,W_j)\}_{j \in J}$, and every 
	$(L,U)\in \OR_G$ we have a composition map 
	\[
	\prod_{j \in J} \OR_G(\{(H_i,V_i)\}_{i\in I_j}, (K_j,W_j))\times 
	\OR_G(\{(K_j,W_j)\}_{j\in J}, (L,U)) \to\OR_G(\{(H_i,V_i)\}_{i\in I},(L,U))
	\]
	which is defined by the formulas
	\[
	(\bigoplus_{i \in I_j} V_i \to W_j, \bigoplus_{j \in J} W_j \to U)\mapsto (\bigoplus_{i \in I}V_i=\bigoplus_{j \in J}\bigoplus_{i \in I_j} V_i \to \bigoplus_{j \in J} W_j \to U)
	\]
	and 
	\[
	((g_iH_i)_{i \in I_j}, (g_{j}K_j)_{j\in J}) \mapsto (g_jg_iH_i)_{j\in J, i \in I_j}. 
	\]
	Note that for any colour $(H,V)\in \OR_G$, there is an identity element 
	$[eH, 1_V]\in \OR_G((H,V),(H,V))$.
	Using Lemma~\ref{lem-criterion} one can check that this composition is continuous, 
	associative and unital and so that $\OR_G$ is indeed a topological coloured operad. 
	We leave the details to the interested reader. 
\end{construction}

\begin{remark}\label{rem-pi_G}
	We can endow the topological category $\OO_{G,\pr}^{\op}$ with a topological coloured 
	operad structure whose colours are the objects of $\OO_{G,\pr}$, and whose multimorphism spaces are given by 
	\[
	\OO_{G,\pr}(\{G/H_i\}_{i \in I}, G/K)= \OO_{G,\pr}(G/K, \prod_{i \in I} G/H_i)= (\prod_{i \in I}G/H_i)^K
	\]
	with composition defined in the obvious way. The associated $\infty$-operad models the cocartesian monoidal structure. There is a canonical projection functor
	of topological coloured operads
	\[
	\pi_G \colon\OR_G \to \OO_{G,\pr}^{\op}.
	\]
	By Lemma~\ref{ORG_operad}, we can lift $\pi_G$ to a map of $\infty$-operads 
	$\OR_G^\otimes \to (\OO_{G,\pr}^{\op})^\amalg$, which by abuse of notation we 
	still denote by $\pi_G$. 
\end{remark}

Recall that because $\DgTop{\I}{G}$ is a symmetric monoidal topological model category, we can construct a topological coloured operad whose colors are given by the bifibrant objects of $\DgTop{\I}{G}$ and the multimorphism spaces are given by 
\[
\Mul_{N^\otimes((\DgTop{\I}{G}^\circ)^{\op})}(\{X_i\}, Y)=\DgTop{\I}{G}(Y, \bigotimes_{i\in I} X_i).
\] Furthermore the associated $\infty$-operad models the symmetric monoidal structure on $(\DgTop{\I}{G}[W_{lvl}^{-1}])^{\op}$.

\begin{lemma}\label{lem-L-coloured-operad}
	The functor $L$ of Theorem~\ref{thm-presheaf-I-G-spaces} lifts to a fully faithful functor of topological coloured operads.
\end{lemma}

\begin{proof}
	We define a functor between coloured operads by
	\[
	\OR_G \to (\DgTop{\I}{G}^\circ)^{\op}, \qquad \{(H_i,V_i)\} \mapsto \OR_G(\bigotimes (H_i,V_i), (e,-)).
	\]
	Using Equation~\ref{eq-tensor} we can rewrite this functor 
	in more familiar terms as
	\[
	\OR_G(\{(H_i,V_i)\}, (e,-))= (\prod_i G) \times_{(\prod_i H_i)} \Li(\bigoplus_i V_i, -) 
	\simeq \bigotimes_{i} (G \times_{H_i} \Li_{V_i}).
	\]
	By construction this functor defines a coloured operad map which lifts $L$.
	Using this description of the functor and the fact that 
	$G \times_H \Li_{W}$ corepresents the functor $X \mapsto X(W)^K$, we also see that the map 
	induced on multimorphism spaces
	\[
	\OR_G(\{(H_i,V_i)\}_{i \in I}, (K,W)) \to \DgTop{\I}{G}(G \times_K \Li_{W}, \bigotimes_{i\in I} 
	G \times_{H_i} \Li_{V_i})
	\]
	is a homeomorphism. Therefore the functor of coloured operads is fully faithful.
\end{proof}

The map $L$ of topological operads constructed above induces a map $L\colon \OR_G^\otimes\rightarrow (\DgTop{\CI}{G}[W_{lvl}^{-1}]^\otimes)^{\op}$ of $\infty$-operads. Furthermore this functor is again fully faithful. 

\begin{corollary}\label{thm-I-G-spaces-sym-mon-equiv}
	The functor $L \colon \OR_G^\otimes \to (\DgTp{\I}{G}[W_{lvl}^{-1}]^\otimes)^{\op}$ induces a symmetric monoidal equivalence
	\[
	\DgTop{\I}{G}[W_{lvl}^{-1}] \simeq \DgSpc{\OR_G},
	\] 
	where the right hand side is equipped with the Day convolution product.
\end{corollary}

\begin{proof}
This follows from Corollary \ref{cor-elmendorf}, where we argue as in Theorem~\ref{thm-presheaf-I-G-spaces} and use Lemma~\ref{lem-L-coloured-operad}.
\end{proof}

As a convenient reference, let us summarize the final description of $G$-prespectrum objects which combines all of the identifications obtained.

\begin{corollary}\label{cor:G-prespectra-are-modules}
	Let $G$ be a Lie group. Then there is a symmetric monoidal equivalence 
	\[
	\PSp_G\simeq \Mod_{S_G}(\DgSpcp{\OR_G}).
	\]
\end{corollary} 

\begin{proof}
	Combine Theorem~\ref{thm-I-G-spaces-sym-mon-equiv}, Corollary~\ref{cor:PSp_first} and Proposition~\ref{prop:Day_conv_pointed}.
\end{proof}

\begin{remark}
We will often implicitly identify $\PSp_G$ with $\Mod_{S_G}(\DgSpcp{\OR_G})$ for the remainder of the paper. 
\end{remark}

\subsection{$\I$-spaces and $\ORf$-spaces.}
We now undertake a similar analysis for the $\infty$-category of $\I$-spaces localized at the faithful level equivalences. Because many of the details are similar, we will be briefer in this section than the previous one.

\begin{definition}\label{def-ORf}
	We define a topological category $\ORf$ whose objects are pairs $(G,V)$ where $G$ is a 
	compact Lie group and $V$ is a faithful $G$-representation. The morphism spaces are 
	given by 
	\[
	\ORf((G,V),(H,W))= (\Li(V,W)/G)^H.
	\]
	There is a composition map
	\[
	\circ \colon \ORf((H,W), (L, U))\times \ORf((G,V), (H,W)) \to \ORf((G,V), (L,U))
	\]
	given by $([\psi],[\varphi])\mapsto [\psi \circ \varphi]$. Similarly to Lemma \ref{lem-composition-well-def}, one may verify this composition is well-defined, associative, unital and continuous.
\end{definition}

\begin{example}\label{ex-semifree}
	By definition $\ORf((G,V),(e,W))=\Li(V,W)/G$. Thus we can identify the functor 
	\[
	\ORf((G,V), (e,-))\colon \I \to \Top
	\] 
	with the semifree $\I$-space $\Li_{G,V}$ from Construction~\ref{con-semifree}. Recall this $\I$-space corepresents the functor $X\mapsto X(V)^G$.
\end{example}

\begin{definition}
	We let $\DgSpc{\ORf}$ denote the $\infty$-category of $\ORf$-spaces which is the 
	$\infty$-category of functors $\ORf \to \Spc$. We also write $\DgSpcp{\ORf}$ for 
	the $\infty$-category of functors $\ORf \to \Spcp$.
\end{definition}

We now prove the main result of this subsection.

\begin{theorem}\label{thm-presheaf-I-spaces}
	There is an equivalences of $\infty$-categories 
	\[
	\DgTop{\I}{}[W_{f-lvl}^{-1}]\simeq \DgSpc{\ORf}.
	\]
\end{theorem}

\begin{proof}
	The discussion in Example~\ref{ex-semifree} shows that there exists a functor of 
	topological categories (and so of $\infty$-categories)
	\[
	(\ORf)^{\op} \to \DgTop{\I}{}, \qquad (G,V)\mapsto \ORf((G,V), (e,-))=\Li_{G,V}.
	\]
	This is fully faithful by definition of $\ORf$. Since the $\I$-spaces 
	$\Li_{G,V}$ are bifibrant in the faithful level model structure, the composite 
	\[
	(\ORf)^{\op}\to \DgTop{\I}{} \to \DgTop{\I}{}[W_{f-lvl}^{-1}] 
	\]
	is also fully faithful. We note that the semifree $\I$-space $\Li_{G,V}$ 
	corepresents the functor $X \mapsto X(V)^G$, which commutes with small homotopy colimits. 
	Indeed the $G$-fixed points functor commutes will small homotopy colimits by the discussion in Example~\ref{ex-S_G}, and so does the evaluation functor $X \mapsto X(V)$ since it preserves all colimits (as they are calculated pointwise), faithful level equivalences by definitions and cofibrations (as one can verify by checking on the set of generating cofibrations).  
	Finally, the collection of objects $\{\Li_{G,V} \mid (G,V)\in\ORf\}$ is jointly 
	conservative by definition of the faithful level equivalences.
	Thus the claimed equivalence follows by applying Theorem~\ref{thm-elmendorf}.
\end{proof}

We now discuss how the symmetric monoidal structure on $\DgTop{\I}{}^c[W_{f-lvl}^{-1}]$ translates to $\DgSpcp{\ORfgl}$.

\begin{lemma}\label{lem-ORf_is_sym_mon}
	The topological category $\ORf$ is symmetric monoidal with unit object $(e, 0)$ and 
	tensor product given by $(G,V)\otimes (H,W)=(G \times H, V \oplus W)$. 
	In particular, the $\infty$-category of $\ORf$-spaces admits a symmetric monoidal 
	structure given by Day convolution. 
\end{lemma}

\begin{proof}
	The first claim is a straightforward verification. The second claim then follows 
	from Corollary~\ref{cor:tensor-in-Day-convolution}.
\end{proof}

Write $\ORf^\otimes$ for the $\infty$-operad associated to symmetric monoidal topological category $\ORf$.

\begin{lemma}\label{lem-L_gl-sym-mon}
	The functor $L_{gl} \colon \ORf \to (\DgTop{\I}{}[W_{f-lvl}^{-1}])^{\op}$ given by $(G,V)\mapsto \Li_{G,V}$ lifts to a fully faithful symmetric monoidal functor 
	\[
	L_{gl} \colon \ORf \to (\DgTop{\I}{}[W_{f-lvl}^{-1}])^{\op}
	\] of $\infty$-categories.
\end{lemma}

\begin{proof}
	It suffices to observe that Equation~(\ref{tensor-product-semifree}) implies that $L_{gl}\colon \ORfgl\rightarrow \DgTop{\I}{}$ is a strong monoidal functor. 
\end{proof}

\begin{corollary}\label{thm-I-spaces-sym-mon-equiv}
	There is a symmetric monoidal equivalence
	\[
	\DgTop{\I}{}[W_{lvl}^{-1}] \simeq \DgSpc{\ORf},
	\]where the right hand side is symmetric monoidal via Day convolution.
\end{corollary}

\begin{proof}
This follows from Corollary \ref{cor-elmendorf}, where we argue as in Theorem~\ref{thm-presheaf-I-spaces} and use Lemma~\ref{lem-L_gl-sym-mon}.	
\end{proof}

Summarizing all of the identifications made, we have the following description of the symmetric monoidal $\infty$-category of faithful global prespectra.

\begin{corollary}\label{cor:PSpgl=PSpfgl}
	Let $G$ be a Lie group. Then there is a symmetric monoidal equivalence 
	\[
	\PSpfgl\simeq \Mod_{\Sfgl}(\DgSpcp{\ORfgl}).
	\] 
\end{corollary} 

\begin{proof}
	Combine Corollary~\ref{cor:PSp_first}, Theorem~\ref{thm-I-spaces-sym-mon-equiv} and Proposition~\ref{prop:Day_conv_pointed}.
\end{proof}

\begin{remark}
	We will often implicitly identify $\PSpfgl$ with $\Mod_{\Sfgl}(\DgSpcp{\ORfgl})$.
\end{remark}

\section{Functoriality of equivariant prespectra}\label{sec:funct_prespectra}
The goal of this section is to construct a functor $\PSp_\bullet\colon\Glo^{\op}\to \Cat^\otimes_\infty$ sending a compact Lie group $G$ to the symmetric monoidal $\infty$-category of $G$-prespectra of Definition~\ref{def-prespectra}, and to compute its (partially) lax limit. By Corollary~\ref{cor:G-prespectra-are-modules}, the $\infty$-category of $G$-prespectra can be identified with the category of modules over a certain object $S_G$ in $\DgSpcp{\OR_G}$. Therefore our first step is to construct a functor sending a compact Lie group $G$ to the $\infty$-category $\DgSpcp{\OR_G}$. 

In the unstable case we observed that the relevant functoriality was induced by the functoriality of the partial slices $\Orb_{/G}$ in $\Glo$. Formally, the functoriality of the categories $\DgSpcp{\OR_G}$ is induced by a (pro)functoriality of the categories $\OR_G$, and we will see that this is once again given by "passing to the slices" of a global analogue $\ORgl$ of the individual equivariant categories $\OR_G$. The category $\ORgl$ will be fibred over $\Glo$ and its objects will consist of pairs $(G,V)$, where $G$ is a compact Lie group and $V$ is an arbitrary $G$-representation. Furthermore we will see that restricting to faithful representations, we recover $\ORfgl$.

\begin{construction}
    Let $G,H$ be compact Lie groups and $V$ and $W$ be orthogonal $G$ and $H$-representations respectively. We equip the topological space
    \[\Hom(H,G)\times\Li(V,W)\]
    with the right $G$-action and the left $H$-action given by
    \[(\alpha,\varphi)\cdot g=(c_g\alpha,\varphi g^{-1})\quad\text{and}\quad h\cdot(\alpha,\varphi)=(\alpha,h\varphi \alpha(h)^{-1})\,.\]
    Since the $G$ and $H$-actions commute, there is a residual $G$-action on the fixed points $\left(\Hom(H,G)\times\Li(V,W)\right)^H$. By definition, the fixed points space can be characterized as the space of pairs $(\alpha,\varphi)$ where $\alpha\colon H\to G$ is a Lie group homomorphism and $\varphi\colon V\to W$ is an $H$-equivariant isometry (where $H$ acts on $V$ via $\alpha$). If $K$ is another compact Lie group and $U$ is an orthogonal $K$-representation, we define a composition map
    \[(\Hom(H,G)\times\Li(V,W))^H\times (\Hom(K,H)\times\Li(W,U))^K\to (\Hom(K,G)\times\Li(V,U))^K,\quad (\alpha,\varphi)\cdot (\beta,\psi)=(\alpha\beta,\varphi\psi)\]
    that is compatible with the various actions, so that it induces an associative and unital composition map on the respective action groupoids:
    \[(\Hom(H,G)\times\Li(V,W))^H\sslash G\times (\Hom(K,H)\times\Li(W,U))^K\sslash H\to (\Hom(K,G)\times\Li(V,U))^K\sslash G\,.\]
\end{construction} 

\begin{definition}\label{definition:ORgl}
    Let $\ORgl$ be the topological category whose objects are pairs $(G,V)$ where $G$ is a compact Lie group and $V$ is an orthogonal $G$-representation. Its morphism spaces are defined to be
    \[\ORgl((G,V),(H,W))=|(\Hom(H,G)\times\Li(V,W))^H\sslash G|\]
    where $|-\sslash G|$ is the geometric realization of the action groupoid of $G$ on $\Li(V,W)$ (as in Definition~\ref{def:glo}). As in Lemma~\ref{lem-ORf_is_sym_mon}, one 
    sees that $\ORgl$ admits a symmetric monoidal structure given by $(G,V)\otimes (H,W)\simeq (G\times H,V\oplus W)$. We write $\ORgl^\otimes$ for the associated $\infty$-operad.
\end{definition}

The next result tells us that the $\infty$-category $\ORfgl$ from Definition~\ref{def-ORf} is equivalent to the subcategory of $\ORgl$ spanned by the faithful representations.

\begin{lemma}\label{lem:faithful_sub_ORgl}
    Let $\CC$ be the symmetric monoidal subcategory of $\ORgl$ spanned by $(G,V)$ where $V$ is a faithful $G$-representation. Then there is a symmetric monoidal functor of topological categories $\CC\to \ORfgl$ sending $(G,V)$ to $(G,V)$, which induces a homotopy equivalence on mapping spaces (and so it is an equivalence of the underlying $\infty$-categories).
\end{lemma}

\begin{proof}
    The functor is the identity on objects, so it suffices to define it on mapping spaces. 
    For any  $(G,V),(H,W)\in \CC$, let us consider the map
    \[p\colon(\Hom(H,G)\times\Li(V,W))^H \to (\Li(V,W)/G)^H\]
    sending $(\alpha,\varphi)$ to $[\varphi]$. We claim that this map exhibits the target as the quotient of the source by $G$. Firstly, note that the map is $G$-equivariant. Let us show that its fibres are exactly the $G$-orbits. Suppose we have a point $[\varphi]$ in the target and let us choose a representative $\varphi\colon V\to W$. Then we know that for every $h\in H$ $h\cdot[\varphi]=[h\varphi]=[\varphi]$. Then necessarily there exists $\alpha(h)\in G$ such that $h\varphi=\varphi\alpha(h)^{-1}$. Note that the element $\alpha(h)$ is unique since $V$ is a  faithful $G$-representation. Then the map $h\mapsto\alpha(h)$ is a Lie group homomorphism and its graph is closed in $H\times G$ (since it is a fibre of the continuous map $H\times G\to \Li(V,W)$ sending $(h,g)$ to $h\varphi g^{-1}$), so it is continuous. Then it is clear that $(\alpha,\varphi)$ is a preimage of $[\varphi]$, and so $p$ is surjective.
    
    On the other hand, if $(\alpha,\varphi)$ and $(\alpha',\varphi')$ have the same image under $p$, then there is some $g\in G$ so that $\varphi'=\varphi g$. A simple computation as before shows that this forces $\alpha'=c_g\alpha$ (since the $G$-action on $\Li(V,W)$ is faithful, $\alpha'$ is determined by $\varphi'$). Moreover, the action of $G$ on $(\Hom(H,G)\times\Li(V,W))^H$ is free and proper, and so $p$ is a principal $G$-bundle. In particular it induces a natural equivalence of topological groupoids
    \[(\Hom(H,G)\times\Li(V,W))^H\sslash G \simeq (\Li(V,W)/G)^H\]
    and so a homotopy equivalence
    \[\left|(\Hom(H,G)\times\Li(V,W))^H\sslash G\right| \simeq (\Li(V,W)/G)^H\]
    
    Finally, it is easy to check that $p$ is compatible with composition and sends the identity to the identity. Therefore it induces an equivalence of $\infty$-categories $\CC\to \ORfgl$. 
    We leave to the reader to check that the above can be given the structure of a symmetric monoidal equivalence.
\end{proof}

\begin{remark}\label{rem-pigl-s0-adjoint}
    There is a pair of functors of topological categories 
    \[
    s_0 \colon \Glo^{\op} \to\ORgl\,,\ \pi_{gl}\colon\ORgl\to\Glo^{\op}
    \]
    given by $s_0(G)=(G,0)$ and $\pi_{gl}(G,V)=G$ on objects. Note that $s_0$ and $\pi_{gl}$ 
    are both symmetric monoidal, where $\Glo$ is symmetric monoidal under the cartesian product (and therefore $\Glo^{\op}$ is equipped with the cocartesian symmetric monoidal structure). This implies that the functors $\pi_{gl}$ and $s_0$ lift to maps of $\infty$-operads $\pi_{gl}\colon \ORgl^{\otimes} \to (\Glo^{\op})^{\amalg}$ and $s_0\colon(\Glo^{\op})^\amalg\to\ORgl^\otimes$ respectively.
\end{remark}

\begin{lemma}\label{lem:fiber_OR_gl->Glo}
  Let $\{(G_i,V_i)\},(H,W)$ be objects of $\ORgl^\otimes$, and consider the map
   	\[
    \pi_{gl}\colon\Mul_{\ORgl}(\{(G_i,V_i)\},(H,W))\to \Mul_{\Glo^{\op}}(\{G_i\}, H).
    \]
     The homotopy fibre of this map over a group homomorphism $\alpha\colon H\to \prod_i G_i \in (\Glo^{\op})^\amalg$ is equivalent to the space of $H$-equivariant isometries $\bigoplus_i V_i\to W$ where $H$ acts on $\bigoplus_i V_i$ via $\alpha$.
\end{lemma}

\begin{proof} 
 Put $V=\bigoplus_i V_i$ and $G=\prod_i G_i$ so that $\alpha \colon H \to G$ 
 and we can rewrite the map induced by $\pi_{gl}$ as 
 \[
 \Map_{\ORgl}((G,V), (H,W)) \to \Map_{\Glo^{\op}}(G, H)=\Map_{\Glo}(H,G).
 \]
    We recall from Proposition \ref{prop:glo_hom_spaces} that the $G$-space $\Hom(H,G)$ decomposes as a disjoint union of orbits
    \[\Hom(H,G)\simeq \coprod_{(\alpha)} G/C(\alpha),\]
    where $\alpha$ is a conjugacy class of homomorphisms and $C(\alpha)$ is the centralizer of the image of $\alpha$. Therefore we have a decomposition
    \[\Map_{\ORgl}((G,V),(H,W))\simeq \left((\Hom(H,G)\times\Li(V,W))^H\right)_{hG}\simeq \coprod_{(\alpha)} \Li(V,W)^H_{hC(\alpha)},\]
    depending on the choice of an $\alpha$ in each conjugacy class. This lies above the decomposition  
    \[\Map_{\Glo}(H,G)\simeq \coprod_{(\alpha)} BC(\alpha)\] from Proposition \ref{prop:glo_hom_spaces}
    via the canonical maps $\Li(V,W)^H\to \ast$. Therefore the homotopy fibre over $\alpha$ is precisely $\Li(V,W)^H$.
\end{proof}

\begin{lemma}
 The functor $\pi_{gl} \colon \ORgl^\otimes \to (\Glo^{\op})^\amalg$ is a cocartesian 
 fibration, and therefore exhibits $\ORgl^\otimes$ as a $(\Glo^{\op})^\amalg$-monoidal 
 $\infty$-category.
\end{lemma}

\begin{proof}
 Consider $\{(G_i, V_i)\}_{i \in I} \in \ORgl^\otimes$, and let us set $V=\bigoplus_i V_i$ 
 and $G=\prod_i G_i$ so that $V$ is naturally a $G$-representation. Since $\pi_{gl}$ is a 
 map of $\infty$-operads, it is enough to find cocartesian lifts over active morphisms whose 
 target is in $\Glo^{\op}$. A multimorphism from $\{G_i\}$ to $H$ in $(\Glo^{\op})^\amalg$ is the datum of a continuous group homomorphism $\alpha \colon H \to G$. Consider the multimorphism 
 $f \in\ORgl^\otimes( \{(G_i,V_i)\},(H,\alpha^*V))$ lying over the map $\alpha$ which is represented 
 by the element 
 \[
 [\alpha, 1_V] \in \vert (\Hom(H,G) \times \Li(V, \alpha^* V))^H \sslash G \vert.
 \] 
 We claim that this is a cocartesian edge. This follows from the fact that for all $(L,W)\in \ORgl^\otimes$, the square 
  \[
 \begin{tikzcd}
  \Mul_{\ORgl}((H,\alpha^*V), (L,W)) \arrow[r,"f^*"] \arrow[d,"\pi_{gl}"] & \Mul_{\ORgl}(\{(G_i,V_i)\}, (L,W)) 
  \arrow[d, "\pi_{gl}"] \\
  \Mul_{\Glo^{\op}}(H, L) \arrow[r,"\alpha^*"] & \Mul_{\Glo^{\op}}(\{G_i\}, L)
  \end{tikzcd}
 \]
 is a homotopy pullback of spaces. We can verify this by checking that the 
 vertical fibres are equivalent. This is now a consequence of 
 Lemma~\ref{lem:fiber_OR_gl->Glo}.
\end{proof}

\begin{definition}
	We define $\Rep\colon \Glo^{\op}\rightarrow \Cat_\infty^\otimes$ to be the functor corresponding to $\ORgl^\otimes$ under the equivalence of Proposition~\ref{prop:equiv_amalg}.
\end{definition}

\begin{remark}	
$\Rep(G)$ is the $\infty$-category corresponding to the topologically enriched category with objects $V$ a $G$-representation, and morphism spaces $\Rep(V,W) = \Li(V,W)^G$, the space of $G$-equivariant linear isometries from $V$ to $W$. This is a symmetric monoidal category via direct sum. The functoriality in $\Glo$ is given by restriction of representations along group homomorphisms.
\end{remark}

Recall from Remark~\ref{rem-pi_G} that there is a map of $\infty$-operads
$\pi_G \colon \OR_G^\otimes \to (\OO_{G,\pr}^{\op})^\amalg$. Also note that there is a 
canonical functor $\OO_{G,\pr} \to \Glo$ which sends an object $G/H$ to $H$ and acts as 
\[\OO_{G,\pr}(G/H, G/K)\simeq \{g\in G\mid c_g(H)\subseteq K\}_{\h K} \rightarrow \hom(H,K)_{\h K},\quad g\mapsto [c_g\colon H\rightarrow K].\] This is an immediate generalization of the functor used in Lemma~\ref{lem:Orbslices} to (not necessarily compact) Lie groups. We denote the opposite of this functor by $\iota_G$. It induces a map of cocartesian $\infty$-operads which we denote by $\iota_G^{\amalg}$. We are now ready to state the next result. 

\begin{lemma}\label{lem:ORG-as-a-pullback}
    Let $G$ be a Lie group. Then there is a canonical map of $\infty$-operads 
    $\nu_G \colon \OR_G^\otimes \to \ORgl^\otimes$ and a cartesian square of 
    $\infty$-operads
    \[\begin{tikzcd}
        \OR_G^\otimes\ar[r, "\nu_G"]\ar[d, "\pi_G"'] & \OR_{gl}^\otimes\ar[d,"\pi_{gl}"]\\
        (\OO_{G,\pr}^{\op})^\amalg\ar[r, "\iota_G^{\amalg}"] & (\Glo^{\op})^\amalg.
    \end{tikzcd}\]
\end{lemma}

\begin{proof}
    It will suffice to construct the map $\nu_G$ at the level of topological coloured operads 
 and then apply Lemma~\ref{ORG_operad}. Recall from Definition~\ref{def-OR_G} that
    \[\OR_G((H,V),(K,W))=(G\times_H \Li(V,W))^K\]
    where $G\times_H \Li(V,W)$ is the quotient of $G\times \Li(V,W)$ by the right $H$-action $(g,\varphi)\cdot h=(gh,\varphi h)$. Since the $H$-action is free, we can identify the quotient with the homotopy quotient  (see \cite{korschgen}*{Theorem A.7} for example) and so there is a canonical identification
    \[\OR_G((H,V),(K,W))=|(G\times \Li(V,W))^K \sslash H|\]
    that respects composition. Moreover under this identification, the multilinear spaces of the coloured operad structure are given by 
    \[
    \OR_G(\{(H_i, V_i)\}_i, (K,W))= |(\prod_i G \times \Li(\bigoplus_i V_i, W))^K \sslash \prod H_i |.
    \]
    Therefore we may define a functor of topological coloured operad $\OR_G\to \ORgl$ by sending 
    $(H,V)$ to $(H,V)$ and on the multimorphism spaces we take the map which is induced by the map of topological groupoids
    \[(\prod_i G\times \Li(\bigoplus_i V_i,W))^K \sslash \prod_i H_i\to (\Hom(K,\prod_i H_i)\times
    \Li(\bigoplus_i V_i,W))^K \sslash \prod_i H_i, \qquad 
    (\{g_i\},\varphi)\mapsto ((c_{g_i}|_K)_{i},\varphi)\,.\]
   A tedious but simple calculation shows that these maps respect composition. This defines a map 
    $\nu_G \colon \OR_G^\otimes \to \ORgl^\otimes$ as required. 
    
    Another tedious calculation shows that the square in the lemma commutes (already as a square of topological operads) and that it is a pullback on 
    $0$-vertices. Therefore it is enough to show that every induced square 
    \[
    \begin{tikzcd}
    \Mul_{\OR_G}(\{(H_i,V_i)\}, (K,W)) \arrow[d,"\pi_G"] \arrow[r,"\nu_G"] &  
    \Mul_{\ORgl}(\{(H_i,V_i)\}, (K,W)) \arrow[d,"\pi_{gl}"]\\
    \Mul_{\OO_{G,\pr}^{\op}}(\{G/H_i\}, G/K) \arrow[r,"\iota_G"] & 
    \Mul_{\Glo^{\op}}(\{H_i\}, K)
    \end{tikzcd}
    \]
    of multimorphism spaces is a homotopy pullback. It suffices to check that the vertical homotopy fibres are 
    equivalent. 
    A morphism $\varphi\colon G/K \to \prod G/H_i$ in $\OO_{G,\pr}$ amounts 
    to giving elements $g_i \in G$ such that $c_{g_i}(K)\subseteq H_i$. The homotopy 
    fibre of $\pi_G$ over $\varphi$
    is given by the space of $K$-equivariant isometries $\bigoplus_i V_i \to W$ 
    where $K$ acts on each $V_i$ via $c_{g_i}$. The map $\iota_G$ sends $\varphi$ to 
    $(c_{g_i}\colon K \to H_i)$ and the 
    homotopy fibre over this is again the space of $K$-equivariant isometries as 
    above by Lemma~\ref{lem:fiber_OR_gl->Glo}. 
    As the vertical homotopy fibres are equivalent, the square is a pullback of 
    $\infty$-operads. 
\end{proof}

We write $\Arinj(\Glo)$ for the \textit{full} subcategory of $\Ar(\Glo)$ spanned by the injective group homomorphisms.

\begin{definition}\label{def-OR}
We define $\tOR^\otimes$ via the following pullback of $\infty$-operads
    	\[
    \begin{tikzcd}
    	\tOR^\otimes \arrow[d,"\pi_{\mathrm{inj}}"'] \arrow[r]        & \ORgl^\otimes \arrow[d,"\pi_{gl}"] \\		(\Arinj(\Glo)^{\op})^\amalg \arrow[r, "s^{\op}"] & (\Glo^{\op})^\amalg.    
    \end{tikzcd}
    \]
    Thus an object of $\tOR$, the underlying $\infty$-category of $\tOR^\otimes$, is a pair $(\alpha \colon H\to G, V)$ where $\alpha$ is injective and $V$ is a $H$-representation.
\end{definition}

\begin{lemma}\label{lemma:tOR-promonoidal}
	The composition 
	\[\pi\colon  \tOR^\otimes\xrightarrow{\pi_{\mathrm{inj}}} (\Arinj(\Glo)^{\op})^\amalg\xrightarrow{t^{\op}} (\Glo^{\op})^\amalg\]
	gives $\tOR^\otimes$ the structure of a $(\Glo^{\op})^\amalg$-promonoidal $\infty$-category, whose operadic fibre over $G$ is exactly $\OR_G^\otimes.$
\end{lemma}

\begin{proof}
We will show that each of the two maps in the defining composite is promonoidal in turn. Note that both are maps of $\infty$-operads. The map $\pi_{\mathrm{inj}}$ is a pullback of a cocartesian fibration, and therefore again cocartesian. The second map is then promonoidal by Example~\ref{example:cart-over-cocart-pro}.

Finally we note that the operadic fibre of $t^{\op}$ over $G$ is $(\OO_G^{\op})^\amalg$ by Lemma~\ref{lem:Orbslices} and the observation that $(-)^\amalg$ preserves pullbacks. Therefore, the calculation of the operadic fibre follows from Lemma~\ref{lem:ORG-as-a-pullback} and the observation that the composite $(\OO_G^{\op})^\amalg\rightarrow (\Arinj(\Glo)^{\op})^\amalg\xrightarrow{t^{\op}} (\Glo^{\op})^\amalg$ is equivalent to $\iota_G^\amalg$.
\end{proof}

Because $\pi$ is a promonoidal category over $(\Glo^{\op})^\amalg$ with operadic fibre $\OR_G^\otimes$, morally it represents a profunctor of promonoidal $\infty$-categories. Therefore we can extract an honest symmetric monoidal functor by taking copresheafs. This will be the functor $\Glo^{\op}\to \Cat_\infty^\otimes$ sending $G$ to $\DgSpcp{\OR_G}$.

\begin{definition}\label{def-functo-OR_bullet-space}
    The Day convolution $\Fun_{\Glo^{\op}}(\tOR^\otimes,\Spcp^\wedge \times (\Glo^{\op})^{\amalg})^{Day}$ is a $(\Glo^{\op})^\amalg$-monoidal $\infty$-category, whose operadic fibre over $G\in\Glo$ equals
    \[\Fun_{\Glo^{\op}}(\tOR^\otimes,\Spcp^\wedge\times (\Glo^{\op})^\amalg)^{Day} \times_{(\Glo^{\op})^\amalg} \Finp\simeq \Fun(\tOR^\otimes\times_{(\Glo^{\op})^{\amalg}}\Finp,\Spcp^\wedge)^{Day}\simeq \DgSpcp{\OR_G}\,\]
    by Example~\ref{ex-operadic-fibre} and Lemma~\ref{lemma:tOR-promonoidal}. We define $\DgSpcp{\OR_\bullet}\colon \Glo^{\op}\to \Cat^\otimes_\infty$ to be the functor associated to it under the equivalence of Proposition~\ref{prop:equiv_amalg}. 
\end{definition}

\begin{lemma}\label{lemma:cartesian-arrows-over-Orb}
	Let $\tOR$ be the underlying category of the $\infty$-operad $\tOR^\otimes$. Then the projection map
	\[\pi\colon \tOR\to \Glo^{\op}\]
	is cartesian over $\Orb^{\op}$, and an edge $(\sigma, \phi)\in \tOR$ is $\pi$-cartesian if and only if $s^{\op}(\sigma)$ and $\phi$ are equivalences.   
\end{lemma}

\begin{proof}
	Suppose we have an injection $\alpha\colon H\rightarrow G$, and an object $(\beta\colon K\rightarrow H, V)\in \tOR$. As noted before, the map $t^{\op}\colon \Arinj(\Glo)^{\op}\rightarrow \Glo^{\op}$ is a cartesian fibration. Furthermore over an injection $\alpha\colon H\to G$, cartesian lifts with target $\beta\colon K\to H$ are given by squares $\sigma$
	\[
	\begin{tikzcd}
		K \arrow[d, "\alpha\beta"' ] & K \arrow[d,"\beta"] \arrow[l, "\sim"] \\
		G 											 & H. \arrow[l, "\alpha"]
	\end{tikzcd}
	\] 
	In particular, we note that cartesian lifts of injections are sent to equivalences by the source functor $s^{\op}\colon \Arinj(\Glo)^{\op} \to \Glo^{\op}$. Lifting $s^{\op}(\sigma)$ to an equivalence $\phi \in \ORgl$ with target $(K,V)$, we obtain an edge $(\sigma,\phi)$ which lies over $\alpha$ and ends at $(\beta, V)$. Because both components of the edge $(\sigma,\phi)$ in $\tOR$ are $\pi$-cartesian, the edge $(\sigma,\phi)$ is itself $\pi$-cartesian. This shows that there are enough cartesian edges in $\tOR$ over injections, and that they are exactly of the form claimed.
\end{proof}

\begin{lemma}\label{lem:ORgl-is-localization-of-tOR}
    The projection map
    \[\tOR^\otimes\to \ORgl^\otimes\]
    induces a fully faithful symmetric monoidal functor
    \[\DgSpcp{\ORgl}\to \DgSpcp{\tOR}\]
    via restriction, with essential image those functors $F\colon \tOR\to \Spcp$ that send cartesian arrows over $\Orb^{\op}$ to equivalences.
\end{lemma}

\begin{proof}
 Recall from Lemma \ref{lem:adjointsource} that the source projection $\Arinj(\Glo)\to \Glo$ has a fully faithful left adjoint $\Glo\to \Arinj(\Glo)$ given by the diagonal embedding. Therefore, by the functoriality of the cocartesian operad \cite[Proposition~2.4.3.16]{HA}, it follows that the source projection
 \[(\Arinj(\Glo)^{\op})^\amalg\to (\Glo^{\op})^\amalg\]
 has a fully faithful operadic right adjoint. Since Bousfield localizations are stable under basechange, it follows that the projection
 \[\tOR^\otimes\to \ORgl^\otimes\]
 again has a fully faithful operadic right adjoint. Therefore $\tOR\rightarrow \ORgl$ is a Bousfield localization on underlying $\infty$-categories and moreover the fully faithful functor
 \[\DgSpc{\ORgl}\to\DgSpcp{\tOR}\]
 is symmetric monoidal by Proposition~\ref{proposition:functoriality-of-presheaves}(b). Finally, because $\tOR\rightarrow \ORgl$ is a Bousfield localization, the essential image of the functor $\Fun(\ORgl,\Spcp) \to\Fun(\tOR,\Spcp)$ is given by those functors which send the edges inverted by the map $\tOR\rightarrow \ORgl$ to equivalences. But these are exactly the cartesian arrows over the injections by Lemma~\ref{lemma:cartesian-arrows-over-Orb}.
\end{proof}

\begin{lemma}\label{lem:laxlim-of-OR-spaces}
    There are symmetric monoidal equivalences
    \[\laxlim_{G\in\Glo^{\op}} \DgSpcp{\OR_G}\simeq \DgSpc{\tOR}\qquad\textrm{ and }\qquad\laxlimdag_{G\in\Glo^{\op}} \DgSpcp{\OR_G}\simeq \DgSpcp{\ORgl}\]
    where the lax limit is marked over the subcategory $\Orb\subseteq\Glo$ of all objects and injective maps.
\end{lemma}
\begin{proof}
    By Proposition~\ref{proposition:laxlimit-is-norm} there is a symmetric monoidal equivalence
    \[\laxlim_{G\in\Glo^{\op}}\DgSpcp{\OR_G} \simeq N_p \Fun_{\Glo^{\op}}(\tOR^\otimes ,\Spcp^\wedge \times(\Glo^{\op})^{\amalg})^{Day}\]
    where $p:(\Glo^{\op})^\amalg\to \Finp$ is the structure morphism of $(\Glo^{\op})^{\amalg}$. Applying the formula of Day convolution twice (see Definition~\ref{def-day-conv}), and the transitivity of norms of operads, we obtain 
    	\begin{align*}
    	\laxlim \DgSpcp{\OR_\bullet} &\simeq  N_{p} N_{\pi} \pi^* (\Spcp^\wedge \times (\Glo^{\op})^\amalg) \\
    	&\simeq N_{\pi p} (\pi^* p^* \Spcp^\wedge )\\
    	&\simeq \Fun(\tOR^\otimes,\Spcp^\wedge)^{Day}\\
    	&= \DgSpcp{\tOR}
    \end{align*}
    To compute the partially lax limit we appeal to Remark \ref{rem-partially-lax-underlying} to reduce to a statement on underlying categories. Combining Remarks \ref{rem-norm-underlying} and \ref{rem-norm-extend-sec}, we conclude that the underlying $\infty$-category of the $\infty$-operad $N_p \Fun_{\Glo^{\op}}(\tOR^\otimes,\Spcp^\wedge \times(\Glo^{\op})^\amalg)^{Day}$ is given by sections of the cocartesian fibration $\pi_*\pi^*(\Spcp\times \Glo^{\op})$, where by slight abuse of notation we write $\pi = U(\pi)$. Therefore we may calculate
    \[
    \Fun_{/\Glo^{\op}}(\Glo^{\op},\pi_*\pi^*(\Spcp\times\Glo^{\op})) \simeq
    \Fun_{/\Glo^{\op}}(\OR,\Spcp\times\Glo^{\op}) \simeq \Fun(\OR,\Spcp)
    \]
    using the definition of the left adjoints $\pi^*$ and $\pi_!$.    
    Now by Theorem \ref{thm-lax-limit-section} the partially lax limit of the diagram in question is given by the full subcategory of the left-most category spanned by those sections which map edges in $\Orb^{\op}$ to cocartesian arrows. 
    We now apply~\cite[Corollary 3.2.2.13]{HTT} (with $p\colon \OR\times_{\Glo^{\op}} \Orb^{\op} \to \Orb^{\op}$, $q\colon \Spcp \times \Orb^{\op} \to \Orb^{\op}$ and $T=\pi_*\pi^*(\Spcp\times\Glo^{\op})\times_{\Glo^{\op}}\Orb^{\op}$) together with Lemma \ref{lemma:cartesian-arrows-over-Orb}, to see that these sections corresponds to those functors in $\Fun_{/\Glo^{\op}}(\OR,\Spcp\times \Glo^{\op})$ which send cartesian edges over $\Orb^{\op}$ to cocartesian edges of $\Spcp\times \Orb^{\op}\rightarrow \Orb^{\op}$. These are exactly those maps which are equivalences in the first component, and therefore such sections corresponds to functors $F\colon \OR\rightarrow \Spcp$ which map cartesian edges over $\Orb$ to equivalences. Therefore we conclude by applying Lemma~\ref{lem:ORgl-is-localization-of-tOR}.
\end{proof}

\begin{proposition}\label{proposition:functoriality-of-prespectra}
    There exists a functor $\PSp_\bullet:\Glo^{\op}\to \Cat_\infty^\otimes$ sending $G$ to $\PSp_G$.    
    Moreover, there is a symmetric monoidal equivalence
    \[\laxlimdag_{G\in\Glo^{\op}} \PSp_G\simeq \Mod_{S_{gl}}(\DgSpcp{\ORgl}).\]
\end{proposition}

\begin{proof}
    There is a lax symmetric monoidal topologically enriched functor $S_{gl}\colon\ORgl\to\Spcp$ sending $(G,V)$ to 
    $(S^V)^G$. This induces a lax symmetric monoidal functor of $\infty$-operads, which uniquely specifies a commutative algebra in $\DgSpcp{\tOR}$ 
    by~\cite{HA}*{Example 2.2.6.9}, where we view $\DgSpcp{\ORgl}$ as a symmetric monoidal subcategory of $\DgSpcp{\tOR}$ using Lemma~\ref{lem:ORgl-is-localization-of-tOR}. Applying Theorem~\ref{thm:modules-in-laxlimit} to the lax limit of 
    Lemma~\ref{lem:laxlim-of-OR-spaces} shows that there is a functor sending $G$ to 
    $\Mod_{S_G}(\DgSpcp{\OR_G})\simeq \PSp_G$ (see Corollary~\ref{cor:G-prespectra-are-modules}) 
    whose lax limit is $\Mod_{S_{gl}}(\DgSpcp{\tOR})$.

    Finally, we have to calculate the subcategory corresponding to the partially lax limit. Because the natural transformation $\PSp_G\to \DgSpcp{\OR_G}$ is point-wise conservative, we 
    can check that an object lies in the partial lax limit of $\PSp_G$ by checking 
    that its image lies in the partially lax limit of $\DgSpcp{\OR_G}$. In other words, we have a 
    pullback square of symmetric monoidal $\infty$-categories
    \[
    \begin{tikzcd}
    \laxlim_{G}^{\dagger}\PSp_G  \arrow[r] \arrow[d] &  \laxlim_{G} \PSp_G \arrow[d]\\
    \laxlim_{G}^{\dagger} \DgSpcp{\OR_G} \arrow[r]& \laxlim_{G} \DgSpcp{\OR_G} \,.
    \end{tikzcd}    
    \]
    Therefore, by Lemma~\ref{lem:laxlim-of-OR-spaces} and the previous paragraph 
    we have a symmetric monoidal equivalence
    \[\laxlimdag_{G\in\Glo^{\op}}\PSp_G \simeq \Mod_{S_{gl}}(\Fun(\tOR,\Spcp))\times_{\Fun(\tOR,\Spcp)} \Fun(\ORgl,\Spcp) \]
    Finally, since $\Sgl\in \Fun(\ORgl,\Spcp)$ this implies that
    \[\laxlimdag_{G\in\Glo^{\op}}\PSp_G \simeq \Mod_{S_{gl}}(\DgSpcp{\ORgl})\,.\qedhere\]\end{proof}

\begin{notation}
	We write $\PSplax$ for the $\infty$-category $\Mod_{S_{gl}}(\DgSpcp{\ORgl})$, and identify it with $\laxlimdag\PSp_\bullet$.
\end{notation}

Recall the definition of the diagram $\Spc_\bullet\colon \Glo^{\op}\rightarrow \Cat_\infty^\otimes$ from Construction~\ref{const:Orbslices}, which sends a group $G$ to the $\infty$-category of $G$-spaces. We would like to construct a natural transformation $\Sigma^\infty\colon \Spc_\bullet \rightarrow \PSp_\bullet$, whose component at $G$ is given by an analogue of the suspension prespectrum functor. Morally, this sends a $G$-space $X$ to the $S_G$-module $(H,V)\mapsto (X\wedge S^V)^H$. We make this precise in the next construction. Let us first fix some notation; we write $\Spc_{\bullet,\ast}$ for the composite $(-)_*\circ\Spc_{\bullet}$ of $\Spc_{\bullet}$ with the functor which sends a presentably symmetric monoidal category to the $\infty$-category of pointed objects.

\begin{construction}\label{con-suspension-prespectra}
We will construct natural transformations of functors $\Glo^{\op}\to \Cat_\infty^\otimes$
	\[
	\Spc_\bullet\to \Spc_{\bullet,\ast} \to \PSp_\bullet
	\]
	
The first natural transformation is simply given by postcomposing $\Spc_\bullet$ with the natural transformation $(-)_+\colon\mathrm{id} \to({-})_{\ast}$ of functors $(\mathrm{Pr}^{\mathrm{L}})^\otimes \rightarrow (\mathrm{Pr}^{\mathrm{L}})^\otimes$.
	
	For the second natural transformation, we will construct it as a composite
	\[\Spc_{\bullet,\ast} \rightarrow \DgSpcp{\OR_\bullet} \rightarrow \PSp_\bullet.\] 
	For the latter transformation $\DgSpcp{\OR_\bullet} \to\PSp_\bullet$, we simply note that the free module functors
	\[
	S_{G} \otimes -\colon \DgSpcp{\OR_G}\to \Mod_{S_G}(\DgSpcp{\OR_G})\simeq \PSp_G
	\]
    are symmetric monoidal and fit into a natural transformation by the second half of Theorem \ref{thm:modules-in-laxlimit}.
    
    For the first, it will be technically convenient to construct the natural transformation $\Spc_{\bullet,\ast}^\wedge \rightarrow \DgSpcp{\OR_\bullet}$ as a map of 
    $(\Glo^{\op})^\amalg$-monoidal $\infty$-categories and then to use 
    Proposition~\ref{prop:equiv_amalg}. 
    
    For this, we need to pin down the $(\Glo^{\op})^\amalg$-monoidal $\infty$-category which corresponds to $\Spc_{\bullet,\ast}$ under Proposition~\ref{prop:equiv_amalg}. Note that the map $t^{\op}\colon (\Arinj(\Glo)^{\op})^\amalg \rightarrow (\Glo^{\op})^\amalg$ exhibits $ \Arinj(\Glo)^{\op}$ as a $(\Glo^{\op})^\amalg$-monoidal category, see Example~\ref{example:cart-over-cocart-pro}. We claim that $\Spc_{\bullet,\ast}$ corresponds to the Day convolution 
    \[
    \Fun_{\Glo^{\op}}((\Arinj(\Glo)^{\op})^\amalg, \Spcp^\wedge\times (\Glo^{\op})^\amalg)^{Day}.
    \]
    To see this, we first note that 
    \[
    \Fun_{\Glo^{\op}}((\Arinj(\Glo)^{\op})^\amalg, \Spc^\times \times (\Glo^{\op})^\amalg)^{Day}
    \] 
    classifies $\Spc_\bullet^\times$, because it does so on underlying categories (combine Remark \ref{rem-norm-underlying} and \cite{GHN}*{Proposition 7.3}) and the forgetful functor $\Cat_\infty^\otimes \rightarrow \Cat_\infty$ is faithful when restricted to cartesian monoidal $\infty$-categories.
    Now we observe that the $(\Glo^{\op})^\amalg$-monoidal functor 
    \[
    ((-)_+)_\ast\colon \Fun_{\Glo^{\op}}(\Arinj((\Glo^{\op})^\amalg, \Spc^\times \times (\Glo^{\op})^\amalg)^{Day}\rightarrow \Fun_{\Glo^{\op}}(\Arinj((\Glo^{\op})^\amalg, \Spc^\wedge_\ast \times (\Glo^{\op})^\amalg)^{Day}
    \] 
    agrees pointwise with $(-)_+$, and therefore by the universal property of taking pointed objects (see \cite{HA}*{Proposition 4.8.2.11}) $\Fun_{\Glo^{\op}}(\Arinj((\Glo^{\op})^\amalg, \Spc^\wedge_\ast \times (\Glo^{\op})^\amalg)^{Day}$ must classify $\Spc_{\bullet,\ast}$. 
    
    Now we can construct the $(\Glo^{\op})^{\amalg}$-monoidal functor which will induce $\Spc_{\bullet,\ast} \rightarrow \DgSpcp{\OR_\bullet}$. Pulling back the functor $s_0$ of Remark~\ref{rem-pigl-s0-adjoint} along $t^{\op}$ we obtain a commutative diagram
    \[
    \begin{tikzcd}
    \Arinj(\Glo)^{\op})^\amalg \arrow[rd,"t^{\op}"'] \arrow[rr,"s_{0,\mathrm{inj}}"]&  & 
    \tOR^\otimes  \arrow[dl,"\pi"]\\
                                      & (\Glo^{\op})^\amalg&
    \end{tikzcd}
    \]
    where $t^{\op}$ and $\pi$ exhibit the sources as 
    $(\Glo^{\op})^\amalg$-promonoidal $\infty$-categories by Lemma~\ref{lemma:tOR-promonoidal}, so that $s_{0,\mathrm{inj}}$ is a map 
    of $(\Glo^{\op})^\amalg$-promonoidal $\infty$-categories. One can then verify that $s_{0,\mathrm{inj}}$ satisfies the hypotheses of Proposition~\ref{proposition:functoriality-of-presheaves}(a), and there exists a $(\Glo^{\op})^\amalg$-monoidal functor
    \[
    (s_{0,\mathrm{inj}})_!\colon\Fun_{\Glo^{\op}}((\Arinj(\Glo)^{\op})^\amalg, \Spcp^\wedge\times (\Glo^{\op})^\amalg)^{Day}\to \Fun_{\Glo^{\op}}(\OR^\otimes, \Spcp^\wedge\times (\Glo^{\op})^\amalg)^{Day}
    \]
    which then induces the required natural transformation. This description shows as well that the component at $G$ coincides with $\Li_0$, and so the composite functor $\Spc_{G,\ast}\to\PSp_G$ is analogous to the usual suspension prespectrum functor $F_0(-)$. We will formulate a precise statement to this effect as Proposition~\ref{prop-suspensions-agree}.
\end{construction}

\section{Functoriality of equivariant spectra}

In the previous section we have constructed the functor 
\[
\PSp_\bullet\colon \Glo^{\op}\rightarrow \Cat_\infty^\otimes,
\] 
and calculated its partially lax limit. In this section we will show that this functor descends to a diagram $\Sp_\bullet$ where on every level we restrict to the subcategory of spectrum objects. Furthermore, we will prove that the functoriality obtained in this way agrees with the standard functoriality of equivariant spectra under the restriction-inflation functors. Finally, we will compute the partially lax limit of $\Sp_\bullet$ as a Bousfield localization of $\PSplax = \laxlimdag \PSp_\bullet$.  

Given a continuous group homomorphism $\alpha\colon H\rightarrow G$ between compact Lie groups, we write
\[
\alpha^*\colon \DgSpcp{\OR_G}\to \DgSpcp{\OR_H}
\]
for the symmetric monoidal functor induced by $\alpha$. Our goals require a better understanding of $\alpha^*$. We start by studying the interaction between $\alpha^*$ and the Quillen adjunction 
of Construction~\ref{con-free-IG-space}
\[
\Li_V\colon G\Top \leftrightarrows \DgTop{\CI}{G} \cocolon \mathrm{ev}_V
\] 
for a given $G$-representation $V$. However before we do this, we first need to understand how these adjunctions manifest themselves under the equivalences 
\[\Spc_{G}\simeq \DgSpc{\OO_G}\quad \text{and} \quad \DgSpc{\OR_G} \simeq \DgTop{\I}{G}\]
of Example~\ref{ex-S_G} and Theorem~\ref{thm-presheaf-I-G-spaces}. 

\begin{remark}\label{rem:free_ORG_space}
Consider $X\in\DgTop{\I}{G}$ and a $G$-representation $V$. Then the $G$-space $X(V)$ corresponds to the presheaf
\[G/H\mapsto X(V)^H\simeq \Map_{\DgTop{\I}{G}}(G\times_H\Li_{V|_H},X)\,.\]
Note that $G\times_H \Li_{V|_H}$ is the image of $(H,V|_H)$ under the embedding $L$ of 
Theorem~\ref{thm-presheaf-I-G-spaces}. Therefore, if we let $s_V\colon \OO_G^{\op}\to \OR_G$ be the cocartesian section of $\pi_G$ sending $G/G$ to $(G,V)$, we have $s(G/H)\simeq(H,V|_H)$, so we can identify $\mathrm{ev}_V$ with 
\[s_V^*\colon \DgSpc{\OR_G}\to \Spc_G\qquad X\mapsto X\circ s_V\] and similarly for the pointed version. It follows that the derived functor associated to $\Li_V$ is given by the left Kan extension functor $(s_V)_!$. Finally, we can compute that this is given by
\[(\Li_V X)(H,W)\simeq \Li(V,W)^H\times X^H\, ,\]
by the following Lemma.
\end{remark}

\begin{lemma}\label{lem:LKE-along-cocartesian-section}
Let $\pi\colon\CE\to \CB$ be a cocartesian fibration of $\infty$-categories and $s\colon\CB\to \CE$ be a cocartesian section. For every functor $F\colon\CB\to \CC$ where $\CC$ is a cocomplete $\infty$-category, we can compute the left Kan extension along $s$ by
\[(s_!F)(e)\simeq \Map_{\pi^{-1}(\pi e)}(s\pi (e), e)\times F(\pi (e))\]
for all $e\in E$.
\end{lemma}

\begin{proof}
	By the usual formula for left Kan extensions we have that 
	\[(s_!F)(e)\simeq \colim_{b\in \CB\times_\CE \CE_{/e}} F(b)\,.\]
	We claim that the projection $\CB\times_\CE \CE_{/e}\to \CB_{/\pi e}$ is a left fibration with fibre over $f\colon b\to \pi e$ given by $\Map^f_\CE(s(b),e)$. 
	In particular, since $F$ is constant along the fibres of this fibration and $\CB_{/\pi e}$ has a final object, we have
	\[\colim_{b\in \CB\times_\CE \CE_{/e}} F(b)\simeq \colim_{[f\colon b\,\rightarrow \pi e] \in \CB_{/\pi e}} \Map^f_\CE(s(b),e)\times F(b) \simeq \Map_{\pi^{-1}(\pi e)}(s\pi (e),e)\times F(\pi (e))\,.\]
	It only remains to prove that the functor $\CB\times_\CE \CE_{/e}\to \CB_{/\pi e}$ is a left fibration. That is, we need to show that for every diagram
	\[\begin{tikzcd}
		\Lambda^n_i \ar[r]\ar[d] & \CB\times_{\CE}\CE_{/e}\ar[d]\\
		\Delta^n\ar[r]\ar[ur,dashed] & \CB_{/\pi e}
	\end{tikzcd}\]
	with $0\le i<n$ there exists a dotted arrow completing the diagram. Using the definition of slice $\infty$-categories, this is equivalent to finding a dotted arrow completing the dotted diagram
	\[\begin{tikzcd}
		\Lambda^n_i\star \Delta^0 \ar[r,"F"]\ar[d] & \CE\ar[d,"\pi"]\\
		\Delta^n\star\Delta^0\simeq\Delta^{n+1} \ar[r,"G"] \ar[ur,dashed]& \CB
	\end{tikzcd}\]
	where $F$ restricted to $\Lambda^n_i\subseteq \Lambda^{n+1}_i$ is given by the restriction of $sG$. This diagram is a diagram of marked simplicial sets when we give $\CB$ the total marking, $\CE$ the cocartesian marking and on the left column the marking $(\Lambda^n_i)^\sharp\star \Delta^0\to (\Delta^n)^\sharp\star \Delta^0$. Since the left vertical arrow is left marked anodyne by \cite[Lemma~4.10]{Exp2-1}, the lift exists.
\end{proof}
Having understood the adjunction $\Li_V\dashv \mathrm{ev}_V$, we now discuss how this interacts with the functor $\alpha^*$.

\begin{proposition}\label{proposition:formulas-for-alpha*}
	Let us fix an arrow $\alpha\colon H\to G$ in $\Glo$.
	\begin{enumerate}
		\item Given a pointed $G$-space $X$, there is a natural equivalence
		\[\alpha^*\Li_VX \simeq \Li_{\alpha^*V}(\alpha^*X)\]
		\item Given a pointed $\OR_G$-space $Y$, there is a natural equivalence
		\[\alpha^*\mathrm{ev}_V Y\simeq \mathrm{ev}_{\alpha^*V}\alpha^*Y\]
		\item Under the two previous identifications, the counit natural transformation
		\[\Li_V\mathrm{ev}_V X\to X\]
		is sent by $\alpha^*$ to
		\[\Li_{\alpha^*V}\mathrm{ev}_{\alpha^*}V(\alpha^*X)\to \alpha^*X\,\]
		the counit natural transformation for $\alpha^*V$ applied to $\alpha^*X$.
	\end{enumerate}
\end{proposition}
\begin{proof}
	Write $\OO_\alpha\simeq \Arinj(\Glo)\times_{\Glo}[1]$ (using the target map $t \colon \Arinj(\Glo)\to \Glo$) and let $i_0\colon \OO_H\to \OO_\alpha$, $i_1\colon \OO_G\to \OO_\alpha$ be the inclusions of the fibres over 0 and 1 respectively. Similarly, write $\OR_\alpha:=\ORgl\times_{\Glo^{\op}}\OO_\alpha^{\op}$ and $j_0,j_1\colon \OR_H,\OR_G \to \OR_\alpha$ for the inclusion of the fibre of $\OR_\alpha\rightarrow [1]^{\op}$ over 0 and 1 respectively. Therefore by Remark~\ref{rem:pushforward-in-exponential} we can identify
	\[\alpha^*\simeq i_0^*(i_1)_!\colon \Spc_{G,\ast}\to \Spc_{H,\ast}\qquad \alpha^*\simeq j_0^*(j_1)_!\colon \DgSpcp{\OR_G}\to \DgSpcp{\OR_H}\,.\]
	Let $s_V \colon \OO_G^{\op}\to \OR_G$ be the cocartesian section of $\pi_G\colon \OR_G \to \OO_G^{\op}$ which sends 
	$G/G$ to $(G,V)$. Similarly let $s\colon\OO_\alpha^{\op}\to \OR_\alpha$ be the cocartesian section sending the initial object $i_1(G/G)$ of $\OO_\alpha^{\op}$ to $j_1(G,V)$. Then $s$ restricts to $s_V$ on $\OO_G^{\op}$ and to $s_{\alpha^*V}$ on $\OO_H^{\op}$, since a cocartesian section is determined by where it sends the initial object. Therefore by Remark~\ref{rem:free_ORG_space} we obtain 
	\[
	\alpha^*\Li_V X\simeq \alpha^*(s_V)_!X\simeq j_0^* (j_1s_V)_! X \simeq j_0^* s_! (i_1)_!X\,
	\] 
	for every pointed $G$-space $X$. Using the formula for $s_!$ described in Lemma~\ref{lem:LKE-along-cocartesian-section} we see that the above can be identified with $(s_{\alpha^*V})_! i_0^*(i_1)_!X$, thus proving the first statement.
	
	Now let $Y$ be an $\OR_G$-space. Then we claim that $s^*(j_1)_!Y$ is left Kan extended from $\OO_G^{\op}$. In fact this happens if and only if $s^*(j_1)_! Y$ sends the arrows $(G,\alpha L)\to (H,L)$ in $\OO_\alpha^{\op}$ to equivalences. But the arrow
	\[
	[s(G,\alpha L)\to s(H,L)]\simeq{} [(G,\alpha L, V)\to (H,L,\alpha^*V)]
	\]
	is a terminal object of $\OR_G \times_{\OR_\alpha} (\OR_\alpha)_{/(H,L,\alpha^*V)}$ and so it is sent to an equivalence by $(j_1)_!Y$. This implies that
	\[\mathrm{ev}_{\alpha^*V} \alpha^*Y\simeq s_{\alpha^*V}^*(j_1)_!Y\simeq j_0^*s^*(j_1)_!Y\simeq j_0^*(j_1)_! (s_V)^*Y\simeq \alpha^* \mathrm{ev}_VY,\]
	proving the second statement.
	
	Finally we consider for every $\OR_G$-space $Y$, the natural transformation
	\[s_!s^*(j_1)_!Y\to (j_1)_!Y,\]
	and note that this is a natural transformation of functors left Kan extended from $\OR_G$, which restricts to
	\[(s_V)_!s_V^* Y\to Y\textrm{ and } (s_{\alpha^*V})_!s_{\alpha^*V}^* \alpha^*Y\to \alpha^*Y\]
	on the fibres over 0 and 1 respectively. Thus $\alpha^*$ sends the former to the latter, showing the third statement.
\end{proof}

With this result we can show that $\PSp_\bullet$ restricts to a functor on spectrum objects. 

\begin{proposition}\label{prop-transformation-localization-spectra}
	There exists a functor $\Sp_\bullet\colon \Glo^{\op}\rightarrow \Cat_{\infty}^\otimes$ and a natural transformation of functors
	\[
	L_{\bullet}\colon \PSp_{\bullet} \to \Sp_\bullet
	\]
	whose component for a fixed $G$ is the spectrification functor $L_G\colon \PSp_{G}\rightarrow \Sp_{G}$.
\end{proposition}

\begin{proof}
	Consider a group homomorphism $\alpha\colon H\rightarrow G$. We claim that the functor $\PSp_\alpha\colon \PSp_G\rightarrow \PSp_H$ preserves stable equivalences. It suffices to show that it preserves the generating equivalences $G\rtimes_K\lambda_{V,W}$ of Proposition~\ref{prop-Sp_G-local}. Moreover, since $G$ is compact, we can restrict to the cofinal set $W$ of $K$-representations that are extended from $G$.
	
	First note that $\lambda_{V,W}\simeq (G\rtimes_K F_V(S^0))\otimes \lambda_{0,W}$. Since $\PSp_\alpha$ is symmetric monoidal by construction and stable equivalences are stable under tensor product, it suffices to show that $\PSp_\alpha(\lambda_{0,W})$ is a stable equivalence. We claim it is equivalent to $\lambda_{0,\alpha^*W}$. In fact $\lambda_{0,W}$ is exactly the counit of the adjunction $F_W\dashv \mathrm{ev}_W$ of Construction~\ref{cons-lambda-maps-equivariant} applied to $S_G$. Therefore we can factor it as
	\[ (F_W \mathrm{ev}_W)S_G \simeq (S_G\otimes -)\Li_W \mathrm{ev}_W U S_G\to (S_G\otimes -)US_G\to S_G\]
	where $(S_G\otimes -)\dashv U$ is the free-forgetful adjunction between $\PSp_G$ and $\DgSpcp{\OR_G}$, and the arrows are the counits of the respective adjunctions. Then our claim follows from Theorem~\ref{thm:modules-in-laxlimit} and Proposition~\ref{proposition:formulas-for-alpha*}.
	
	Knowing that $\PSp_\alpha$ preserves stable equivalences, we can combine Construction~\ref{con-suspension-prespectra} and 
	Corollary~\ref{cor:laxlim-monoidal}, to obtain $\Sp_\bullet$ and the natural transformation 
	$L_\bullet \colon \PSp_{\bullet}\to \Sp_\bullet$. 
\end{proof}

Recall that we constructed a natural transformation $\Sigma^\infty_\bullet\colon \Spc_{\bullet,\ast}\rightarrow \PSp_\bullet$ in Construction~\ref{con-suspension-prespectra}, which pointwise was our analogue of the suspension prespectrum functor. We may compose this with the natural transformation $L_\bullet$ to obtain a new natural transformation, which we again denote by $\Sigma^\infty_\bullet$. 

\begin{proposition}\label{prop-suspensions-agree}
The component of $\Sigma^\infty_\bullet\colon \Spc_{\bullet,\ast}\rightarrow \Sp_\bullet$ at the group $G$ is equivalent to the standard suspension spectrum functor. 
\end{proposition}

\begin{proof}
Considering the component at $G$, we observe that the functor $\Sigma^\infty_G$ is defined as the composition
\[\Spc_{G,\ast}\to \DgSpcp{\OR_G}\to \Mod_{S_G}(\DgSpcp{\OR_G})\simeq\PSp_G\to \Sp_G\]
where the first functor is $\Li_0$ (i.e. precomposition along $\OR_G\to \OO_G^{\op}$), the second functor is the free $S_G$-module functor $(S_G\otimes-)$ and the third functor is the localization functor. These functors are all modeled by left Quillen functors
\[G\Top_{*}\to\DgTp{\I}{G}\to \Sp_G^O\to \Sp_G^O\]
given by the constant $\I$-$G$-space, the free $S_G$-module and the identity respectively. Therefore $\Sigma^\infty_G$ is modelled by their composition, which is exactly the suspension spectrum functor constructed in \cite{MandellMay}.
\end{proof}

This suffices for us to conclude that the functoriality of $\Sp_\bullet$ agrees morphism-wise with the functoriality of equivariant spectra in restriction, by the universal property of $G$-spectra.

\begin{corollary}\label{cor:funct-Sp_bullet}
The functor $\Sp_\bullet \colon \Glo^{\op} \to \Cat^\otimes_\infty$ sends a compact Lie group $G$ to $\Sp_G$ and a continuous group homomorphism $\alpha\colon H \to G$ to the restriction functor $\alpha^* \colon \Sp_G \to \Sp_H$.
\end{corollary}

\begin{proof}
	Consider the commutative diagram
	\[
	\begin{tikzcd}
		\Sp_G \arrow[r,"\Sp_\alpha"] & \Sp_H \\
		\Spc_{G,\ast}\arrow[u,"\Sigma^\infty_G"] \arrow[r,"\Spc^{\alpha}_{\ast}"'] & \Spc_{H,\ast}
		\arrow[u,"\Sigma^\infty_H"'] \\
		\Spc_{G} \arrow[u,"(-)_+"] \arrow[r,"\Spc^\alpha"'] & \Spc_{H}\arrow[u,"(-)_+"']
	\end{tikzcd}
	\]
	of symmetric monoidal functors. By the universal property of 
	$G$-spectra~\cite{gepner2020equivariant}*{Corollary C.7}, the functor $\Sp_\alpha$ 
	is uniquely determined 	by $\Spc_{\alpha,\ast}$, 
	and this is completely determined by $\Spc_{\alpha}$ by~\cite{HA}*{Proposition 4.8.2.11}. 
	Finally, Proposition~\ref{prop-restriction} identifies the functor $\Spc^{\alpha}$ with 
	$\alpha^*$. 
\end{proof}

\begin{remark}
	Note that the argument of Corollary \ref{cor:funct-Sp_bullet} in facts shows that the natural transformation $\Sigma_{\bullet,\ast}\colon \Spc_{\bullet,\ast}\rightarrow \Sp_\bullet$ admits a universal property. This forces $\Sp_\bullet$ to coincide with the construction of \cite{norms}*{Section~9} on the subcategory of $\Glo$ spanned by finite groups. This suggests a possible comparison between ultracommutative $\Fin$-global ring spectra in the sense of \cite{Schwede18} and normed spectra in the sense of \cite{norms}. 
\end{remark}
 
We have now constructed $\Sp_\bullet$ and shown that it agrees with the standard functoriality of equivariant spectra. 
We will write $\Splax$ for the partially lax limit $\laxlimdag \Sp_\bullet$. We would like to describe $\Splax$ as a Bousfield localization of $\PSplax$ by applying Lemma~\ref{lem:adjunction-to-the-marked-lax-limit}. To do this requires the following two lemmata.

\begin{proposition}\label{lemma:projection-formula-for-injectives}
	Let $\alpha\colon H\to G$ be an injective group homomorphism. Then the functor $\alpha^*\colon\DgSpc{\OR_G}\to\DgSpc{\OR_H}$ has a left adjoint $\alpha_!$. Moreover, under the identification of Theorem~\ref{thm-presheaf-I-G-spaces} the adjunction $\alpha_!\dashv\alpha^*$ corresponds to the Quillen adjunction $G\times_H- \dashv \alpha^*$ of Proposition~\ref{prop-change-of-groups-functors}.
	
	In particular for $X\in \DgSpc{\OR_H}$ and $Y\in \DgSpc{\OR_G}$ the comparison map
	\[\alpha_!(X\otimes \alpha^*Y)\to \alpha_!X \otimes Y\]
	adjoint to $X\otimes\alpha^*Y\to \alpha^*\alpha_!X\otimes\alpha^*Y$ is an equivalence.
\end{proposition}

\begin{proof}
	By the description of Remark~\ref{rem:pushforward-in-exponential} and Lemma~\ref{lemma:cartesian-arrows-over-Orb} it follows that $\alpha^*\colon\DgSpc{\OR_H}\to\DgSpc{\OR_G}$ is given by precomposition along the functor $p_\alpha\colon\OR_H\to\OR_G$ obtained by basechange from $\OO_H^{\op}\to \OO_G^{\op}$. In particular it has a left adjoint $\alpha_!$ given by left Kan extension along $p_\alpha$.
	
	In the proof of Theorem~\ref{thm-presheaf-I-G-spaces} we have constructed a functor 
	$L_H \colon \OR_H^{\op} \to \DgTop{\I}{H}[W_{lvl}^{-1}] $ sending $(K,W)$ to $H \times_K \Li_W$.  We claim there is a commutative diagram 
	\[
	\begin{tikzcd}
		\OR_H^{\op} \arrow[d,"p_\alpha"'] \arrow[rr, bend left, "L_H"']\arrow[r,"\mathrm{Yoneda}"] & \DgSpc{\OR_H} \arrow[d,"\alpha_!"] \arrow[r,"\sim"] & \DgTop{\I}{H}[W_{lvl}^{-1}] \arrow[d,"G \times_H - "]\\
		\OR_G^{\op} \arrow[rr, bend right, "L_G"]\arrow[r,"\mathrm{Yoneda}"'] & \DgSpc{\OR_G} \arrow[r,"\sim"] & \DgTop{\I}{G}[W_{lvl}^{-1}]
	\end{tikzcd}
	\]
	where the horizontal equivalences are given by Theorem~\ref{thm-presheaf-I-G-spaces}.
	The diagram on the left commutes by the universal property of presheaf categories and the outer square commutes by direct verification using the formulas of $L_G$ and $L_H$. Therefore a generation argument using that all the functors preserve colimits, shows that the rightmost diagram commutes too.  The right most vertical functor can be modelled by a left Quillen functor by Proposition~\ref{prop-change-of-groups-functors} so the first claim follows.
	
	Finally, since the map
	\[G\times_H (X\otimes Y)\to (G\times_H X)\otimes Y\]
	is an isomorphism in $\DgTop{\I}{G}$, it follows that the derived formula holds as well.
\end{proof}

\begin{lemma}\label{lem:PSp_alpha-respects-Omega-spectra}
	Let $\alpha\colon H\to G$ be an injective homomorphism of compact Lie groups. Then $\PSp_\alpha\colon \PSp_G\to \PSp_H$ sends $\Sp_G$ into $\Sp_H$.
\end{lemma}
\begin{proof}
	Note that $\PSp_\alpha$ sends $X$ to $S_H\otimes_{\alpha^*S_G}\alpha^*X\simeq \alpha^*X$, since $\alpha$ is injective. Therefore $\PSp_\alpha$ preserves all small limits and colimits, since $\alpha^*$ does, and so it has a left adjoint $L_\alpha$. Moreover, by Lemma~\ref{lemma:projection-formula-for-injectives} there is an equivalence
	\[L_\alpha(X\otimes \PSp_\alpha Y)\simeq L_\alpha(X)\otimes Y\,.\]
	To prove that $\alpha^*(\Sp_G)\subseteq \Sp_H$ it suffices to show that $L_\alpha$ preserves stable equivalences. By confinality the stable equivalences in $\Sp_H$ are generated by those of the form $H\times_M \lambda_{V,W|_M}$ where $M<H$ is a closed subgroup, $V$ is an $M$-representation and $W$ is a $G$-representation. But then
	\[L_\alpha(H\times_M \lambda_{V,W|_M})\simeq L_\alpha((H\times_M F_VS^0)\otimes\alpha^*\lambda_{0,W|_H})\simeq L_\alpha(H\times_M F_VS^0) \otimes \lambda_{0,W}\,.\]
	Since stable equivalences are stable under tensoring and $\lambda_{0,W}$ is a stable equivalence, this proves the thesis.
\end{proof}

Given a compact Lie group $G\in \Glo$, we denote by $\Fgtglo{G}\colon \PSplax \rightarrow \PSp_G$ the canonical functors associated to the universal cone.

\begin{proposition}\label{prop-Lgl-localization}
	The $\infty$-category $\Splax$ is a Bousfield localization of $\PSplax$. 
	We denote the associated left adjoint by $\Llax\colon \PSplax \rightarrow \Splax$.
	Furthermore, the following conditions are equivalent for an object $X \in\PSplax$:
	\begin{itemize}
		\item[(a)] $X$ is in $\Splax$;
		\item[(b)] for every compact Lie group $G$, the $G$-prespectrum $\Fgtglo{G}(X)$ is in $\Sp_G$;
		\item[(c)] for every compact Lie group $G$, the $G$-prespectrum $\Fgtglo{G}X$ is local with respect to the maps $\lambda_{V,W}$ defined in Construction~\ref{cons-lambda-maps-equivariant} for any $G$-representations $V$ and $W$.
	\end{itemize}  
\end{proposition}

\begin{proof}
	Recall that $\Sp_{\bullet}$ was constructed in Proposition \ref{prop-transformation-localization-spectra} by localizing the functor $\PSp_\bullet$ using Lemma \ref{lem:adjunction-to-the-marked-lax-limit}. By the same lemma together with Lemma~\ref{lem:PSp_alpha-respects-Omega-spectra}, we conclude that $\Splax$ is a Bousfield localization and that conditions (a) and (b) are equivalent. By Proposition~\ref{prop-Sp_G-local}, condition (b) is equivalent to the condition that for every compact Lie group $G$ 
	and closed subgroup $H\leq G$, the $H$-prespectrum $\mathrm{res}^G_H\Fgtglo{G}X$ is local with respect to the maps $\{\lambda_{V,W}\}$ where $V$ and $W$ over all $H$-representations. By construction we have $\Fgtglo{H} = \mathrm{res}^G_H\circ \Fgtglo{G}$ so (b) and (c) are equivalent. 
\end{proof}

\section{Global spectra as a partially lax limit}\label{sec:global-spectra-par-lax-lim}

 Recall the functors $\PSp_\bullet,\Sp_\bullet\colon \Glo^{\op} \rightarrow \Cat^\otimes_\infty$ constructed in 
 Propositions~\ref{proposition:functoriality-of-prespectra} 
 and~\ref{prop-transformation-localization-spectra}. We also defined
 \[
 \PSplax:= \laxlimdag_{\Glo^{\op}}\PSp_G \quad \text{and} \quad \Splax := \laxlimdag_{\Glo^{\op}} \Sp_G.
 \]
    The goal of this section is to show that $\Splax$ is symmetric monoidally equivalent to 
    Schwede's $\infty$-category of global spectra $\Spgl$, whose definition is recalled in Definition~\ref{def-SpG-Spgl}.  Our proof will go roughly as follows:
    \begin{itemize}
    \item We will first construct a symmetric monoidal adjunction
    \[
    j_{!}\colon \PSpfgl \simeq \Mod_{\Sfgl}( \DgSpcp{\ORfgl}) \leftrightarrows \Mod_{\Slax}( \DgSpcp{\ORgl}) \simeq \PSplax \cocolon j^*
    \]
    between prespectra objects, where the equivalences are given by 
    Proposition~\ref{proposition:functoriality-of-prespectra} and Corollary~\ref{cor:PSpgl=PSpfgl}.
    \item We note that there are Bousfield localizations $\Spgl \subset \PSpfgl$ and 
    $\Splax \subset\PSplax$. We denote by $\Llax\colon \PSplax \rightarrow \Splax$ the localization functor.
    \item We will then check that $j^*$ preserves spectrum objects, and therefore obtain an 
    induced adjunction 
    \[L_{gl}\circ j_! \colon \Spgl\leftrightarrows \Splax \cocolon j^*\] between the 
    respective localizations. 
    \item We will show that this adjunction is in fact an equivalence, by showing that $j^*$ 
    is conservative on spectrum objects, and that the unit of the adjunction 
    $(L_{gl}\circ j_!, j^*)$ is an equivalence. 
    \end{itemize}

 We start by constructing an adjunction between prespectrum objects. By Lemma \ref{lem:faithful_sub_ORgl} we can identify $\ORf$ with the full subcategory of $\ORgl$ spanned by $(G,V)$, where $V$ is a faithful $G$-representation. Then the canonical inclusion $j\colon \ORf\hookrightarrow \ORgl$ induces an adjunction
\[
j_{!}\colon\DgSpcp{\ORfgl} \leftrightarrows \DgSpcp{\ORgl}\cocolon j^*.
\]
Note that $j_!$ is fully faithful as it is given as a left Kan extension along a fully faithful functor. Moreover the functor $j_{!}$ is strong monoidal by Proposition \ref{proposition:functoriality-of-presheaves}.

\begin{proposition}
	The inclusion $j\colon\ORf\hookrightarrow \ORgl$ admits a right adjoint $q$, which is given on objects by
    \[(G,V)\mapsto (G/\ker(V),V)\,,\]
    where $\ker(V)<G$ is the subgroup of $g\in G$ acting trivially on $V$.    
    In particular the left Kan extension $j_!$ is equivalent to the functor $q^*$ given by precomposition by $q$.
\end{proposition}

\begin{proof}
	The $G/\ker(V)$-representation $V$ is clearly faithful, so to prove the thesis it is enough to show that for every $(H,W)\in \ORf$ the map $(G/\ker(V),V)\to (G,V)$ induces an equivalence on mapping spaces
    \[\Map_{\ORgl}\left((H,W),(G/\ker V)\right)\xrightarrow{\sim}\Map_{\ORgl}\left((H,W),(G,V)\right)\,.\]
    By Definition~\ref{definition:ORgl}, this means we need to show that the map
    \[\left(\Hom(G/\ker V,H)\times\Li(W,V)\right)^{G/\ker V}_{hH}\to \left(\Hom(G,H)\times\Li(W,V)\right)^G_{hH}\]
    given by precomposition with $G\to G/\ker V$ on the first coordinate, is a homotopy equivalence. In fact we will show that
    \[\left(\Hom(G/\ker V,H)\times\Li(W,V)\right)^{G/\ker V}\to \left(\Hom(G,H)\times\Li(W,V)\right)^G\]
    is a homeomorphism. Since it is a continuous map of compact Hausdorff topological spaces, it suffices to show that it is bijective. As $\Hom(G/\ker V,H)\to \Hom(G,H)$ is injective, so is the above map. Therefore to conclude we need to show it is surjective.
    
    Concretely this means that if we have a map $\alpha\colon G\to H$ and an isometry $\varphi\colon W\to V$ that is $G$-equivariant, we need to show that $\alpha$ is trivial when restricted to $\ker V$. But if $g\in \ker V$, then $g$ acts as the identity on $V$, and therefore $\alpha(g)$ acts as the identity on $W$ (since $\varphi$ is $G$-equivariant). Since $W$ is a faithful $H$-representation this implies that $\alpha(g)=1$, as required.
\end{proof}

Note that it is clear from the definitions that $j^* \Slax \simeq \Sfgl$ as commutative algebra objects. As an application of the previous proposition we find:

\begin{corollary}\label{cor:glSvsftS}
The counit map $\epsilon\colon j_{!}\Sfgl \rightarrow \Slax$ is an equivalence of commutative algebra objects. In particular the functors $j_{!} \dashv j^*$ induce an adjunction
\[j_{!}\colon \PSpfgl \simeq \Mod_{\Sfgl} (\DgSpcp{\ORfgl} )\leftrightarrows \Mod_{\Slax}( \DgSpcp{\ORgl}) \simeq \PSplax \cocolon j^*\]
\end{corollary}

\begin{proof}
Because $j$ is strong monoidal, the counit is canonically a map of commutative algebra objects. Therefore for all $(G,V)\in \ORgl$ we compute \[j_!(\Sfgl)(G,V) \simeq \Sfgl(q(G,V)) = (S^{V})^{G/\ker(V)} \simeq (S^{V})^G = \Slax(G,V).\] Because $j_!$ and $j^*$ are strong and lax monoidal respectively and they swap the two algebra objects, they induce functors as in the statement which are evidently adjoint.
\end{proof}

We will now use the adjunction
\[
j_{!}\colon \PSpfgl \leftrightarrows  \PSplax \cocolon j^*
\] 
to induce an adjunction at the level of spectrum objects. To do this we need to see how the adjunction $(j_!,j^*)$ interacts with the full subcategories of spectrum objects. To this end we briefly rephrase the discussion of local objects in $\PSpfgl$ given at the end of Section \ref{sec-cat-of-eq-prespectra}. 

\begin{remark}\label{rmk:fglo_local_tested_after_j}
Recall from Proposition~\ref{prop-glspectra-local-object} that $\Spgl$ is a Bousfield localization of $\PSpfgl$ at the morphisms $\{\lambda_{G,V,W}\}$ where $G$ is a
compact Lie group and $V$ and $W$ are $G$-representations with $W$ faithful.
Because $j_{!} \colon \PSpfgl \to \PSplax$ is fully faithful, we can equivalently require that $j_{!} X$ is local with respect to the maps $j_{!} (\lambda_{G,V,W})$, where $W$ is a faithful representation. These maps again corepresent the $G$-fixed points of the adjoint structure map $\tilde{\sigma}_{G,V,W}$, and therefore we will denote them by $\lambda^\dagger_{G,V,W}$, and similarly we will write $F^\dagger_{G,V}$ for $j_!F_{G,V}$.
\end{remark}

We have seen in Construction~\ref{cons-lambda-maps} that for any compact Lie group $G$ and $G$-representation $V$, there is a functor $\mathrm{ev}_{G,V}\colon \PSpfgl\to \Spc_{G,\ast}$ that sends a faithful global prespectrum $X$ to the $G$-space $X(V)$. 
Under the equivalence 
\[
\PSpfgl\simeq \Mod_{\Sfgl} (\DgSpcp{\ORfgl} )
\]
this functor can be modelled as follows. Consider the cocartesian section $s_V\colon \OO_G^{\op}\rightarrow \OR_G$ which is determined by the object $(G,V)\in\OR_G$ and write $k_V$ for the composite $\OO_G^{\op}\xrightarrow{s_V} \OR_G\xrightarrow{\nu_G} \ORgl$. If $V$ is faithful then $k_V$ lands in $\ORfgl$ and so we can define $\mathrm{ev}_{G,V}$ as the following composite of right adjoints
\[ 
\Mod_{\Sfgl} (\DgSpcp{\ORfgl} )\xrightarrow{\mathrm{fgt}} \DgSpcp{\ORfgl} \xrightarrow{k_V^*} \Spc_{G,\ast}.
\] 
Similarly, as discussed in Construction~\ref{cons-lambda-maps-equivariant}, there is a functor $\mathrm{ev}_V\colon \PSp_G \to \Spc_{G,\ast}$ sending a $G$-prespectrum $X$ to the $G$-space $X(V)$. Under the equivalence 
\[
\PSp_G \simeq \Mod_{S_G}(\DgSpcp{\OR_G})
\]
this functor is modelled by the composite 
\[
\Mod_{S_G}(\DgSpcp{\OR_G}) \xrightarrow{\mathrm{fgt}} \DgSpcp{\OR_G} \xrightarrow{s_V^*} \Spc_{G,\ast}\, ,
\]
see also Remark~\ref{rem:free_ORG_space}.

\begin{remark}
From the previous discussion we conclude that there is a commutative diagram of right adjoints:
\[
\begin{tikzcd}
\PSpfgl \arrow[d,"\sim"'] & \PSplax \arrow[d,"\sim"] \arrow[l,"j^*"'] \arrow[r,"U_G^{gl}"] & \PSp_G \arrow[d,"\sim"] \\
\Mod_{S_{fgl}}(\DgSpcp{\ORfgl}) \arrow[d,"\mathrm{fgt}"'] & \Mod_{S_{gl}^\dagger}(\DgSpcp{\ORgl}) 
\arrow[l,"j^*"'] \arrow[d,"\mathrm{fgt}"] \arrow[r,"\nu_G^*"] & \Mod_{S_G}(\DgSpcp{\OR_G}) \arrow[d,"\mathrm{fgt}"] \\
\DgSpcp{\ORfgl} \arrow[dr,"k^*_W"'] & \DgSpcp{\ORgl} \arrow[l,"j^*"'] \arrow[d,"k_W^*"]\arrow[r,"\nu_G^*"] & \OR_G-\Spcp \arrow[dl,"s^*_W"] \\
 & \Spc_{G,\ast} & \;.\\
\end{tikzcd}
\]
Using that the corresponding diagram of left adjoints commute, we see that for all $X\in\PSplax$ and $G$-representations $V$ and $W$ with $W$ faithful, the following diagram commutes 
\begin{equation}\label{diagram-lambda-maps}
\begin{tikzcd}
	{\Spc_{G,\ast}(S^0,X(W))} \arrow[rr, "{\tilde{\sigma}_{V,W}}"]  & & {\Spc_{G,\ast}(S^V,X(V\oplus W))} \arrow[d, "\sim"]                  \\
	{\PSp_G(F_{W} S^0,\Fgtglo{G}(X))} \arrow[u, "\sim"] \arrow[rr,"\lambda_{V,W}^*"] & & {\PSp_G(F_{V\oplus W} S^V,\Fgtglo{G}(X))} \arrow[d, "\sim"] \\
	{\PSpfgl(F_{G,W}S^0, j^*X)} \arrow[u, "\sim"] \arrow[rr,"\lambda_{G,V,W}^*"]      &      & {\PSpfgl(F_{G,V\oplus W}S^V, j^*X)}                         \arrow[d,"\sim"'] \\
	\PSplax(F^\dagger_{G,W}S^0, X) \arrow[u, "\sim"]\arrow[rr,"(\lambda^\dagger_{G,V,W})^*"] & & \PSplax(F^\dagger_{G,V\oplus W}S^V, X)
\end{tikzcd}
\end{equation}
so all the various $\lambda$-maps correspond to each other under the various adjunctions. 
\end{remark}

Given any compact Lie group $G$ and any faithful $G$-representation $W$, we define a functor \[\Fgtfgl{G}{W}\colon \PSpfgl\rightarrow \PSp_{G}\] as the composite 
\[
\PSpfgl\xrightarrow{j_{!}} \PSplax \xrightarrow{\Fgtglo{G}} \PSp_{G} \xrightarrow{\sh_{W}} \PSp_{G}
\] 
where $\sh_{W}$ denotes the shift $W$-functor, given by cotensoring by $F_W S^0$. 

\begin{theorem}\label{thm:fibrant_glo}
An object $X\in \PSpfgl$ is in $\Spgl$ if and only if for every compact Lie group $G$ and faithful $G$-representation $W$, the object $\Fgtfgl{G}{W}(X)$ is in $\Sp_G$. Furthermore, the functors $\{\Fgtfgl{G}{W}\}_{(G,W)}$ are also jointly conservative. 
\end{theorem}

\begin{proof}

 By Remark~\ref{rmk:fglo_local_tested_after_j}, we know that $X\in\PSpfgl$ is in $\Spgl$ 
 if and only if $j_!X\in\PSplax$ is local with respect to the set of maps 
 $\{\lambda_{G,V,W}^\dagger\}$ where $G$ runs over all compact Lie groups and 
 $V$ and $W$ are $G$-representations with $W$ faithful. 
 The commutative diagram~(\ref{diagram-lambda-maps}) (together with the fact that $j^*j_!X\simeq X$) shows that this is equivalent to 
 asking that for all compact Lie groups $G$, the object $U_G^{gl}(j_!X)$ is local with respect 
 to $\{\lambda_{G,V,W}\}$ where $V$ and $W$ are as above.  
 
We next note that by definition, given an abitrary $G$ prespectrum $Y$, the map \[\lambda_{U,V}^*\colon 	{\PSp_G(F_{V} S^0,\sh_W Y)}  \rightarrow {\PSp_G(F_{U\oplus V} S^U,\sh_W Y)}\] is equivalent to $\lambda_{U, V\oplus W}^*$. Also recall that given a faithful $G$-representation $W$, $W\oplus U$ is also faithful for any $G$-representation $U$. 

These two observations combine to imply that $U_G^{gl}(j_!X)$ is local with respect to $\{\lambda_{V,W}\}$ for $G,V$ and $W$ as above if and only if for all compact Lie groups $G$ and 
 faithful $G$-representations $W$, the object 
 $\sh_W U_G^{gl}j_!(X)=U_{G,W}^{fgl}X$ is local with respect to $\{\lambda_{V,U}\}$ 
 for arbitrary $G$-representations $V$ and $U$.
  
On the other hand by Proposition~\ref{prop-Sp_G-local}, $U_{G,W}^{fgl}X$ is in $\Sp_G$ if and only if for all closed subgroups $H \leq G$, the $H$-prespectrum $\mathrm{res}^G_H U_{G,W}^{fgl}X=U_{H,\mathrm{res}^G_H W}^{fgl}X$ is local with respect to 
$\{\lambda_{V,U}\}$ for arbitrary $H$-representations $V$ and $U$, and $W$ a 
 faithful $G$-representation. Varying these statements over all compact Lie groups, we find that $\Fgtfgl{G}{W} X$ is in $\Sp_G$ for all compact Lie groups $G$ and all faithful $G$-representations $W$ if and only if for all $G$ and all faithful $G$-representations $W$, the $G$-prespectrum $\Fgtfgl{G}{W}X$ is $\{\lambda_{V,U}\}$-local for arbitrary $G$-representation $V$ and $U$. This is identical to the condition of the previous paragraph, and so we obtain the first claim in the theorem. 
 For the second statement, note that after forgetting module structures, the functor $\Fgtfgl{G}{W}$ is given by restriction along the functor \[\sh_{W}\colon\OR_{G}\rightarrow\ORfgl, (G/H,U)\mapsto (H,U\oplus \mathrm{res}^G_H(W)).\] The claim then follows from the fact that the functors $\{\sh_{W}\}_{(G,W)}$ where $G$ runs over all compact Lie groups and $W$ 
 all faithful $G$-representations, are jointly essentially surjective.
\end{proof}

The following is the key fact about the right adjoint $j^*$.

\begin{proposition}\label{prop:fgt_comm}
Let $G$ be a compact Lie group and let $W$ be a faithful $G$-representation. Then the following square commutes:
\begin{center}
\begin{tikzcd}
	\PSp_G \arrow[d, "\sh_W"'] & \PSplax \arrow[l, "\Fgtglo{G}"'] \arrow[d, "j^{*}"]                       \\
	\PSp_G                                     & \PSpfgl \arrow[l, "\Fgtfgl{G}{W}"']  .
\end{tikzcd}
\end{center}
\end{proposition}

\begin{proof}
The unit of the adjunction $j_{!}\dashv j^*$ provides a natural transformation 
\[\Fgtfgl{G}{W}j^* =\sh_W\Fgtglo{G} j_! j^* \rightarrow \sh_W\Fgtglo{G}\] 
which we claim is a natural equivalence. This follows from the fact that on underlying objects $\sh_W U^{gl}_G$ is given by restriction along the functor $\OR_{G}\rightarrow \ORgl, (H,V)\mapsto (H,\mathrm{res}^G_H(W)\oplus V)$. This only sees levels in the image of $\ORfgl$, where the unit is an equivalence.
\end{proof}

\begin{corollary}
	Suppose $X\in \Splax$. Then $j^*(X)\in \Spgl$. In particular we obtain a functor \[j^*\colon \Splax\rightarrow \Spgl,\] which admits a left adjoint given by $\Llax\circ j_{!}$.
\end{corollary}

\begin{proof}
	Because $X$ is in $\Splax$, we obtain that $\Fgtglo{G}(X)$ is a $G$-spectrum 
	by Proposition~\ref{prop-Lgl-localization}. Note that the functor $\mathrm{sh}_W$ preserves $G$-spectra for every $G$-representation $W$. We deduce using Proposition \ref{prop:fgt_comm} that $\Fgtfgl{G}{W}j^{*}(X)$ is a $G$-spectrum for every $G$ and $W$ faithful. Therefore by Theorem \ref{thm:fibrant_glo} $j^*(X)$ is contained in $\Spgl$.
\end{proof}

\begin{proposition}
$j^{*}\colon \Splax \rightarrow \Spgl$ is conservative. 
\end{proposition}

\begin{proof}
Let $f\colon X\rightarrow Y$ be a map in $\PSplax$ such that $j^*(f)$ is an equivalence. This implies that $f_{(G,W)}$ is an equivalence of spaces for every faithful $G$-representation $W$. We finish the argument by proving that if $f$ is in fact a map between objects in $\Splax$, then $f_{(G,V)}$ is an equivalence for \textit{every} $G$-representation $V$ if and only if it is an equivalence for \textit{faithful} $G$-representations. The forward direction is trivial. For the converse, note that because $\PSplax$ is a partially lax limit, the collection of functors $\{\Fgtglo{G}\}_{G}$ is jointly conservative. Now our assumptions tell us that $\Fgtglo{G}(f)_{(G,W)}$ is an equivalence for every faithful $G$-representation $W$. But because $f$ is in fact in $\Splax$, both the source and target of $\Fgtglo{G}(f)$ are $G$-spectra. Therefore our claim reduces to the fact that a map between $G$-spectra which is an equivalence on faithful levels, is already an equivalence. The collection of faithful representations is cofinal in all representations, and so this is clear.
\end{proof}

\begin{theorem}
	The unit of the adjunction $$\Llax\circ j_{!}\colon \Spgl\leftrightarrows \Splax\cocolon j^*$$ is an equivalence.
\end{theorem}

\begin{proof}
Consider $X\in \Spgl$. Let $\eta_X\colon X\rightarrow j^{*} \Llax j_{!} X$ be the unit of the adjunction $\Llax\circ j_{!} \leftrightarrows j^*$  evaluated at $X$. This adjunction is given as a composite of two adjunctions and so the unit is given by the composite
\[
X\xrightarrow{\eta'} j^* j_{!} X \xrightarrow{j^* (\gamma)} j^* \Llax j_{!} X,
\] where $\eta'$ is the unit of the adjunction $j_{!}\dashv j^*$ and $\gamma$ exhibits $\Llax j_{!} X$ as the localization of $j_{!} X$ in $\PSplax$. However recall that $j_{!}$ is fully faithful and therefore the first of the two maps is an equivalence. So it suffices to prove that the second map is also an equivalence.

The functors $\Fgtfgl{G}{W}$ are jointly conservative, and so we will prove that $\Fgtfgl{G}{W}(j^*(\gamma))$ is an equivalence for every $(G,W)$, where $W$ is faithful. Applying Proposition \ref{prop:fgt_comm} we conclude that $\Fgtfgl{G}{W}(j^*(\gamma))$ is equivalent to \[\mathrm{sh}^W \Fgtglo{G}(\gamma): \mathrm{sh}^W \Fgtglo{G} j_{!} X \rightarrow \mathrm{sh}^W \Fgtglo{G} \Llax j_{!} X.\] By Proposition~\ref{prop-Lgl-localization}, $\Fgtglo{G}(\gamma)$ is equivalent to \[\gamma_G\colon \Fgtglo{G}j_{!} X\rightarrow L_G\Fgtglo{G}j_{!} X,\] where $\gamma_G$ exhibits $L_G\Fgtglo{G}j_{!} X$ as the localization of $\Fgtglo{G}j_{!} X$ in $\PSp_G$. Spectrification of $G$-prespectra commutes with $\mathrm{sh}^W$, and therefore $\mathrm{sh}^W (\gamma_G)$ gives the localization of $\Fgtfgl{G}{W}(X) = \mathrm{sh}^W  \Fgtglo{G} j_{!} X$ in $\PSp_{G}$.  Recall that $X\in \Spgl$, and so $\Fgtfgl{G}{W}(X)$ is a $G$-$\Omega$-spectrum by Theorem \ref{thm:fibrant_glo}. Therefore $\mathrm{sh}^W(\gamma_G)$ is an equivalence, concluding the proof.
\end{proof}

\begin{theorem}\label{thm-laxlim-global-spectra}
There is a symmetric monoidal equivalence $j^{*}\colon\Splax:=\laxlimdag_{G} \Sp_G \rightarrow \Spgl$.
\end{theorem}

\begin{proof}
We have proven that $j_{!}\dashv j^*$ is an adjunction in which the right adjoint is conservative, and the unit is a natural equivalence. Therefore the functors are an adjoint equivalence. Moreover $j_{!}$ is strong monoidal, which implies that $j^*$, as its inverse, is also strong monoidal.
\end{proof}

\section{Proper equivariant spectra as a limit}\label{sec-proper-section}
The goal of this section is to exhibit the $\infty$-category of genuine proper $G$ spectra $\Sp_G$ as a limit over the proper orbit category $\OO_{G,\pr}^{\op}$ of a diagram \[\Sp_{(-)}\colon \OO_{G,\pr}^{\op}\rightarrow \Cat_\infty, \qquad G/H\rightarrow \Sp_{H}.\] In contrast to the case of global spectra, once the diagram has been constructed, the identification of the limit will be almost immediate. In fact even the general strategy for constructing the diagram is essentially identical. For this reason we will be brief and refer to Section \ref{sec:funct_prespectra} for the relevant details.

Recall from Lemma \ref{lem:ORG-as-a-pullback} that the $\infty$-operad $\OR_G^\otimes$ fits into a pullback 
   \[\begin{tikzcd}
	\OR_G^\otimes\ar[r, "\nu_G"]\ar[d, "\pi_G"'] & \OR_{gl}^\otimes\ar[d,"\pi_{gl}"]\\
	(\OO_{G,\pr}^{\op})^\amalg\ar[r, "\iota_G^{\amalg}"] & (\Glo^{\op})^\amalg.
\end{tikzcd}\]

Because $\OR_{gl}^\otimes \rightarrow (\Glo^{\op})^\amalg$ is a cocartesian fibration which by definition classifies the functor $\mathrm{Rep}(-)$, we immediately obtain:

\begin{proposition}
For every Lie group $G$, the forgetful functor 
$\pi_G \colon \OR_G^\otimes \rightarrow (\OO_{G,\pr}^{\op})^\amalg$ is a cocartesian fibration which 
classifies the functor \[\OO_{G,\pr}^{\op}\rightarrow \Cat_\infty^\otimes, \quad G/H \mapsto \mathrm{Rep}(H).\]
\end{proposition}

\begin{definition}
	We define $\ttOR_G^\otimes$ via the following pullback of operads:
	\[
	\begin{tikzcd}
		\ttOR_G^\otimes \arrow[d, "\pi_{\Ar}"] \arrow[r]        & \OR_G^\otimes \arrow[d] \\		(\Ar(\OO_{G,\pr})^{\op})^\amalg \arrow[r, "s^{\op}"] & (\OO_{G,\pr}^{\op})^\amalg      
	\end{tikzcd}
	\]
\end{definition}

We consider $\ttOR_G^\otimes$ as living over $\OO_{G,\pr}$ via the composite
\[\pi\colon \ttOR_G^{\otimes} \xrightarrow{\pi_{\Ar}} (\Ar(\OO_{G,\pr})^{\op})^{\amalg} \xrightarrow{t^{\op}} (\OO_{G,\pr}^{\op})^{\amalg}.\]
Just as in Lemma \ref{lemma:tOR-promonoidal}, we can show that $\ttOR_G^\otimes$ is a pro-$(\OO_G)^\amalg$-monoidal category.

\begin{proposition}
The functor $\pi\colon \ttOR_G \rightarrow (\OO_{G,\pr}^{\op})$, given by restricting $\pi$ to underlying categories, is a cartesian fibration. Furthermore an edge $(f,g)\in \ttOR_G$ is cartesian if and only if $s^{\op}(f)$ and $g$ are equivalences.
\end{proposition}

\begin{proof}
The proof is analogous to Lemma \ref{lemma:cartesian-arrows-over-Orb}.
\end{proof}	

\begin{proposition}\label{prop:fibers_tORG}
$\ttOR_G^\otimes \times_{(\OO_G^{\op})^\amalg} \{G/H\} \simeq \OR_H^\otimes$.
\end{proposition}

\begin{proof}
The pullback $P=\ttOR_G^\otimes \times_{(\OO_G^{\op})^\amalg} \{G/H\}$ fits into the following diagram
\[
\begin{tikzcd}
	P \arrow[d] \arrow[r]                    & \ttOR_G \arrow[d] \arrow[r]                         & \OR_G^\otimes \arrow[d] \arrow[r]    & \ORgl^\otimes \arrow[d] \\
	(\OO_{H}^{\op})^\amalg \arrow[d] \arrow[r] & {(\Ar(\OO_{G,\pr})^{\op})^\amalg} \arrow[d] \arrow[r] & {(\OO_{G,\pr}^{\op})^\amalg} \arrow[r] & (\Glo^{\op})^\amalg           \\
	\{G/H\} \arrow[r]                        & {(\OO_{G,\pr}^{\op})^\amalg}                          &                                      &                            
\end{tikzcd}
\] in which every square is a pullback. One can show by direct computation that the middle composite $(\OO_H^{\op})^\amalg \rightarrow (\Glo^{\op})^\amalg$ is equivalent to $\iota_H^\amalg$. Therefore the result follows from Lemma \ref{lem:ORG-as-a-pullback}.
\end{proof}
	
\begin{definition}
Consider the Day convolution operad \[\Fun_{\OO_{G,\pr}}(\ttOR_G^\otimes,\Spcp^\wedge \times (\OO_{G,\pr}^{\op})^\amalg)^{Day}.\] Just as in Section \ref{sec:funct_prespectra}, this is an $(\OO_{G,\pr}^{\op})^\amalg$-monoidal category. We define \[\DgSpcp{\OR_\bullet}\colon \OO_{G,\pr}^{\op} \rightarrow \Cat_\infty^\otimes\] to be the functor associated to it by the equivalence of Proposition \ref{prop:equiv_amalg}. By Proposition \ref{prop:fibers_tORG}, the value of $\DgSpcp{\OR_\bullet}$ at $G/H$ is equivalent to $\DgSpcp{\OR_H}.$
\end{definition}

\begin{proposition}\label{prop:tORG_localization}
The projection map $\ttOR_G^\otimes \rightarrow \OR_{G}^\otimes$ induces a fully faithful symmetric monoidal functor $\DgSpcp{\OR_G} \rightarrow \DgSpcp{\ttOR_G}$, given by restriction. A functor $F\colon \ttOR_G\rightarrow \Spcp$ is in its essential image if and only if $F$ sends $\pi$-cartesian edges to equivalences.
\end{proposition}

\begin{proof}
The argument is identical to that of Lemma \ref{lem:ORgl-is-localization-of-tOR}.
\end{proof}

\begin{lemma}
There is a symmetric monoidal equivalence
\[
\lim_{\OO_{G,\pr}^{\op}} \DgSpcp{\OR_\bullet} \simeq \DgSpcp{\OR_G}.
\]
\end{lemma}

\begin{proof}
The calculation at the beginning of the proof of Lemma \ref{lem:laxlim-of-OR-spaces} shows that the lax limit of the diagram $\DgSpcp{\OR_\bullet}$ is equivalent to the symmetric monoidal category $\DgSpcp{\ttOR_G}$ To compute the actual limit, we can once again argue on underlying categories by appealing to Remark \ref{rem-partially-lax-underlying}. Note that by Remark \ref{rem-norm-underlying}, the underlying category of $\DgSpcp{\ttOR_G}$ is equivalent to $\Fun(\ttOR_G,\Spcp)$. The analysis of the second half of the proof of Lemma \ref{lem:laxlim-of-OR-spaces} implies that the limit is equivalent to the full subcategory spanned by the functors which send $\pi$-cartesian edges to equivalences. By Proposition \ref{prop:tORG_localization} this subcategory is equivalent to $\Fun(\OR_{G},\Spcp)$.
\end{proof}

Recall from Definition \ref{def-prespectra} that $\OR_G$-spaces admit an algebra object $S_G$, whose restriction to $\OR_H$-spaces for $H$ a compact subgroup of $G$ is equivalent to $S_H$.

\begin{corollary}
There exists a functor $\PSp_\bullet\colon \OO_{G,\pr}^{\op}\rightarrow \Cat_\infty^\otimes$, and one calculates
\[\lim_{\OO_{G,\pr}^{\op}} \PSp_\bullet \simeq \Mod_{S_G}(\DgSpc{\OR_G})\]
\end{corollary}

\begin{proof}
Once again $\PSp_\bullet$ is defined as $\Mod_{S^\bullet}(\DgSpcp{\OR_\bullet})$, using Theorem~\ref{thm:modules-in-laxlimit}. An argument as in Proposition \ref{proposition:functoriality-of-prespectra} allows us to calculate the limit.
\end{proof}

So far we have constructed and computed the limit of the diagram $\PSp_\bullet\colon \OO_{G,\pr}^{\op}\rightarrow \Cat_\infty^\otimes$. Given a map $\alpha\colon H \hookrightarrow K\subset G$ in $\OO_{G,\pr}$, the induced map $\PSp_K\rightarrow \PSp_H$ is by construction equivalent to the global functoriality constructed in Section \ref{sec:funct_prespectra} evaluated at $\alpha$. Therefore the results there imply that $\PSp_\alpha$ preserves spectrum objects, and so we obtain a diagram $\Sp_\bullet\colon \OO_{G,\pr}^{\op}\rightarrow \Cat_\infty^\otimes$. Furthermore, Corollary \ref{cor:funct-Sp_bullet} implies that $\Sp_\alpha\colon \Sp_K\rightarrow \Sp_H$ agrees with the standard restriction functor between equivariant spectra. To calculate the limit of $\Sp_\bullet$, we apply Lemma \ref{lem:adjunction-to-the-marked-lax-limit} to conclude:

\begin{corollary}
 $\lim_{\OO_{G,\pr}^{\op}} \Sp_\bullet$ is a Bousfield localization of $\Mod_{S_G}(\DgSpcp{\OR_G})$ at the objects $X$ whose restriction to $\Mod_{S_H}(\DgSpcp{\OR_H})$ is an $H$-spectrum for every compact subgroup $H$ of $G$.
\end{corollary}

Recall from Section \ref{sec:eq_prespectra} that the category of genuine proper $G$-spectra is also a Bousfield localization of $\Mod_{S_G}(\DgSpcp{\OR_G})$. Therefore it remains to show that the two subcategories agree.

\begin{proposition}
An object $X\in \Mod_{S_G}(\DgSpcp{\OR_G})$ is a $G$-spectrum if and only if for every compact subgroup $H\leq G$, the restriction of $X$ to $\Mod_{S_H}(\DgSpcp{\OR_H})$ is a $H$-spectrum.
\end{proposition}

\begin{proof}
Recall from Proposition \ref{prop-Sp_G-local} that an object $X\in \PSp_G$ is a $G$-spectrum if and only if for all compact subgroups $H\leq G$, the object $\mathrm{res}^G_H X$ is local with respect to $\lambda_{H,V,W}$. Now by definition $\mathrm{res}^G_H X$ is a $G$-spectrum if and only if $\mathrm{res}^H_K\mathrm{res}^G_H X$ is local with respect to $\lambda_{K,V,W}$. However because $\mathrm{res}^H_K\mathrm{res}^G_H = \mathrm{res}^G_K$, we conclude that the two conditions of the theorem agree.
\end{proof}

Thus we can conclude the main theorem of this section:

\begin{theorem}\label{thm-lim-proper-spectra}
	The category of proper $G$-spectra is equivalent to the limit of the diagram $\Sp_\bullet\colon \OO_{G,\pr}^{\op}\rightarrow \Cat_\infty^\otimes$, in symbols
	\[\Sp_{G} \simeq \lim_{\OO_{G,\pr}^{\op}} \Sp_\bullet.\]
\end{theorem}

\part{Appendix}

\appendix
\section{Tensor product of modules in an \texorpdfstring{$\infty$}{infty}-category}

    The goal of this section is to provide a proof of Theorem~\ref{theorem:modules-are-pullbacks} below, which will be useful when studying lax limits of $\infty$-categories of modules. This section uses some technical results about the theory of $\infty$-operads as developed in \cite{HA} and so it should be skipped on a first reading.

    \begin{definition}\label{def:module-operad}
    We define $\CMod^\otimes$ to be the $\infty$-operad corresponding to the symmetric multicategory with two objects $a$ and $m$ with
    \[\Mul(\{x_i\},a)=\begin{cases}\ast & \textrm{if }\forall i,\, x_i=a\\ \varnothing&\textrm{otherwise}\end{cases}\quad \Mul(\{x_i\},m)=\begin{cases} \ast & \textrm{ if }|\{i\mid x_i=m\}|=1\\ \varnothing&\textrm{otherwise.}\end{cases}\]
    We know by \cite[Proposition~7]{Glasman-modules} or \cite[Lemma~B.1.1]{Hinich-rectification} that for every $\infty$-operad $\CC^\otimes$ there is a natural equivalence of $\infty$-categories
    \[\Mod^{\Finp}(\CC)\simeq \Alg_{\CMod^\otimes}(\CC)\,.\]
    Our goal is to give a similar description of the tensor product of modules over a commutative algebra, that is of the family of $\infty$-operads $\Mod(\CC)^\otimes$. In order to do so we will introduce a variant of $\CMod^\otimes$ which parametrizes finite sets of modules.
    \end{definition}
    
    \begin{construction}
    Let $\tCMod^\otimes$ be the category whose objects are triples $(\langle n\rangle, \langle m\rangle, \{S_i\}_{i=1,\dots,n})$ where $\langle n\rangle,\langle m\rangle\in \Finp$ and $\{S_i\}$ is a family of pairwise disjoint subsets of $\langle m\rangle$. A map $(\langle n\rangle,\langle m\rangle,\{S_i\})\to (\langle n'\rangle,\langle m'\rangle,\{S'_i\})$ is a pair of maps $f:\langle n\rangle\to \langle n'\rangle$ and $g:\langle m\rangle\to \langle m'\rangle$ in $\Finp$ such that
    \begin{itemize}
        \item for every $i\in \langle n\rangle^\circ$, we have $g(S_i)\subseteq S'_{f(i)}\cup\{\ast\}$ (where $S'_{\ast}=\varnothing$)
        \item for every $i\in f^{-1}\langle n'\rangle^\circ$ and every $s'\in S'_{f(i)}$ there is exactly one $s\in S_i$ such that $g(s)=s'$.
    \end{itemize}
    \end{construction}
    
    \begin{lemma}
        The projection $\tCMod^\otimes\to \Finp\times\Finp$ that forgets the subsets $\{S_i\}$ is a $\Finp$-family of $\infty$-operads in the sense of \cite[Definition~2.3.2.10]{HA}, with inert arrows exactly those arrows that are sent to an equivalence by the first projection and to an inert arrow by the second projection.
    \end{lemma}
    \begin{proof}
        The inert arrows are the arrows
        \[(\operatorname{id}_{\langle n\rangle}, f):(\langle n\rangle, \langle m\rangle, \{S_i\})\to (\langle n\rangle,\langle m'\rangle, \{f(S_i)\cap \langle m'\rangle^\circ\})\]
        where $f:\langle m\rangle\to \langle m'\rangle$ is an inert arrow in $\Finp$. It is easy to check that they satisfy all necessary properties.
    \end{proof}
    
    \begin{notation}
        For every $\infty$-category $X\to \Finp$ with a functor to $\Finp$ we will write $\tCMod_X^\otimes$ for the $X$-family of $\infty$-operads $X\times_{\Finp}\tCMod^\otimes$, where we pullback along the composite \[\tCMod^\otimes\to \Finp\times\Finp\xrightarrow{\pr_1} \Finp.\] Note that $\tCMod^\otimes_{\langle 1\rangle}$ is equivalent to $\CMod^\otimes$. Intuitively the fibre $\tCMod_{\langle n\rangle}^\otimes$ is the $\infty$-operad controlling pairs $(A,\{M_i\})$ where $A$ is a commutative algebra and $\{M_i\}$ is an $n$-tuple of $A$-modules.
        
        We will write $a_n$ for the object $(\langle n\rangle, \langle 1\rangle, \{\varnothing\})$ and $m_{j,n}$ for the object $(\langle n\rangle,\langle 1\rangle,\{S_i\})$ where $S_i=\varnothing$ for $i\neq j$ and $S_j=\{1\}$. It's easy to see these are all the objects of the underlying category of the generalized operad \[\tCMod^\otimes\rightarrow \Finp\times \Finp \xrightarrow{\pr_2} \Finp\]
    \end{notation}
    
    First we will prove a generalization of \cite[Proposition~7]{Glasman-modules} that shows how $\tCMod^\otimes$ controls the tensor product of modules over commutative algebras.
    
    \begin{proposition}\label{prop:modules-using-tCMod}
        Let $X\in (\Cat_\infty)_{/\Finp}$ be an $\infty$-category over $\Finp$ and let $\CC^\otimes\in \Op_\infty$ be an $\infty$-operad. Then there is a natural equivalence
        \[\Alg_{\tCMod_X}(\CC^\otimes)\simeq\Fun_{/\Finp}(X,\Mod^{\Finp}(\CC)^\otimes).\]
    \end{proposition}
    \begin{proof}
            Let $\CK\subseteq \Ar(\Finp)$ be the full subcategory of semi-inert arrows \cite[Notation~3.3.2.1]{HA}. Consider the pullback
            \[
            \begin{tikzcd}
            	X\times_{\Finp}\CK \arrow[d] \arrow[r] & X \arrow[d] \\
            	\CK \arrow[r, "s"] \arrow[d, "t"]      & \Finp       \\
            	\Finp.                                  &            
            \end{tikzcd}
            \]We will say that an arrow $(f,g)$ in $X\times_{\Finp}\CK$ is inert if $f$ is an equivalence and $t(g)$ is an inert edge of $\Finp$(this is different from the convention in \cite{HA}, but it is more suited to the current proof). Then recall that by \cite[Construction~3.3.3.1]{HA} the $\infty$-category $\Mod(\CC)^\otimes$ is defined so that there is a natural fully faithful inclusion
            \[\Fun_{/\Finp}(X,\Mod^{\Finp}(\CC)^\otimes)\to \Fun_{/\Finp}(X\times_{\Finp}\CK,\CC^\otimes),\]
            where $X\times_{\Finp}\CK$ lives over $\Finp$ by the vertical composite in the diagram above, with essential image those functors sending inert arrows of $X\times_{\Finp}\CK$ to inert arrows.
            
            There is a functor $\CK\to \tCMod$ sending a semi-inert arrow $[s\colon \langle n\rangle\to\langle m\rangle]$ to $(\langle n\rangle,\langle m\rangle, \{\{s(i)\}\cap \langle m\rangle^\circ\}_i)$. It identifies $\CK$ with the full subcategory of $\tCMod$ spanned by those triples $(\langle n\rangle,\langle m\rangle,\{S_i\})$ where $|S_i|\le 1$ for every $i\in \langle n\rangle^\circ$. Moreover an arrow in $X\times_{\Finp}\CK$ is inert if and only if its image in $\tCMod_X$ is inert. Therefore restricting along this inclusion induces a natural transformation
            \[\Alg_{\tCMod_X}(\CC^\otimes)\to\Fun_{/\Finp}(X,\Mod^{\Finp}(\CC)^\otimes)\,.\]
            Our goal now is to prove this is an equivalence of $\infty$-categories. This follows from \cite[Proposition~4.3.2.15]{HTT} together with the following two statements, where we write $p\colon \CC^\otimes\rightarrow \Finp$ for the structure map of $\CC^\otimes$:
            \begin{enumerate}
                \item Every map $F:X\times_{\Finp}\CK\to \CC^\otimes$ over $\Finp$ that sends inert arrows to inert arrows admits a right $p$-Kan extension to $\tCMod_X$ that sends inert arrows to inert arrows;
                \item A functor $F:\tCMod_X\to \CC^\otimes$ which sends inert arrows to inert arrows is the right $p$-Kan extension of its restriction to $X\times_{\Finp}\CK$.
            \end{enumerate}
        
            Let $(x,\langle m\rangle,\{S_i\})$ be an object of $\tCMod_X$ and write $S=\coprod_i S_i\subseteq \langle m\rangle^\circ$. Let us consider the functor
            \[\CP(S)^{\op}\to \tCMod_X\]
            sending a subset $A\subseteq S$ to $(x,\langle m\rangle/(S\smallsetminus A), \{A\cap S_i\})$ and all arrows to inert arrows. This induces a functor
            \[\CP(S)^{\op}\to (\tCMod_X)_{(x,\langle m\rangle,\{S_i\})/}\,,\]
            which sends $A$ to the inert morphism collapsing all elements of $S$ not in $A$ to the basepoint. If we let $\CQ(S)\subseteq \CP(S)$ be the subposet of those elements $A$ such that $|A\cap S_i|\le 1$ for every $i$ we obtain a functor
            \[\CQ(S)^{\op}\to (X\times_{\Finp}\CK)_{(x,\langle m\rangle,\{S_i\})/}\]
            to the comma category, which has a right adjoint given by
            \[\left[(f,g)\colon (x,\langle m\rangle,\{S_i\})\to (x',\langle m'\rangle,\{S'_i\})\right]\mapsto g^{-1}\left(\coprod_i S'_i\right)\cap S\, ,\]
            and therefore is coinitial. Thus, by \cite[Proposition~4.3.1.7 and Lemma~4.3.2.13]{HTT} it suffices to show the following two conditions
            \begin{enumerate}
                \item Let $F:X\times_{\Finp}\CK\to \CC^\otimes$ sending inert arrows to inert arrows, then the composition
                \[\CQ(S)^{\op}\to X\times_{\Finp}\CK\to \CC^\otimes\]
                has a $p$-limit diagram sending all edges to inert edges.
                \item Let $F:\tCMod_X\to \CC^\otimes$ sending inert arrows to inert arrows, then the composition
                \[(\CQ(S)^{\op})^\triangleleft\to \CP(S)^{\op}\to \tCMod_X\to \CC^\otimes\]
                is a $p$-limit diagram, where the first functor sends the cone point to $S\subseteq S$.
            \end{enumerate}
            Both of them are now an immediate consequence of the characterization of $p$-limit diagrams in terms of mapping spaces \cite[Remark~4.3.1.2]{HTT} and the definition of $\infty$-operads. 
    \end{proof}

    Now we will obtain a description of inert and cocartesian arrows of $\Mod^{\Finp}(\CC)^\otimes$ in terms of the model of Proposition~\ref{prop:modules-using-tCMod}.
    \begin{construction}[Bar construction]
        There is a functor
    \[B:(\Delta^{\op})^\triangleright\to \tCMod^\otimes\]
    sending $[n]$ to $(\langle 2\rangle,\Hom_\Delta([n],[1])_+,\{\{r_0\},\{r_1\}\})$ where $r_i$ is the constant arrow at $i$, and the point at $\infty$ to $m_{1,1}=(\langle 1\rangle, \langle 1\rangle,\{1\})$. Concretely this sends $[n]$ to the object $(m_{2,1},a,\dots,a,m_{2,2})$ in the fibre over $\langle n+2\rangle$ of the $\infty$-operad $\tCMod^\otimes_{\langle 2\rangle}$ (and so it encodes the bar construction in $\tCMod^\otimes_{\langle 2\rangle}$).
    \end{construction}

    \begin{lemma}\label{lemma:cocart-arrows-in-mod}
        Let $e\colon \Delta^1\to \Mod^{\Finp}(\CC)^\otimes$ be an arrow, and let $e_0\colon \langle n\rangle\to \langle n'\rangle$ be the image of $e$ in $\Finp$. Write
        \[F_e\colon \tCMod_{\Delta^1}^\otimes\to \CC^\otimes\]
        for the functor corresponding to $e$ via the isomorphism of Proposition~\ref{prop:modules-using-tCMod}.
        \begin{enumerate}
            \item The arrow $e$ is inert if and only if $e_0$ is inert and $F_e$ sends the arrows $a_n\to a_{n'}$ and $m_{i,n}\to m_{e_0i,n'}$ to cocartesian arrows.
            \item Suppose that $\CC^\otimes$ is a symmetric monoidal $\infty$-category compatible with geometric realizations, and that $e_0$ is the unique active arrow from $\langle2\rangle$ to $\langle 1\rangle$. Then $e$ is cocartesian if and only if $F_e$ sends the arrow $a_2\to a_1$ to a cocartesian arrow and the composition
            \[(\Delta^{\op})^\triangleright\xrightarrow{B} \tCMod^\otimes_{\Delta^1}\xrightarrow{F_e} \CC^\otimes\]
            is an operadic colimit diagram.
        \end{enumerate}
    \end{lemma}
    \begin{proof}
        This is immediate from the proofs of \cite[Proposition 3.3.3.10 and Theorem~4.5.2.1]{HA} and the identification of Proposition~\ref{prop:modules-using-tCMod}.
    \end{proof}
    
    \begin{construction}
        There is a square of $\infty$-categories
        \[\begin{tikzcd}
        \Finp \times\Finp \ar[r, "(1{,}\wedge)"]\ar[d,"j_1"] & \Finp\times\Finp\ar[d,"j_2"]\\
        \Finp\times\CMod^\otimes\ar[r,"\phi"] & \tCMod^\otimes
    \end{tikzcd}\]
    where
    \begin{itemize}
            \item The top horizontal arrow sends $(\langle n\rangle,\langle m\rangle)$ to $(\langle n\rangle,\langle n\rangle\wedge\langle m\rangle)$;
            \item The arrow $j_1$ sends $(\langle n\rangle,\langle m\rangle)$ to $(\langle n\rangle, (\langle m\rangle,\varnothing))\in\Finp\times\CMod^\otimes$;
            \item The arrow $j_2$ sends $(\langle n\rangle,\langle m\rangle)$ to $(\langle n\rangle,\langle m\rangle,\{\varnothing\})\in \tCMod^\otimes$;
            \item The arrow $\phi$ sends $(\langle n\rangle,(\langle m\rangle,S))\in\Finp\times\CMod^\otimes$ to $(\langle n\rangle,\langle n\rangle \wedge\langle m\rangle,\{\{i\}\times S\})$.
    \end{itemize}
    Since each of these functors sends inert arrows to inert arrows, it induces for every $X\in (\Cat_\infty)_{/\Finp}$ a natural square
    \[\begin{tikzcd}
        \Fun_{/\Finp}(X,\Mod^{\Finp}(\CC)^\otimes)\simeq \Alg_{\tCMod_X^\otimes}(\CC^\otimes)\ar[r]\ar[d] & \Fun_{/\Finp}(X,\Alg_{\CMod^\otimes}(\CC)^\otimes)\ar[d]\\
        \Fun_{/\Finp}(X,\Finp\times \CAlg(\CC))\simeq \Alg_{X\times\Finp}(\CC^\otimes)\ar[r] & \Fun_{/\Finp}(X,\CAlg(\CC)^\otimes)
    \end{tikzcd}\]
    and therefore a natural square of $\infty$-categories over $\Finp$
    \begin{equation}\label{eqn:diagram-of-mods}\begin{tikzcd}
        \Mod^{\Finp}(\CC)^\otimes\ar[r]\ar[d] & \Alg_{\CMod^\otimes}(\CC)^\otimes\ar[d]\\
        \Finp\times\CAlg(\CC)\ar[r] & \CAlg(\CC)^\otimes
    \end{tikzcd}\end{equation}
    \end{construction}
    Our goal now is to show that the square \eqref{eqn:diagram-of-mods} is cartesian. To do so we will show that the right vertical arrow is a cocartesian fibration in favourable situations.
    \begin{lemma}\label{lemma:families-of-commutative-algebras}
            Let $\CI$ be an $\infty$-category and $\CC^\otimes\to \CI^\amalg$ be an $\CI^\amalg$-monoidal $\infty$-category compatible with geometric realizations. Then the map of $\infty$-operads
            \[p_\CI:\Alg_{\CMod^\otimes/\CI^\amalg}(\CC)^\otimes\to \Alg_{\Finp/\CI^\amalg}(\CC)^\otimes\]
            is a cocartesian fibration.
    \end{lemma}
    \begin{proof}
        Note that by \cite[Proposition~3.2.4.3.(3)]{HA} this is a map of cocartesian fibrations over $\CI^\amalg$. Moreover the fibre over $\{x_j\}_{j\in J}\in \CI^\amalg$ is given by
        \[\prod_{j\in J} \Mod(\CC_{x_j})\to \prod_{j\in J} \CAlg(\CC_{x_j})\]
        and therefore it is a cocartesian fibration by \cite[Theorem~4.5.3.1]{HA}. Therefore by \cite[Proposition~2.4.2.11]{HTT} $p_\CI$ is a locally cocartesian fibration with locally cocartesian arrows those given by the composition of a fibrewise cocartesian arrow and a cocartesian arrow over $\CI^\amalg$. In order to prove it is a cocartesian fibration it suffices to show then that the composition of two locally cocartesian arrow is locally cartesian, that is that fibrewise cocartesian arrows are stable under pushforward along arrows in $\CI^\amalg$. Unwrapping the various cases it suffices to show that for every $x,y\in\CI$ and arrow $f:x\to y$ the squares
        \[\begin{tikzcd}
            \Mod(\CC_x)\times \Mod(\CC_x)\ar[r,"\otimes"]\ar[d] & \Mod(\CC_x)\ar[d]\\
            \CAlg(\CC_x)\times \CAlg(\CC_x)\ar[r,"\otimes"] & \CAlg(\CC_x)
          \end{tikzcd}\textrm{ and }
          \begin{tikzcd}
            \Mod(\CC_x)\ar[r,"f_*"]\ar[d] & \Mod(\CC_y)\ar[d]\\
            \CAlg(\CC_x)\ar[r,"f_*"] & \CAlg(\CC_y)
          \end{tikzcd}\]
          are maps of cocartesian fibrations. That is that for every two maps of commutative algebras $A\to A'$, $B\to B'$, $A$-module $M$, and $B$-module $N$, the canonical maps
          \[(M\otimes N)\otimes_{A\otimes B} (A'\otimes B')\simeq (M\otimes_A A')\otimes (N\otimes_B B')\textrm{ and }f_*(M\otimes_A B)\simeq f_*M\otimes_{f_*A}f_*B\,.\]
          are equivalences. This is easily seen to be true since $f_*$ is symmetric monoidal and commutes with geometric realization, and the tensor product commutes with geometric realization in each variable.
    \end{proof}
    Finally we arrive at the main result of this section.
    
    \begin{theorem}\label{theorem:modules-are-pullbacks}
        The square~\eqref{eqn:diagram-of-mods} is cartesian for every $\infty$-operad $\CC^\otimes$.
    \end{theorem}
    \begin{proof}
        Let us do first the case where $\CC^\otimes$ is a symmetric monoidal $\infty$-category compatible with geometric realizations. Then both vertical arrows are cocartesian fibrations by \cite[Theorem~4.5.3.1]{HA} and Lemma~\ref{lemma:families-of-commutative-algebras}. Moreover the description of cocartesian arrows in Lemma~\ref{lemma:cocart-arrows-in-mod} and \cite[Proposition~3.2.4.3.(4)]{HA} shows that
        \[\Mod^{\Finp}(\CC)^\otimes\to (\Finp\times\CAlg(\CC))\times_{\CAlg(\CC)^\otimes} \Alg_{\CMod^\otimes}(\CC)^\otimes\]
        is a map of cocartesian fibrations over $\Finp$. So it suffices to show that it induces an equivalence on fibres. Since it is a map of generalized operads, it suffices to show it induces an equivalence on the fibres over $\langle 0\rangle$ and $\langle 1\rangle$. But this is immediate by Proposition~\ref{prop:modules-using-tCMod}.
        
        Now let us show the result for small $\infty$-operads $\CC$. Indeed, it is clear by inspection that if the square~\eqref{eqn:diagram-of-mods} is cartesian for an $\infty$-operad, then it is cartesian for any full suboperad. But every small $\infty$-operad embeds as a full suboperad of a symmetric monoidal $\infty$-category compatible with small colimits. Indeed this is just the composition $\CC^\otimes\to \operatorname{Env}\CC^\otimes\to \CP(\operatorname{Env}\CC)^\otimes$ where $\operatorname{Env}\CC^\otimes$ is the symmetric monoidal envelope of $\CC^\otimes$, and the second arrow is the Yoneda embedding.
    
        Finally, since every $\infty$-operad is a sufficiently filtered union of small suboperads, the thesis is true for any $\infty$-operad.
    \end{proof}

\bibliographystyle{plain}
\bibliography{reference}

\end{document}